\documentclass[11pt,oneside,english]{article}
\usepackage[T1]{fontenc}
\usepackage[latin9]{inputenc}
\usepackage{babel}
\usepackage{mathtools}
\usepackage{amstext}
\usepackage{amsthm}
\usepackage{amssymb}
\usepackage{stmaryrd}
\usepackage[unicode=true,pdfusetitle,
 bookmarks=true,bookmarksnumbered=false,bookmarksopen=false,
 breaklinks=false,pdfborder={0 0 1},backref=false,colorlinks=false]
 {hyperref}

 \usepackage{changes}

\makeatletter
\numberwithin{equation}{section}
\numberwithin{figure}{section}
\theoremstyle{plain}
\newtheorem{thm}{\protect\theoremname}[section]
\theoremstyle{definition}
\newtheorem{defn}[thm]{\protect\definitionname}
\theoremstyle{remark}
\newtheorem{rem}[thm]{\protect\remarkname}
\theoremstyle{plain}
\newtheorem{assumption}[thm]{\protect\assumptionname}
\theoremstyle{plain}
\newtheorem{prop}[thm]{\protect\propositionname}
\theoremstyle{plain}
\newtheorem{lem}[thm]{\protect\lemmaname}
\theoremstyle{plain}
\newtheorem{cor}[thm]{\protect\corollaryname}

\newtheorem*{theorem}{Theorem}

\usepackage[a4paper]{geometry}
\usepackage{amsfonts}

\usepackage{times}

\allowdisplaybreaks[3]

\makeatother

\providecommand{\assumptionname}{Assumption}
\providecommand{\corollaryname}{Corollary}
\providecommand{\definitionname}{Definition}
\providecommand{\lemmaname}{Lemma}
\providecommand{\propositionname}{Proposition}
\providecommand{\remarkname}{Remark}
\providecommand{\theoremname}{Theorem}

\begin{document}
\title{Well-posedness of the obstacle problem for generalized Dean--Kawasaki equation\thanks{
The first author is also supported by the National Science Foundation of Shanghai (Grant No.25ZR1402408). The second author is supported by Open Foundation of the State Key Laboratory of Mathematical Sciences (Grant No. 09), Beijing Institute of Technology Research Fund Program for Young Scholars and MIIT Key Laboratory of Mathematical Theory and Computation in Information Security.} }
\author{Ruoyang Liu\thanks{Department of Mathematics, Shanghai Normal University, Shanghai, China (Email: {\tt ryliu@shnu.edu.cn}).}
\and
Rangrang Zhang \thanks{ Corresponding Author, School of Mathematics and Statistics,
Beijing Institute of Technology, Beijing 100081, China (Email:{\tt rrzhang@amss.ac.cn}).}
}
\date{}
\maketitle

\vspace{-1ex}
\begin{abstract}

We investigate the obstacle problem for generalized Dean--Kawasaki equations driven by correlated conservative noise, establishing the existence, uniqueness, and $L^1$-stability of stochastic kinetic solutions.
Our core strategy combines a kinetic characterization of the Skorokhod condition with a precise description of the reflection measure term associated with the obstacle, in which the barrier substitutes the solution.
This formulation makes the reflection mechanism explicit at the kinetic level and yields a stable framework adapted to $L^1$ doubling of variables method.
Consequently, under a merely continuous obstacle and the same structural assumptions as in the obstacle-free setting, we obtain well-posedness over the full porous-medium regime, covering degenerate diffusion and the critical square-root noise coefficient.
This extends the existing theory of obstacle problems for stochastic partial differential equations to a class of degenerate equations with singular diffusion coefficients.

\medskip

\medskip


\end{abstract}

\section{Introduction}
In this paper, we study the well-posedness and stability of the obstacle problem for generalized Dean--Kawasaki equations with correlated conservative noise.
More precisely, given a nonnegative upper obstacle $\psi=\psi(x,t)$, we seek a nonnegative random field $\{u(x,t): t\geq0,\ x\in\mathbb{T}^d\}$ which is constrained by the unilateral condition $u\leq\psi$ and obeys the generalized Dean--Kawasaki equations
\begin{equation}
\partial_{t}u=\Delta\Phi(u)-\nabla\cdot(\sigma(u)\circ\dot{\xi}^{F}+g(u))\label{eq:main}
\end{equation}
in the non-contact set $\{(x,t):u(x,t)<\psi(x,t)\}$. Here $\dot{\xi}^{F}$ denotes a spatially correlated, temporally white vector-valued noise, and the stochastic integral is interpreted in the Stratonovich sense.

The generalized Dean--Kawasaki equation is a conservative stochastic partial differential equations (SPDEs) postulated by fluctuating hydrodynamics, which can be viewed as mesoscopic model for non-equilibrium fluctuations. A basic example is the classical Dean--Kawasaki equation with correlated noise,
\begin{align*}
\partial_{t}u=\Delta u-\nabla\cdot(\sqrt{u}\circ\dot{\xi}^{F})
\end{align*}
which can capture the fluctuation of Brownian particles (see \cite{dean1996langevin,kawasaki1998microscopic}). Additionally,
 conservative SPDEs of the form
\begin{align*}
\partial_{t}u=\Delta\Phi(u)-\nabla\cdot \Phi(u)-\nabla\cdot(\Phi^{\frac{1}{2}}(u)\circ\dot{\xi}^{F})
\end{align*}
can serve as ``mesoscopic'' models of the asymmetric zero-range process (see \cite{dirr2016entropic,fehrman2024well,fehrman2023non}).
In the present paper, the coefficients $\Phi,\sigma,$ and $g$ are allowed to be sufficiently general.
In particular, our setting includes the porous-medium family
\[\Phi(r)=r^m,\qquad m\in(0,\infty),\]
 which encompasses the slow-diffusion regime $m>1$ and the fast-diffusion regime $m<1$.
Consequently, the equation \eqref{eq:main} considered herein combine two singular features: degenerate diffusion and conservative noise with a low-regularity coefficient.

Obstacle problems for deterministic PDEs have been extensively studied within the framework of variational inequalities; see, for instance, \cite{lions1967variational,brezis1972problemes,bensoussan1982applications,alt1983quasilinear,caffarelli1998obstacle,bogelein2011degenerate,bogelein2015obstacle,scheven2015existence,caffarelli2017obstacle,athanasopoulos2019parabolic}.
In the stochastic case, Haussmann and Pardoux \cite{haussmann1989stochastic} first  formulated the obstacle problem through the framework of the Skorokhod problem.
A key feature of this approach is the introduction of a nonnegative Radon measure $\nu$ in \eqref{eq:main}, which means
\begin{equation}
\partial_{t}u=\Delta\Phi(u)-\nabla\cdot(\sigma(u)\circ\dot{\xi}^{F}+g(u))-\nu,\label{eq:SVI}
\end{equation}
together with the Skorokhod condition
\begin{equation}
\langle u-\psi,\nu\rangle=0            \quad(\text{a.s.})
\label{eq:skorohod}
\end{equation}
The measure $\nu$ plays the role of a minimal downward forcing which keeps the solution below the upper barrier $\psi$, while \eqref{eq:skorohod} ensures that this forcing acts only on the contact set $\{u=\psi\}$. Then, the solution to the obstacle problem becomes a dual $(u,\nu)$.
This viewpoint has become a standard framework for studying obstacle problems for SPDEs \cite{rockner2013stochastic,denis2014obstacle,yang2019obstacle,brzezniak2023reflection,liu2024obstacle}, and it has been successfully applied to the heat equation driven by space-time white noise \cite{nualart1992white,donati1993white,xu2009white} as well as to quasilinear SPDEs \cite{matoussi2010obstacle,yang2019obstacle,brzezniak2023reflection}.
However, for degenerate stochastic equations, results within the Skorokhod framework are relatively sparse, and strong assumptions on the coefficients and the barrier are typically required.
For instance, Yang and Zhang \cite{yang2019obstacle} studied quasilinear degenerate equations with Lipschitz nonlinearities;
R\"ockner, Wang, and Zhang \cite{rockner2013stochastic} considered the stochastic porous medium equations with the constant lower obstacle $\psi\equiv0$;
and the first author and Tang \cite{liu2024obstacle} obtained the well-posedness under a special assumption on the barrier.

For the degenerate conservative equation \eqref{eq:main} with singular coefficients considered in this paper, however, it seems that \eqref{eq:skorohod} does not provide a suitable starting point.
The main difficulty is that the a priori estimates available in our setting (Lemma \ref{lem:priori-estimates}) only yield $\nu$ as a Radon measure, while the solution $u$ lacks the regularity (such as continuity or quasi-continuity\cite{denis2014obstacle}) required to give a strong interpretation of $\langle u-\psi,\nu\rangle$ and to pass to the limit in penalized approximations (see Section \ref{sec:exist} for more details).
To overcome this difficulty, one possible strategy is to introduce an appropriate notion of solution that incorporates the Skorokhod condition only implicitly. Recently, Du and the first author \cite{du2024well} studied the obstacle problem for stochastic nonlinear diffusion equations by using a new entropy formulation, in which the Skorokhod condition is bypassed through a suitable modification of the entropy inequality.
Specifically, for each convex function $\eta\in C^2(\mathbb{R})$ with compact support of $\eta^{\prime\prime}$, they tested the equation against $\eta^\prime(u)$, and replaced $\eta^\prime(u)\nu$ with $\eta^\prime(\psi)\nu$.
This substitution serves two purposes: it circumvents the definitional difficulty arising from the unmatched regularity of $u$ and $\nu$, while simultaneously incorporating the Skorokhod condition directly into the entropy inequality.
However, the entropy formulation is essentially restricted to cases where $\sigma$ in \eqref{eq:main} is Lipschitz continuous (see e.g. \cite{dareiotis2020nonlinear,fehrman2024well}), the barrier is twice continuously differentiable and cannot handle singular noise coefficient, such as those that are only locally $1/2$-H\"{o}lder continuous, as considered in the present paper.

\subsection{Main results}

We aim to lift the idea of \cite{du2024well} to the kinetic level to resolve the critical square-root diffusion coefficient in the generalized Dean--Kawasaki equation \eqref{eq:main} and reduce the regularity assumption on the barrier.
Our starting point is the kinetic formulation developed in \cite{fehrman2024well} for \eqref{eq:main} without obstacles.
By introducing a non-negative obstacle defect measure to capture
the effect of the barrier constraint, we provide a finer kinetic characterization of the Skorokhod condition.
Importantly, this defect measure can be reformulated as a component of the kinetic measure, thus the obstacle problem is incorporated into a genuinely kinetic framework that is robust under penalized approximations and well adapted to the $L^1$ doubling-of-variables method.  We provide a detailed explanation in Section
\ref{sec-kinetic}.

Under the assumption that the barrier $\psi$ is merely continuous in both time and space, and without imposing any additional conditions on the coefficients $\Phi, \sigma, g$ beyond those employed in \cite{fehrman2024well}, we establish the well-posedness of the obstacle problem for the equation \eqref{eq:main}. It reads as follows.
\begin{theorem}
(Theorems \ref{thm:uniqueness}, \ref{thm:existence}, Corollary \ref{cor-1}) Let Assumption \ref{assu:assum for F} and Assumption \ref{assu:assu for unique} hold for $\dot{\xi}^{F}, \psi$, $u_{\mathrm{init}}, \Phi, \sigma, g$.  Then the obstacle problem for \eqref{eq:main} has at most one stochastic kinetic solution with the initial value $u_{\mathrm{init}}$ in the sense of Definition \ref{def:def of stochastic kinetic solution}. Moreover, if $(u_1,\nu_1)$ and $(u_2,\nu_2)$ are two stochastic kinetic solutions with initial data $u_{1,\mathrm{init}}$ and $u_{2,\mathrm{init}}$, then
almost surely for every $t\in(0,T)$,
\begin{align*}
\|u_1(t)-u_{2}(t)\|_{L^1(\mathbb{T}^d)}
\leq\|u_{1,\mathrm{init}}-u_{2,\mathrm{init}}\|_{L^1(\mathbb{T}^d)}.
\end{align*}
If further Assumption \ref{assu:assu for existence} is in force for some $p\in [2, \infty)$, there exists a unique stochastic kinetic solution $(u,\nu)$ in the sense of Definition \ref{def:def of stochastic kinetic solution}.

The same results also hold for the lower obstacle problem for \eqref{eq:main} (the constraint is changed to $u\geq\psi$), under the additional assumptions $|\sigma(r)|\leq C(1+|r|)$ and $p=2$. This includes the classical Dean--Kawasaki equation
\begin{align*}
\mathrm{d}u & =\Delta\Phi(u)\mathrm{d}t+\nabla\cdot\sqrt{\Phi(u)}\circ\mathrm{d}\xi^{F},
\end{align*}
provided that $|\Phi(r)|\leq C(1+|r|^{2})$ for all $r\in\mathbb{R}$, and thus in particular covers the case $\Phi(r)=r$.

\end{theorem}

We present two motivating examples to illustrate that \eqref{eq:SVI} provides a natural mathematical framework for stochastic transport under local upper-capacity constraints. As such constraints arise naturally in both interacting particle systems and reactive transport, our results are applicable across these contexts.
The first example is drawn from weakly asymmetric zero-range-type processes with a time-dependent capacity field.
In the absence of such capacity constraints, the mesoscopic evolution is expected to combine nonlinear diffusion with conservative stochastic flux \cite{dirr2016entropic,fehrman2023non,fehrman2024well}.
If one further assumes that any excess mass above the local capacity is removed, the macroscopic density is constrained by $u\le \psi$, while the measure $\nu$ records the cumulative removed mass.
Another example stems from reactive transport in porous media, where $u$ denotes the concentration of a dissolved species and $\psi$ represents a local saturation or solubility threshold.
When precipitation or immobilization is much faster than transport, the dissolved concentration is effectively kept below $\psi$, and the measure $\nu$ describes the mass transferred from the mobile phase to an immobile phase \cite{appelo2004geochemistry,steefel2005reactive,krautle2011semismooth,banshoya2023simulation}.

\subsection{Kinetic characterization of the Skorokhod condition}\label{sec-kinetic}
In this subsection, we explain the main idea behind our kinetic formulation of the obstacle problem.

\textbf{(I)}\quad We first review the verification of the formal Skorokhod condition (may be not well-defined):
\[
\int_0^T\int_{\mathbb{T}^d}(u(x,t)-\psi(x,t))\nu(\mathrm{d}x,\mathrm{d}t)=0.
\]
Taking the upper obstacle problem as an example. Since $u\leq\psi$ and measure $\nu$ is nonnegative, we have
\[
\int_0^T\int_{\mathbb{T}^d}(u(x,t)-\psi(x,t))\nu(\mathrm{d}x,\mathrm{d}t)\leq0.
\]
The difficult part is to prove the reverse inequality, which is obtained through penalization approximation.
In the classical penalty method, one considers the penalized equation
\begin{equation}
\partial_{t}u_\varepsilon=\Delta\Phi(u_\varepsilon)-\nabla\cdot(\sigma(u_\varepsilon)\circ\dot{\xi}^{F}+g(u_\varepsilon))-\frac1\varepsilon (u_\varepsilon-\psi)^+.
\end{equation}
Defining the penalty term $\nu_\varepsilon:= (u_\varepsilon-\psi)^+/\varepsilon$.
If one can establish a strong convergence of $u_\varepsilon$ to a $u$ in $C(Q_T)$ and a weak convergence of the penalized term $\frac{1}{\varepsilon}(u_\varepsilon-\psi)^+$ to a $\nu$ in the space of nonnegative Borel measures $\mathcal{M}(Q_T)$, then it is straightforward to deduce that the ``Skorokhod term'' $\langle u_\varepsilon-\psi,\nu_\varepsilon\rangle$ satisfies
\begin{equation}\label{eq:reverse SKorokhod}
0\leq\int_0^T\int_{\mathbb T^d}
(u_\varepsilon-\psi)\nu_\varepsilon\mathrm dx\mathrm dt\overset{{\varepsilon\to0}}{\longrightarrow} \int_0^T\int_{\mathbb T^d}(u-\psi)\nu(\mathrm dx,\mathrm dt).
\end{equation}
This, together with the first inequality, yields the Skorokhod condition.

However, when applying this method to more general equations, it encounters limitations in establishing the strong convergence of $u_\varepsilon$ required to prove \eqref{eq:reverse SKorokhod}, especially for the degenerate equation with singular coefficients considered in this paper.

\textbf{(II)}\quad We now explain how the obstacle condition is represented in the kinetic formulation.
The key idea is not to impose the formal Skorokhod condition directly, but rather to identify the ``kinetic Skorokhod term'': an object generated by the penalization procedure that serves as the kinetic counterpart of the classical condition.

Compared with the weak formulation, the renormalized kinetic formulation tests the equation not only with respect to the spatial variable, but also looks at different level sets of the solution $u$.
In other words, it studies the sets $\{(x,t): u(x,t)>\xi\}$ for all levels $\xi$.
This point of view is useful for singular equations, as it avoids working directly on $u$ when $u$ is equal to some singular values. Consequently, it is well suited for obstacle problems, since an obstacle condition is exactly a restriction on the possible levels of $u$.

More precisely, the kinetic function of $u_\varepsilon$ is defined as
\[
\chi_\varepsilon(x,t,\xi)=\mathbf 1_{\{0<\xi<u_\varepsilon(x,t)\}}.
\]
Here, $\xi$ is the level parameter.
The function $\chi$ records whether the value of $u_\varepsilon$ is above the level $\xi$.
Note that for a test function $\phi\in C_c^\infty(\mathbb T^d\times(0,\infty))$, we have
\[
\phi(x,u_\varepsilon)=\int_0^\infty\phi(x,\xi)\chi_\varepsilon(x,t,\xi)\mathrm{d}\xi.
\]
Then, we can rewrite the penalized equation into a kinetic form with the test function $\phi(x,\xi)$.
See \eqref{eq:kinetic for penalty} or \cite{fehrman2024well} for further details.
Here, we focus only on the new contribution generated by the penalization term $\nu_\varepsilon$.
The penalty term would formally give
\[
\int_0^T\int_{\mathbb T^d}
\phi(x,u_\varepsilon)\nu_\varepsilon\,\mathrm dx\mathrm dt .
\]
Our idea is to replace this term by
\[
\int_0^T\int_{\mathbb T^d}
\phi(x,\psi)\nu_\varepsilon\,\mathrm dx\mathrm dt,
\]
This replacement is crucial for the limiting procedure: the resulting term is well suited for passing to the limit $\varepsilon\to0$, even when $\nu_\varepsilon$ converges only weakly in the space of measures, owing to the continuity of $\phi(x,\psi)$.
Furthermore, the resulting difference
\begin{equation}\label{eq:kinetic-obstacle-difference}
\int_0^T\int_{\mathbb T^d}\big(\phi(x,u_\varepsilon)-\phi(x,\psi)\big)\nu_\varepsilon\mathrm dx\mathrm dt ,
\end{equation}
 can be interpreted as a penalized form of the kinetic Skorokhod term.
Indeed, when $\phi(x,r)=r$, this term is exactly  the ``Skorokhod term'' appearing in the classical verification of the Skorokhod condition.
For a general test function $\phi$, however, \eqref{eq:kinetic-obstacle-difference} provides a whole family of such quantities, with one corresponding to each kinetic level.
This is the reason why it contains more precise information than the scalar condition obtained from the linear test function.

Using the fundamental theorem of calculus in the kinetic variable, we have
\[
\begin{aligned}
&\int_0^T\int_{\mathbb T^d}
\big(\phi(x,u_\varepsilon)-\phi(x,\psi)\big)\nu_\varepsilon
\,\mathrm dx\mathrm dt     \\
&=
\int_0^T\int_{\mathbb T^d}\int_0^1
\partial_\xi\phi\big(x,\psi+\theta(u_\varepsilon-\psi)\big)
\frac1\varepsilon\big[(u_\varepsilon-\psi)^+\big]^2
\,\mathrm d\theta\,\mathrm dx\mathrm dt .
\end{aligned}
\]
If we define the measure
\[
\lambda_\varepsilon(\mathrm dx,\mathrm d\xi,\mathrm dt):=\int_0^1\frac1\varepsilon\big[(u_\varepsilon-\psi)^+\big]^2\delta_0\big(\xi-\psi-\theta(u_\varepsilon-\psi)\big)\mathrm d\theta\,\mathrm dx\mathrm dt\mathrm d\xi.
\]
Then, we obtain
\begin{equation}\label{eq:lambda-epsilon-action}
\int_0^T\int_{\mathbb T^d}
\big(\phi(x,u_\varepsilon)-\phi(x,\psi)\big)\nu_\varepsilon
\,\mathrm dx\mathrm dt
=
\int_0^T\int_{\mathbb T^d}\int_0^\infty
\partial_\xi\phi(x,\xi)\,
\lambda_\varepsilon(\mathrm dx,\mathrm d\xi,\mathrm dt).
\end{equation}
Thus, $\lambda_\varepsilon$ is not introduced artificially. Instead, it is precisely the measure which records the level-by-level cost of replacing $\phi(x,u_\varepsilon)$ by $\phi(x,\psi)$ in the reflection term.
Moreover, the support of $\lambda_\varepsilon$ lies on the segment in the kinetic variable joining $\psi$ and $u_\varepsilon$:
\[
\xi=\psi+\theta(u_\varepsilon-\psi),\qquad 0\leq\theta\leq1.
\]
As a result, $\lambda_\varepsilon$ quantifies both the magnitude by which the penalized solution exceeds the obstacle and the kinetic levels at which such excess occurs.
The standard a priori estimate (see \eqref{eq:priori} for our equation, and \cite[Lemma 4]{denis2014obstacle} or \cite[Theorem 3.3]{liu2024obstacle} for other cases)
\[
\mathbb E\int_0^T\int_{\mathbb T^d}
\frac1\varepsilon\big[(u_\varepsilon-\psi)^+\big]^2
\,\mathrm dx\mathrm dt
\leq C
\]
implies that $\{\lambda_\varepsilon\}_{\varepsilon>0}$ is uniformly bounded as a family of nonnegative measures on
$\mathbb T^d\times(0,\infty)\times[0,T]$.
Hence, up to a subsequence, $\lambda_\varepsilon \rightharpoonup \lambda$ weakly as measures.
The limit $\lambda$ is called the ``obstacle defect measure''.

Importantly, our analysis does not rely on any strong convergence property of $u^{\varepsilon}$, ensuring the broad applicability of the result. On the other hand, In the  classical Skorokhod formulation in \cite{donati1993white,denis2014obstacle,yang2019obstacle,liu2024obstacle}, we have $\lambda\equiv0$ based on its definition.

Building upon the above result, in the limiting kinetic formulation, the penalty term is therefore represented by two pieces:
\[
-\delta_0(\xi-\psi)\nu\qquad\text{and}\qquad\partial_\xi\lambda .
\]
The first piece corresponds to  the reflected force evaluated at the obstacle level and
the second piece is the defect measure arising from the replacement of $u$ with $\psi$ in the kinetic test.
Combined together, they constitute the kinetic counterpart of the Skorokhod mechanism.

Formally, compared with the formal kinetic term $-\delta_0(\xi-u)\nu$, the obstacle part of the limiting kinetic equation has the form $\partial_\xi \lambda-\delta_0(\xi-\psi)\nu$, and one may read the identity as
\[
-\delta_0(\xi-u)\nu
=
-\delta_0(\xi-\psi)\nu+\partial_\xi\lambda .
\]
This identity is not used as a pointwise statement.
Instead, it is encoded weakly through \eqref{eq:lambda-epsilon-action}.
This is the kinetic formulation of the Skorokhod condition.
In contrast to the entropy solution framework employed by \cite{du2024well}, our method yields a refined characterization of the Skorokhod term, which is one of the main novelties of the present paper.

Furthermore, we can verify that the measure $\lambda$ is a kinetic measure.
If $m$ denotes the usual parabolic defect measure coming from the degenerate diffusion, we set
\begin{align*}
q:=m+\lambda .
\end{align*}
The limiting kinetic equation is then written with the total kinetic measure $q$.
This is crucial: the obstacle defect measure has the same sign and the same structural role as the standard kinetic defect measure.
Hence, it fits naturally into the kinetic framework and can be handled by the same doubling-of-variables method employed for uniqueness.

\textbf{(III)}\quad Next, we explain why our kinetic formulation is consistent with the classical Skorokhod condition.
Suppose that the solution is regular enough (e.g. continuous or quasi-continuous in \cite{denis2014obstacle}) so that the product $(u-\psi)\nu$ is well-defined.
Then one may take the test function $\alpha(t)\phi(x,r)=\alpha(t)\varphi(x)r$ by approximation, where $\alpha\in C_c^\infty([0,T))$ and $\varphi\in C_c^\infty(\mathbb{T}^d)$ are nonnegative.
Comparing the kinetic equation with directly using It\^o's formula to $u\to\int_{\mathbb{T}^d}\alpha\varphi u^2\mathrm{d}x$, we have,
\[
\int_0^T\int_{\mathbb{T}^d}
\alpha(t)\varphi(x)\psi\,\nu(\mathrm{d}x,\mathrm{d}t)
+
\int_0^T\int_{\mathbb{T}^d}\int_0^\infty
\alpha(t)\varphi(x)\lambda(\mathrm{d}x,\mathrm{d}\xi,\mathrm{d}t)=\int_0^T\int_{\mathbb{T}^d}\alpha(t)
\varphi(x)u\,\nu(\mathrm{d}x,\mathrm{d}t).
\]
Equivalently,
\[
\int_0^T\int_{\mathbb{T}^d}
\alpha(t)\varphi(x)(u-\psi)\,\nu(\mathrm{d}x,\mathrm{d}t)
=
\int_0^T\int_{\mathbb{T}^d}\int_0^\infty
\alpha(t)\varphi(x)\lambda(\mathrm{d}x,\mathrm{d}\xi,\mathrm{d}t).
\]
Since $\lambda$ is nonnegative, based on the arbitrary of nonnegative functions $\alpha$ and $\varphi$, this identity prevents $\nu$ from charging the set where $u<\psi$.
Thus $\nu$ is supported on the contact set $\{u=\psi\}$, and the usual Skorokhod condition is recovered.

Therefore, our kinetic formulation substantially generalizes the classical obstacle framework.
When the function $u$ is regular, it recovers the standard Skorokhod condition.
When the function $u$ is not regular, it still gives a meaningful condition through the obstacle defect measure $\lambda$. This key feature makes the present formulation applicable to
generalized Dean--Kawasaki equation with degenerate diffusion and singular conservative noise.

\subsection{Comments on the results and main contributions}
To our knowledge, the kinetic solution framework has not yet been applied to the obstacle problem for SPDEs in existing literature. The present work makes the first attempt to address this research gap. This paper is closely related to the independent works \cite{du2024well} and \cite{fehrman2024well}. We provide a detailed comparison of our framework with these two key references.
\begin{itemize}
  \item
Compared with the existing results in \cite{du2024well}, which studied the obstacle problem for stochastic porous media equations using entropy formulation, our work offers three main significant generalizations.
First, our framework applies to degenerate diffusion operators and in particular to the whole porous media regime $\Phi(r)=r^m$ with $m\in(0,\infty)$, thereby covering the fast diffusion regime $m<1$ which is not explored in \cite{du2024well}.
Second, our work imposes a minimal regularity assumption on the noise coefficient $\sigma$: local 1/2-H\"older continuity is sufficient whereas \cite{du2024well} requires $C^4$ continuity.
Consequently, our result includes the critical square-root case $\sigma(r)=\sqrt r$.
In particular, our theory is applicable to  genuine generalized Dean--Kawasaki equations with correlated conservative noise, while previous obstacle results require much smoother stochastic flux coefficients.
Third, we substantially relax the assumptions on the obstacle. In contrast to the entropy theory employed in \cite{du2024well}, which requires spatial H\"older regularity for existence and even $C^2$-regularity for uniqueness, we only assume that the barrier is continuous in time and space.

These improvements are accompanied by substantial new challenges. In the uniqueness part, since the standard doubling variables argument for kinetic solutions does not double the time variable, the terms involving the Radon measure $\nu$ cannot be handled by the same time-regularization mechanism available in the entropy formulation \cite{du2024well}. The reflection term becomes sensitive to temporal traces of the solution, thereby we have to work with suitable left- and right-continuous representatives of kinetic solutions in order to control the contribution of the reflection terms in proving the $L^1$-contraction principle.
The existence proof is also considerably more delicate than that in the entropy formulation. Indeed, the penalized approximations simultaneously yield the solution sequence, the penalization terms, the parabolic defect measures, and the obstacle defect measures. Owing to the low regularity of $\sigma$, the available a priori estimates are barely sufficient to identify all of these objects in the limit. A central part of this paper is therefore devoted to obtaining simultaneous compactness and convergence of the kinetic measures, the penalization measures, and the solutions themselves, as well as to proving that the limiting objects satisfy the desired obstacle kinetic formulation.

\item In contrast to \cite{fehrman2024well}, the obstacle formulation introduces two distinct random measures: the obstacle defect measure $\lambda$ and the Radon measure $\nu$. The former is non-negative and can be  incorporated into the kinetic measure. However, the presence of the Radon measure $\nu$ cause discontinuities in the function $u$ at a subset of $[0,T]$, which prevents the kinetic solution from satisfying the initial condition in the classical sense. This differs from \cite{fehrman2024well}, where the kinetic solution is shown to be continuous in time. To remedy this, we introduce a new condition (iii) in Definition \ref{def:def of stochastic kinetic solution}
\[
\lim_{\tau\to0}\frac{1}{\tau}\int_0^{\tau}\int_{\mathbb{T}^d}|u(x,t)-u_{\mathrm{init}}(x)|\mathrm{d}x\mathrm{d}t=0,
\]
to ensure that the kinetic solution $u$ satisfies the initial condition in the weak integral sense. To attain such a weak initial condition, we establish a novel property of the kinetic measure (see \eqref{eq:limit for N q}).

Regarding the proof of uniqueness, it is not a direct repetition of the argument in \cite{fehrman2024well}.
The obstacle constraint introduces an additional measure $\nu$ in the kinetic formulation, which gives rise to new terms that not appear in the non-obstacle case.
A substantial part of the proof is devoted to showing that these additional terms are compatible with the doubling-of-variables argument.
Moreover, the presence of measure $\nu$ may cause the kinetic function $\chi$ to be discontinuous in time.
This makes it necessary to work with the right- and left-continuous versions $\chi^{+}$ and $\chi^{-}$, especially in the integrals with respect to the measure $\nu$ and kinetic measure $q$.
These difficulties are specific to the obstacle problem and constitute the main ingredients of our uniqueness proof.
Furthermore, to construct a stochastic kinetic solution of the obstacle problem for \eqref{eq:main}, we introduce an approximating equation involving the penalization term $(u_{\varepsilon}-\psi)^{+}/{\varepsilon}$.
The main difficulty is that this penalty term  has a Lipschitz constant of order $1/\varepsilon$, which diverges as $\varepsilon \to 0$.
Consequently, the convergence result in \cite{fehrman2024well} is not directly applicable.
We instead establish the monotonicity of $u_{\varepsilon}$ with respect to $\varepsilon$ (see Lemma \ref{lem:(Comparison-theorem)}).
Combined with the non-negativity of the solution, this monotonicity guarantees the convergence of $u_{\varepsilon}$ as $\varepsilon \to 0$.

 \end{itemize}

\subsection{Overview of the literature}

Macroscopic fluctuation theory (MFT) and fluctuating hydrodynamics (FHD) provide a general framework connecting microscopic particle systems with continuous macroscopic fluid dynamics (see \cite{BDSG15} and \cite{S12}).
In the FHD picture, non-equilibrium density fluctuations are described at the mesoscopic level by conservative, and often singular, SPDEs.
A central example is the Dean--Kawasaki equation, introduced by Dean and Kawasaki \cite{dean1996langevin,kawasaki1998microscopic} as a fluctuating continuum model for particle densities.
This equation plays a basic role in fluctuating hydrodynamics and stochastic density functional theory, and it has been used in the study of colloidal suspensions, supercooled liquids, polymeric fluids, and other complex systems.
From the mathematical point of view, Dean--Kawasaki type equations are highly singular, since they combine degenerate diffusion with conservative multiplicative noise, often of square-root type.
For this reason, several earlier works established existence only after modifying the coefficients or drift terms; see \cite{VRS09,KLVR19,AvR10}.
A recent breakthrough was obtained by Fehrman and Gess \cite{fehrman2024well}, who proved well-posedness for function-valued solutions to generalized Dean--Kawasaki equations with correlated noise and locally $1/2$-H\"older coefficients.
See also \cite{fehrman2023non} for the related large-deviation analysis arising in fluctuating hydrodynamics.
The present paper serves as an obstacle-problem counterpart to this line of research.


Within the framework of kinetic solutions, the investigation into the well-posedness theory for stochastic conservative laws has garnered considerable research attention.
Debussche and Vovelle \cite{DV10} analyzed the Cauchy problem in any spatial dimensions and obtained the existence and uniqueness of the kinetic solutions. Subsequently, Debussche, Hofmanov\'a and Vovelle \cite{DHV16}, as well as Gess and Hofmanov\'a \cite{GH18}, extended the notion of kinetic solution to degenerate parabolic stochastic partial differential equations.
More recently, leveraging the kinetic formulation approach, Fehrman and Gess \cite{fehrman2024well} explored conservative stochastic PDE with correlated noise, covering the generalized Dean--Kawasaki equation and the full range of porous medium with fast diffusion exponents. This line of research was further generalized to the Dean--Kawasaki equation with singular nonlocal interaction kernel by Wang, Wu and Zhang \cite{WWZ26}, while Popat \cite{P25} established the well-posedness of kinetic solutions to the generalized Dean--Kawasaki equation on bounded domains. By virtue of the kinetic formulation, we successfully extend the study of Du and Liu \cite{du2024well} in three key aspects: the framework accommodates generalized Dean--Kawasaki conservative noises, covers fully degenerate diffusions (including fast diffusion regimes), and only requires a continuous obstacle condition.

\subsection{Organization of the paper}
This paper is organized as follows. In Section \ref{sec:Pre}, we introduce the assumptions on the noise and obstacle, and further derive the kinetic formulation for the obstacle problem. In Section \ref{sec:unique}, we establish the $L^1-$stability and uniqueness via a doubling variables technique adapted to the obstacle problem and kinetic solutions. In Section \ref{sec:exist}, the existence result is obtained by means of penalization, approximation, and subsequent limit passing. Two auxiliary lemmas are given in the Appendix \ref{sec:Auxiliary}.

\section{Preliminaries}\label{sec:Pre}
Throughout the paper, $\mathbb{T}^d$ denotes a $d-$dimensional torus with volume 1.
Let $\|\cdot\|_{L^p(\mathbb{T}^d)}$ denote the norm of Lebesgue space $L^p(\mathbb{T}^d)$ (or $L^p(\mathbb{T}^d;\mathbb{R}^d)$) for integer $p\in [1,\infty]$. The inner product in $L^2(\mathbb{T}^d)$ will be denoted by
$(\cdot,\cdot)$.
Let $C^{\infty}(\mathbb{T}^d\times (0,\infty))$ denote the space of infinitely differentiable functions on $\mathbb{T}^d\times (0,\infty)$. $C^{\infty}_c(\mathbb{T}^d\times (0,\infty))$ contains all infinitely continuous differentiable functions with compact supports on $\mathbb{T}^d\times (0,\infty)$.
For a non-negative integer $k$ and $p\in [1,\infty]$, denote by $W^{k,p}(\mathbb{T}^d)$ the usual Sobolev space on $\mathbb{T}^d$. Let
$H^a(\mathbb{T}^d)=W^{a,2}(\mathbb{T}^d)$, and $H^{-a}(\mathbb{T}^d)$ stands for the topological dual of $H^a(\mathbb{T}^d)$.
$C^{k}_{\mathrm{loc}}(\mathbb{R})$ denotes $k$-th differential functions on any compact sets in $\mathbb{R}$. Let the bracket $\langle\cdot,\cdot\rangle$ stand for the duality between $C^{\infty}(\mathbb{T}^d)$ and the space of distributions over $\mathbb{T}^d$.
Einstein summation is used throughout this paper.
Define $Q_T:=\mathbb{T}^d\times[0,T)$.

Let $X$ be a real Banach space with norm $\|\cdot\|_X$. The space $L^p([0,T];X)$ denotes the standard Lebesgue space, and $W^{1,p}([0,T];X)$ denotes the standard Sobolev space.

\subsection{The assumption on the noise and the obstacle}\label{sec:propose definition}

Let $(\Omega,\mathcal{F},\mathbb{P},\{\mathcal{F}_t\}_{t\in
[0,T]},(\{B^k(t)\}_{t\in[0,T]})_{k\in\mathbb{N}})$ be a stochastic basis. Without loss of generality, here the filtration $\{\mathcal{F}_t\}_{t\in [0,T]}$ is assumed to be complete and $\{B^k(t)\}_{t\in[0,T]},k\in\mathbb{N}$, are independent ($d$-dimensional)  $\{\mathcal{F}_t\}_{t\in [0,T]}-$Wiener processes.
We use $\mathbb{E}$ to denote the expectation with respect to $\mathbb{P}$.
Let $F=\{f_k\}_{k\in\mathbb{N}}:\mathbb{T}^d\to\mathbb{R}$ be a sequence of continuously differentiable functions on $\mathbb{T}^d$. Define
\begin{align}\label{kk-43}
\xi^F=\sum_{k=1}^{\infty}f_k(x)B_t^k.
\end{align}
For the sequence $\{f_k\}_{k\in\mathbb{N}}$, define
\begin{equation}\label{noise}
	F_1=\sum_{k=1}^{\infty}|f_k|^2,\  F_2=\frac{1}{2}\sum_{k=1}^{\infty}\nabla f_k^2\  \text{and}\  F_3=\sum_{k=1}^{\infty}|\nabla f_k|^2.
	\end{equation}

\begin{assumption} \label{assu:assum for F}
The noise, obstacle and initial data are supposed to satisfy the following conditions.
\begin{description}
          \item[(i)] The functions $\{F_i\}_{i\in \{1,2,3\}}$ are continuous on $\mathbb{T}^d$ and the divergence $\nabla\cdot F_2=\frac{1}{2}\Delta F_1$ is bounded on $\mathbb{T}^d$.
          \item[(ii)] The nonnegative obstacle $\psi$ is continuous on $Q_T$.
          \item[(iii)] The initial data $u_{\mathrm{init}}$ is nonnegative,
 $\mathcal{F}_{0}$-measurable and satisfies $u_{\mathrm{init}}\leq\psi(\cdot,0)$ a.e.\ on $\Omega\times\mathbb{T}^d$.
\end{description}

\end{assumption}
\subsection{The kinetic formulation of the obstacle problem}
In this section, we conduct a formal calculation that motivates the definition of a stochastic kinetic solution to the obstacle problem with the upper obstacle constraint $u\leq \psi$, for which the equation
\begin{equation}
\mathrm{d}u=\Delta\Phi(u)\mathrm{d}t-\nabla\cdot(f^{k}(x)\sigma(u)\circ\mathrm{d}B_{t}^{k}+g(u)\mathrm{d}t)
\end{equation}
is formally satisfied on the non-contact set $\{(t,x):u(t,x)<\psi(t,x)\}$.
Our starting point is the following penalized equation, which is used to approximate the obstacle problem.
\[
\mathrm{d}u_{\varepsilon}=[\Delta\Phi(u_{\varepsilon})-\varepsilon^{-1}(u_{\varepsilon}-\psi)^{+}]\mathrm{d}t-\nabla\cdot(f_{k}(x)\sigma(u_{\varepsilon})\circ\mathrm{d}B_{t}^{k}+g(u_\varepsilon)\mathrm{d}t),\quad\varepsilon>0.
\]
The kinetic formulation of the generalized Dean--Kawasaki equation has been derived by \cite{fehrman2024well}, hence we will pay attention to the new terms arising from penalized term $\nu_{\varepsilon}:=\varepsilon^{-1}(u_{\varepsilon}-\psi)^{+}$.

For a smooth function $S:\mathbb{R}\to\mathbb{R}$, applying the It\^{o}'s formula, with $\{F_i\}_{i\in \{1,2,3\}}$ defined by (\ref{noise}), we have
\begin{align*}
\mathrm{d}S(u_{\varepsilon}) & =\nabla\cdot\big(\Phi^{\prime}(u_{\varepsilon})S^{\prime}(u_{\varepsilon})\nabla u_{\varepsilon}\big)\mathrm{d}t+\frac{1}{2}\nabla\cdot\big(F_{1}[\sigma^{\prime}(u_{\varepsilon})]^{2}S^{\prime}(u_{\varepsilon})\nabla u_{\varepsilon}+\sigma(u_{\varepsilon})\sigma^{\prime}(u_{\varepsilon})S^{\prime}(u_{\varepsilon})F_{2}\big)\mathrm{d}t\\
 & \quad-S^{\prime\prime}(u_{\varepsilon})\Big(\Phi^{\prime}(u_{\varepsilon})|\nabla u_{\varepsilon}|^{2}+\frac{1}{2}F_{1}|\nabla\sigma(u_{\varepsilon})|^{2}+\frac{1}{2}\sigma(u_{\varepsilon})\nabla\sigma(u_{\varepsilon})\cdot F_{2}\Big)\mathrm{d}t-S^{\prime}(u_{\varepsilon})\nabla\cdot g(u_{\varepsilon})\mathrm{d}t\\
 & \quad+\frac{1}{2}S^{\prime\prime}(u_{\varepsilon})\sum_{k=1}^{\infty}|\nabla(\sigma(u_{\varepsilon})f_{k})|^{2}\mathrm{d}t
 -S^{\prime}(u_{\varepsilon})\frac{1}{\varepsilon}(u_{\varepsilon}-\psi)^{+}\mathrm{d}t-S^{\prime}(u_{\varepsilon})\nabla\cdot(\sigma(u_{\varepsilon})f_{k})\mathrm{d}B_{t}^{k}.
\end{align*}
A simple calculation shows that
\begin{align*}
  S^{\prime}(u_{\varepsilon})\frac{1}{\varepsilon}(u_{\varepsilon}-\psi)^{+}
  &=(S^{\prime}(u_{\varepsilon})-S^{\prime}(\psi))\frac{1}{\varepsilon}(u_{\varepsilon}-\psi)^{+}+S^{\prime}(\psi)\nu_{\varepsilon}\\
  &= \int_{0}^{1}S^{\prime\prime}(\psi+\theta(u_{\varepsilon}-\psi))\mathrm{d}\theta\frac{1}{\varepsilon}[(u_{\varepsilon}-\psi)^{+}]^{2}+S^{\prime}(\psi)\nu_{\varepsilon}.
\end{align*}
Define the kinetic function $\chi_\varepsilon$ associated with $u_\varepsilon$ by
\[
\chi_\varepsilon(x,\xi,t):=\bar\chi(u_\varepsilon(x,t),\xi):=\mathbf{1}_{\{0<\xi<u_\varepsilon(x,t)\}},
\]
where $\bar\chi(r,\xi):=\mathbf{1}_{\{0<\xi<r\}}$ is the kinetic indicator function.
Clearly, it holds that
\begin{align*}
S(u_{\varepsilon}(x,t))=\int_{\mathbb{R}}S'(\xi)\chi_{\varepsilon}(x,\xi,t)d\xi,
\end{align*}
provided $S(0)=0$.
In view of
\[
S^{\prime}(\psi)\nu_{\varepsilon}=\int_{\mathbb{R}}S^{\prime}(\xi)\delta_{0}(\xi-\psi)\nu_{\varepsilon}\mathrm{d}\xi,
\]
and
\begin{align*}
 & \int_{0}^{1}S^{\prime\prime}(\psi+\theta(u_{\varepsilon}-\psi))\mathrm{d}\theta\frac{1}{\varepsilon}[(u_{\varepsilon}-\psi)^{+}]^{2}\\
 & =\int_{0}^{1}\int_{\mathbb{R}}S^{\prime\prime}(\xi)\delta_{0}(\xi-\psi+\theta(u_{\varepsilon}-\psi))\mathrm{d}\xi\mathrm{d}\theta\frac{1}{\varepsilon}[(u_{\varepsilon}-\psi)^{+}]^{2}\\
 & =\int_{\mathbb{R}}S^{\prime\prime}(\xi)\int_{0}^{1}\delta_{0}(\xi-(\psi+\theta(u_{\varepsilon}-\psi)))\frac{1}{\varepsilon}[(u_{\varepsilon}-\psi)^{+}]^{2}\mathrm{d}\theta\mathrm{d}\xi,
\end{align*}
we reach
\begin{align*}
\mathrm{d}\chi_{\varepsilon} & =\nabla\cdot\big(\Phi^{\prime}(\xi)\delta_{0}(\xi-u_{\varepsilon})\nabla u_{\varepsilon}\big)\mathrm{d}t+\frac{1}{2}\nabla\cdot\big[\delta_{0}(\xi-u_{\varepsilon})\big(F_{1}[\sigma^{\prime}(\xi)]^{2}\nabla u_{\varepsilon}+\sigma(\xi)\sigma^{\prime}(\xi)F_{2}\big)\big]\mathrm{d}t\\
 & \quad+\partial_{\xi}\Big[\delta_{0}(\xi-u_{\varepsilon})\Big(\Phi^{\prime}(\xi)|\nabla u_{\varepsilon}|^{2}-\frac{1}{2}\sigma(\xi)\nabla\sigma(u_{\varepsilon})\cdot F_{2}-\frac{1}{2}\sigma^{2}(\xi)F_{3}\Big)\Big]\mathrm{d}t-\delta_{0}(\xi-u_{\varepsilon})\nabla\cdot g(u_{\varepsilon})\mathrm{d}t\\
 & \quad+\partial_{\xi}\int_{0}^{1}\frac{1}{\varepsilon}[(u_{\varepsilon}-\psi)^{+}]^{2}\delta_{0}(\xi-(\psi+\theta(u_{\varepsilon}-\psi)))\mathrm{d}\theta\mathrm{d}t-\delta_{0}(\xi-\psi)\nu_{\varepsilon}\mathrm{d}t\\
 & \quad-\delta_{0}(\xi-u_{\varepsilon})\nabla\cdot(\sigma(u_{\varepsilon})f_{k})\mathrm{d}B_{t}^{k}.
\end{align*}

When $\varepsilon\to0$, the limiting kinetic function $\chi$ does not, in general, satisfy the above equation exactly. In the framework of entropy solutions, this defect is reflected in the appearance of an entropy inequality; see \cite{du2024well}.
On the kinetic level, the entropy inequality is quantified exactly by a parabolic defect measure and an obstacle defect measure, which are described in detail below.

Regarding the parabolic defect measure associated with $u_\varepsilon$, it is a nonnegative measure $m_\varepsilon$ on $\mathbb{T}^{d}\times\mathbb{R}\times[0,T]$.
More precisely, in the sense of measures,
\[
\delta_{0}(\xi-u_{\varepsilon})\Phi^{\prime}(\xi)|\nabla u_{\varepsilon}|^{2}\leq m_\varepsilon.
\]

We now analyse the terms associated with the obstacle problem.
Let $\lambda_{\varepsilon}$ denote a nonnegative measure on $\mathbb{T}^d\times \mathbb{R}\times [0,T]$ defined by
\begin{align*}
d\lambda_{\varepsilon}(x,t,\xi):=\int_{0}^{1}\frac{1}{\varepsilon}[(u_{\varepsilon}-\psi)^{+}]^{2}\delta_{0}(\xi-(\psi+\theta(u_{\varepsilon}-\psi)))\mathrm{d}\theta \mathrm{d}x\mathrm{d}t\mathrm{d}\xi.
\end{align*}
A priori estimates for the penalized equation (see Lemma \ref{lem:priori-estimates} and Lemma \ref{lem:(Existence-of-solution for penalized eqn } below) are given by
\begin{align}\label{r-1}
\mathbb{E}\Vert{\varepsilon}^{-1}[(u_\varepsilon-\psi)^{+}]^2\Vert_{L^{1}(Q_{T})}+\mathbb{E}\Vert\varepsilon^{-1}(u_{\varepsilon}-\psi)^{+}\Vert_{L^{1}(Q_{T})}\leq C(1+\Vert u_{\mathrm{init}}\Vert_{L^{2}(\mathbb{T}^d)}^2),
\end{align}
where the constant $C$ is independent of $\varepsilon$.
In particular, the family $\{\lambda_\varepsilon\}_{\varepsilon>0}$ is uniformly bounded in the space of finite nonnegative Radon measures on
$\mathbb{T}^d\times\mathbb{R}\times[0,T]$.
Hence, up to a subsequence, $\lambda_\varepsilon$ converges weakly to a nonnegative measure $\lambda$ in the sense of measures,
that is, for every $\phi\in C_c(\mathbb{T}^d\times\mathbb{R}\times[0,T])$,
\[
\int_{\mathbb{T}^d\times\mathbb{R}\times[0,T]}\phi\,\mathrm{d}\lambda_\varepsilon
\rightarrow
\int_{\mathbb{T}^d\times\mathbb{R}\times[0,T]}\phi\,\mathrm{d}\lambda .
\]
The limiting measure $\lambda$ is called the obstacle defect measure.
We define $q:=m+\lambda$, which corresponds to the total kinetic measure for the limiting equation of the penalized equations.
Furthermore, estimate (\ref{r-1}) ensures that $\nu_{\varepsilon}$ is a finite measure on $Q_T$, thus it admits a limiting measure denoted by $\nu$.

Prior to defining a stochastic kinetic solution to \eqref{eq:main}, we first introduce the concept of kinetic measure.
\begin{defn}
Let $(\Omega,\mathcal{F},\mathbb{P})$ be a probability space with a filtration $(\mathcal{F}_{t})_{t\in[0,\infty)}$.
A kinetic measure is a map $q$ from $\Omega$ to the space of nonnegative, locally finite measures on $\mathbb{T}^{d}\times(0,\infty)\times[0,T]$ that satisfies the property that the process
\begin{align*}
(\omega,t)\in \Omega\times[0,T]\to\int^t_0\int_{\mathbb{R}}\int_{\mathbb{T}^d}\phi(x,\xi)\mathrm{d}q(x,\xi,r)
\end{align*}
is $\mathcal{F}_t-$predictable, for every $\phi\in C^{\infty}_c(\mathbb{T}^d\times (0,\infty))$.
\end{defn}

In the sequel, for any $\phi\in C_{c}^{1}(\mathbb{T}^{d}\times(0,\infty)\times[0,T))$, we write $(\nabla \phi)(x,u(x,t),t)=\nabla \phi(x,\xi,t)\big|_{\xi=u(x,t)}$, meaning that the gradient $\nabla \phi$ is evaluated at the point $(x,u(x,t),t)$.
We are now ready to define the stochastic kinetic solution to the obstacle problem for \eqref{eq:main}.

\begin{defn}
\label{def:def of stochastic kinetic solution}Let $u_{\mathrm{init}}\in L^{1}(\Omega;L^{1}(\mathbb{T}^{d}))$
be nonnegative and $\mathcal{F}_{0}$-measurable.
A stochastic kinetic solution of the obstacle problem for \eqref{eq:main} is a pair $(u,\nu)$, where $u\in L^{1}(\Omega\times[0,T];L^{1}(\mathbb{T}^{d}))$ is a nonnegative, $L^{1}(\mathbb{T}^{d})$-valued $\mathcal{F}_{t}$-predictable function, and $\nu$ is a predictable Radon measure on $Q_T$ satisfying the following properties.
\begin{itemize}
\item[(i)] Almost surely $u\leq \psi$, and $\mathbb{E}[\nu(Q_T)]<\infty$.

\item[(ii)] Preservation of mass: almost surely for every $\phi\in C_c^{\infty}([0,T))$,
\begin{align}\label{eq:L1 perservation}
    -\int_{0}^{T}\int_{\mathbb{T}^d} u(x,s)\partial_s\phi(s)\mathrm{d}x\mathrm{d}s+\int_{0}^{T}\int_{\mathbb{T}^d}\phi(s)\mathrm{d}\nu(x,s)=\phi(0)\int_{\mathbb{T}^d} u_{\mathrm{init}}\mathrm{d}x.
\end{align}
\item[(iii)] The initial condition: almost surely
\begin{align}\label{eq:limit for uinit}
\lim_{\tau\to0}\frac{1}{\tau}\int_0^{\tau}\int_{\mathbb{T}^d}|u(x,t)-u_{\mathrm{init}}(x)|\mathrm{d}x\mathrm{d}t=0.
\end{align}
\item[(iv)]  $\sigma(u)\in L^{2}(\Omega;L^{2}(Q_T))$, \quad $g(u)\in L^{1}(\Omega;L^{1}(Q_T;\mathbb{R}^d))$.

\item[(v)] For every $K\in\mathbb{N}$,
\[
[(u\land K)\lor(1/K)]\in L^{2}(\Omega\times[0,T];H^{1}(\mathbb{T}^{d}))).
\]
\end{itemize}
Furthermore, there exists a kinetic measure $q$ satisfying the following three properties.
\begin{itemize}
\item[(vi)] Almost surely as a nonnegative measure, $q$ satisfies
\[
\delta_{0}(\xi-u)\Phi^{\prime}(\xi)|\nabla u|^{2}\leq q\quad\text{on}\quad\mathbb{T}^{d}\times(0,\infty)\times[0,T].
\]
\item[(vii)] (Vanishing)
\[
\lim_{N\to\infty}\mathbb{E}\big[q(\mathbb{T}^{d}\times[N,N+1]\times[0,T))\big]=0.
\]
\item[(viii)] For every $\phi\in C_{c}^{1}(\mathbb{T}^{d}\times(0,\infty)\times[0,T))$, the kinetic function $\chi$ of $u$ almost surely satisfies
\begin{align}
& -\int_{0}^{T}\int_{\mathbb{R}}\int_{\mathbb{T}^d}\chi(x,\xi,s)\partial_s\phi(x,\xi,s)\mathrm{d}x\mathrm{d}\xi\mathrm{d}s\label{eq:kinetic equation test with time}\\
 & =\int_{\mathbb{R}}\int_{\mathbb{T}^{d}}\bar{\chi}(u_{\mathrm{init}},\xi)\phi(x,\xi,0)\mathrm{d}x\mathrm{d}\xi-\int_{0}^{T}\int_{\mathbb{T}^{d}}\Phi^{\prime}(u)\nabla u\cdot(\nabla\phi)(x,u,s)\mathrm{d}x\mathrm{d}s\nonumber \\
 & \quad-\frac{1}{2}\int_{0}^{T}\int_{\mathbb{T}^{d}}\big(F_{1}(x)[\sigma^{\prime}(u)]^{2}\nabla u+\sigma(u)\sigma^{\prime}(u)F_{2}(x)\big)\cdot(\nabla\phi)(x,u,s)\mathrm{d}x\mathrm{d}s\nonumber\\
 & \quad-\int_{0}^{T}\int_{\mathbb{R}}\int_{\mathbb{T}^{d}}\partial_{\xi}\phi(x,\xi,s)\mathrm{d}q(x,\xi,s)-\int_{0}^{T}\int_{\mathbb{T}^d}\phi(x,u,s)\nabla\cdot g(u)\mathrm{d}x\mathrm{d}s\nonumber \\
 & \quad+\frac{1}{2}\int_{0}^{T}\int_{\mathbb{T}^{d}}\Big(\sigma(u)\sigma^{\prime}(u)\nabla u\cdot F_{2}(x)+\sigma^{2}(u)F_{3}(x)\Big)(\partial_{\xi}\phi)(x,u,s)\mathrm{d}x\mathrm{d}s\nonumber \\
 & \quad-\int_{0}^{T}       \int_{\mathbb{T}^{d}}\phi(x,\psi,s)\mathrm{d}\nu(x,s)-\int_{0}^{T}\int_{\mathbb{T}^{d}}\phi(x,u,s)\nabla\cdot(\sigma(u)f_{k})\mathrm{d}x\mathrm{d}B_{s}^{k}.\nonumber
\end{align}
\end{itemize}
\end{defn}
\begin{rem}
The relevant literature \cite{DV10,fehrman2024well} indicates that the kinetic measure $q$ introduced in Definition \ref{def:def of stochastic kinetic solution} does not affect the continuity of the solution $u$.
This result follows from (vii) in Definition \ref{def:def of stochastic kinetic solution}, together with the observation that the integral of $\partial_{\xi}\phi$ against the measure $q$ in \eqref{eq:kinetic equation test with time} vanishes when the test function $\phi$ approximates a function independent of $\xi$.
In the setting of the obstacle problem, however, there will be additional terms involving Radon measure $\nu$.
Crucially, the integral of $\phi$ against the measure $\nu$ does not vanish when $\phi$ is independent of $\xi$.
The atoms of the measure $\nu$ can therefore cause discontinuities in $u$ at a random (almost surely countable) set of points.
As the standard doubling of variables method for kinetic solutions does not involve mollification in time, this necessitates a special treatment using, for instance, the left- and right-continuous versions of the kinetic function as outlined in Proposition \ref{prop:atom point in t}.
\end{rem}

\begin{rem}
Compared with the work \cite{fehrman2024well}, the presence of the Radon measure $\nu$ necessitates several modifications to the definition of the kinetic solution.
The main adjustments are as follows:
(1) We introduce a revised condition (ii) in Definition \ref{def:def of stochastic kinetic solution} to replace the condition $\|\rho(t)\|_{L^1(\mathbb{T}^d)}=\|\rho_0\|_{L^1(\mathbb{T}^d)}$ used in \cite{fehrman2024well}, which is the key to establish the vanishing property of the kinetic measure at $\xi=0$ (see Proposition \ref{prop:proposition for limit measure}).
(2) Owing to the presence of the Radon measure $\nu$, the function $u$ may be discontinuous in time and hence cannot be expected to satisfy the initial condition in the classical sense.
To remedy it, we incorporate a new condition (iii) in Definition \ref{def:def of stochastic kinetic solution}. This condition is naturally fulfilled by the solution constructed as the limit of penalized solutions; see Lemma \ref{lem:converge to u init} and the proof of Theorem \ref{thm:existence} for further details.
\end{rem}

Since the function $u$ may have discontinuities in time, we need to establish the existence of the left and right weak limits.
The proof can be done by a similar argument as in \cite[Proposition 8]{DV10}, thus we omit it.
\begin{prop}\label{prop:atom point in t}
Let $(u,\nu)$ be a stochastic kinetic solution of the obstacle problem for \eqref{eq:main} under Definition \ref{def:def of stochastic kinetic solution}, and $\chi$ be a kinetic function of $u$.
Then, $\chi$ admits almost surely left and right limits at all points $t\in[0,T]$, which means that for all $t\in[0,T]$ there exist some kinetic functions $\tilde{\chi}_{t}^{\pm}$ on $\Omega\times\mathbb{T}^d\times\mathbb{R}$ such that $\mathbb{P}$-a.s.
\[
\int_{\mathbb{T}^d}\int_{\mathbb{R}}\chi(x,\xi,t\pm\epsilon)\phi(x,\xi)\mathrm{d}\xi\mathrm{d}x\to \int_{\mathbb{T}^d}\int_{\mathbb{R}}\tilde{\chi}_{t}^{\pm}(x,\xi)\phi(x,\xi)\mathrm{d}\xi\mathrm{d}x,\quad \text{as }\epsilon\to 0,
\]
for all $\phi\in C_c^{1}(\mathbb{T}^d\times(0,\infty))$ and almost surely
\begin{align}
&\int_{\mathbb{T}^d}\int_{\mathbb{R}}\big[\tilde{\chi}_{t}^{+}(x,\xi)-\tilde{\chi}_{t}^{-}(x,\xi)\big]\phi(x,\xi)\mathrm{d}\xi\mathrm{d}x\label{eq:right left differnt in t}\\
&=-\int_{0}^{T}\int_{\mathbb{R}}\int_{\mathbb{T}^{d}}\partial_{\xi}\phi(x,\xi)\mathbf{1}_{\{t\}}(s)\mathrm{d}q(x,\xi,s)
-\int_{0}^{T}\int_{\mathbb{T}^{d}}\phi(x,\psi(x,s))\mathbf{1}_{\{t\}}(s)\mathrm{d}\nu(x,s)\nonumber.
\end{align}
In particular, almost surely, the set of $t\in[0,T]$ such that $\tilde{\chi}_{t}^{-}\neq\tilde{\chi}_{t}^+$ is countable.
\end{prop}

Based on the above proposition, we define the predictable function $\chi^{\pm}$ by
\[
\chi^{\pm}(t)=\begin{cases}
\tilde{\chi}_{t}^{\pm},&t\in[0,T] \text{ such that }\tilde{\chi}_{t}^{-}\neq \tilde{\chi}_{t}^{+}\\
\chi(t), &\text{others}.
\end{cases}
\]
Note that, with respect to the filtration generated by the Brownian motions, both $\chi^{+}$ and $\chi^{-}$ are predictable.
Moreover, we have $\chi=\chi^{+}=\chi^{-}$ almost everywhere with respect to the Lebesgue measure on the time interval $[0,T]$.
Hence, the choice among $\chi$, $\chi^{+}$, and $\chi^{-}$ is immaterial in integrals with respect to the Lebesgue measure, as well as in stochastic integrals.
However, for integrals with respect to the measures $\nu$ and $q$, the representative $\chi$ is not sufficient to define the integral unambiguously.
In this case, one must explicitly specify either $\chi^{+}$ or $\chi^{-}$; the two choices may, in general, lead to different values.

For $s,t\in(0,T)$ satisfying $s<t$ and $\epsilon\in(0,(T-t)\land s)$, we can take smooth test functions
in \eqref{eq:kinetic equation test with time} which converge to $\alpha_{\epsilon,s,t}(r)\phi(x,\xi)$, where $\phi\in C^{1}_c(\mathbb{T}^d\times(0,\infty))$, and $\alpha_{\epsilon,s,t}$ is defined by
\begin{equation*}
\alpha_{\epsilon,s,t}(r):=\begin{cases}
1,&s\leq r\leq t,\\
1+(r-s)/\epsilon,&s-\epsilon\leq r\leq s,\\
1-({r-t})/{\epsilon},&t\leq r\leq t+\epsilon,\\
0,&r\geq t+\epsilon \ \text{ or  }\ r\leq s-\epsilon.
\end{cases}
\end{equation*}
Then, taking the limit $\epsilon\to0$ and using the Lebesgue dominated convergence theorem, we get
\begin{align}
& \int_{\mathbb{R}}\int_{\mathbb{T}^d}\chi^{+}(x,\xi,t)\phi(x,\xi)\mathrm{d}x\mathrm{d}\xi-\int_{\mathbb{R}}\int_{\mathbb{T}^d}\chi^{-}(x,\xi,s)\phi(x,\xi)\mathrm{d}x\mathrm{d}\xi
\label{eq:kinetic equation chi+}\\
 & =-\int_{s}^{t}\int_{\mathbb{T}^{d}}\Phi^{\prime}(u)\nabla u\cdot(\nabla\phi)(x,u)\mathrm{d}x\mathrm{d}\tau\nonumber \\
 & \quad-\frac{1}{2}\int_{s}^{t}\int_{\mathbb{T}^{d}}\big(F_{1}(x)[\sigma^{\prime}(u)]^{2}\nabla u+\sigma(u)\sigma^{\prime}(u)F_{2}(x)\big)\cdot(\nabla\phi)(x,u)\mathrm{d}x\mathrm{d}\tau\nonumber\\
 & \quad-\int_{[s,t]}\int_{\mathbb{R}}\int_{\mathbb{T}^{d}}\partial_{\xi}\phi(x,\xi)\mathrm{d}q(x,\xi,s)-\int_{s}^{t}\int_{\mathbb{T}^d}\phi(x,u)\nabla\cdot g(u)\mathrm{d}x\mathrm{d}\tau\nonumber \\
 & \quad+\frac{1}{2}\int_{s}^{t}\int_{\mathbb{T}^{d}}\Big(\sigma(u)\sigma^{\prime}(u)\nabla u\cdot F_{2}(x)+\sigma^{2}(u)F_{3}(x)\Big)(\partial_{\xi}\phi)(x,u)\mathrm{d}x\mathrm{d}\tau\nonumber \\
 & \quad-\int_{[s,t]}\int_{\mathbb{T}^{d}}\phi(x,\psi)\mathrm{d}\nu(x,\tau)-\int_{s}^{t}\int_{\mathbb{T}^{d}}\phi(x,u)\nabla\cdot(\sigma(u)f_{k})\mathrm{d}x\mathrm{d}B_{\tau}^{k}.\nonumber
\end{align}

\section{Uniqueness of stochastic kinetic solutions}
\label{sec:unique}
In this section, we establish the uniqueness of stochastic kinetic solutions to the obstacle problem for \eqref{eq:main}.
Our overall strategy is inspired by the uniqueness argument developed in \cite{fehrman2024well} for the generalized Dean--Kawasaki equation.
A key point of our result is that this argument can be extended to the obstacle problem under exactly the same assumptions on the coefficients $\Phi$, $\sigma$ and $g$, without introducing any additional conditions to accommodate the obstacle constraint.

The assumptions used below are the same as those in Assumption 4.1 of \cite{fehrman2024well}.
\begin{assumption}\label{assu:assu for unique}
Suppose that $\Phi$, $\sigma\in C([0,\infty))$ and $g\in C([0,\infty);\mathbb{R}^d)$ satisfy the followings.
\begin{itemize}
\item[1.] The functions $\Phi$, $\sigma\in C^{1,1}_{\mathrm{loc}}((0,\infty))$ and $g\in C^{1}_{\mathrm{loc}}((0,\infty);\mathbb{R}^d)$.
\item[2.] The function $\Phi(0)=0$ and with $\Phi'>0$ on $(0,\infty)$.
\item[3.] There exists $c\in (0,\infty)$ such that
$
\limsup_{\xi\rightarrow 0^+}{\sigma^2(\xi)}/{\xi}\leq c,
$
which implies that $\sigma(0)=0$.
\item[4.] Either that $\sigma\sigma^\prime\in C([0,\infty))$ with $(\sigma\sigma')(0)=0$ or that $\nabla\cdot F_2=0$ for $F_2$ defined by (\ref{noise}).
\item[5.] There exists $c\in [1,\infty)$ such that
\begin{align*}
\sup_{\xi'\in [0,\xi]}\sigma^2(\xi')\leq c(1+\xi+\sigma^2(\xi)) \quad {\rm{for}}\ \xi\in[0,\infty).
\end{align*}
\item[6.] There exists $c\in [1,\infty)$ such that
\begin{align*}
\sup_{\xi'\in [0,\xi]}|g(\xi')|\leq c(1+\xi+|g(\xi)|) \quad {\rm{for}}\ \xi\in[0,\infty).
\end{align*}
\end{itemize}
\end{assumption}

In the sequel, for simplicity, we write $\int_{t}$ in place of $\int_{0}^{T}\cdot\mathrm{d}t$ (and similarly for $\int_{s}$), $\int_{x}$ in place of $\int_{\mathbb{T}^{d}}\cdot\mathrm{d}x$ (and similarly for $\int_{y}$ and $\int_{z}$), $\int_{\eta}$ in place of $\int_{\mathbb{R}}\cdot\mathrm{d}\eta$ (and similarly for $\int_{\xi}$ and $\int_{\tilde{\xi}}$) and $\int_{\theta}$ in place of $\int_{0}^{1}\cdot\mathrm{d}\theta$.
However, to avoid confusion, we will use the usual notation when the integral is taken on a different domain or the measure is not Lebesgue's measure.

We first recall an integration-by-parts identity in the kinetic formulation, taken from \cite[Lemma 4.3]{fehrman2024well}, and then introduce the regularization kernels used below to handle the singular coefficients.
\begin{lem}
\label{lem:intergrating by part}Let $u\in H^{1}(\mathbb{T}^{d})$
be a nonnegative function, and $\chi$ be the kinetic function of $u$.
Then, for every $\phi\in C_{c}^{\infty}(\mathbb{T}^{d}\times(0,\infty))$,
\[
\int_{x,\xi}\nabla_{x}\phi(x,\xi)\chi(x,\xi)=-\int_{x}\phi(x,u(x))\nabla_{x}u(x).
\]
In particular, if $(u,\nu)$ is a stochastic kinetic solution in the
sense of Definition \ref{def:def of stochastic kinetic solution}, then almost surely for every $\phi\in C_{c}^{\infty}(\mathbb{T}^{d}\times(0,\infty)\times[0,T])$,
\[
\int_{x,\xi,s}\nabla_{x}\phi(x,\xi,s)\chi(x,\xi,s)=-\int_{x,s}\phi(x,u(x,s),s)\nabla_{x}u(x,s).
\]
\end{lem}
\begin{defn}
For every $\delta,\varsigma\in(0,1)$, let $\kappa_{d}^{\varsigma}:\mathbb{T}^{d}\to[0,\infty)$ and $\kappa_{1}^{\delta}:\mathbb{R}\to[0,\infty)$ be standard convolution
kernels of scales $\varsigma$ and $\delta$ on $\mathbb{T}^{d}$ and $\mathbb{R}$ respectively, and let $\kappa^{\varsigma,\delta}$ be
\begin{align}\label{rr-1}
    \kappa^{\varsigma,\delta}(x,y,\xi,\eta):=\kappa_{d}^{\varsigma}(x-y)\kappa_{1}^{\delta}(\xi-\eta),\quad\text{for every }(x,y,\xi,\eta)\in(\mathbb{T}^{d})^{2}\times\mathbb{R}^{2}.
\end{align}
For every $\beta\in(0,1)$, let $\varphi_{\beta}\colon\mathbb{R}\to[0,1]$
be the unique non-decreasing piecewise linear function that satisfies
\[
\varphi_{\beta}(\xi)=1\text{ if }\xi\geq\beta,\quad\varphi_{\beta}(\xi)=0\text{ if }\xi\leq\frac{\beta}{2},\quad\text{and}\quad\varphi_{\beta}^{\prime}(\xi)=\frac{2}{\beta}\mathbf{1}_{\{\beta/2\leq\xi\leq\beta\}},
\]
and for every $N\in\mathbb{N}$, let $\zeta_{N}\colon\mathbb{R}\to[0,1]$
be the unique non-increasing piecewise linear function that satisfies
\[
\zeta_{N}(\xi)=0\text{ if }\xi\geq N+1,\quad\zeta_{N}(\xi)=1\text{ if }\xi\leq N,\quad\text{and}\quad\zeta_{N}^{\prime}(\xi)=-\mathbf{1}_{\{N\leq\xi\leq N+1\}}.
\]
\end{defn}

\begin{prop}
\label{prop:proposition for limit measure}Let Assumption \ref{assu:assum for F} and Assumption \ref{assu:assu for unique} hold for $\xi^F$, $\psi$, $u_{\mathrm{init}}$, $\Phi$ and $\sigma$.
Let $(u,\nu)$ be stochastic kinetic solutions of the obstacle problem for \eqref{eq:main} in the sense of Definition \ref{def:def of stochastic kinetic solution}. Then, it follows almost surely that
\[
\underset{\beta\to0}{\mathrm{lim\,inf}}\big(\beta^{-1}q(\mathbb{T}^{d}\times[\beta/2,\beta]\times[0,T))\big)=0.
\]
\end{prop}

\begin{proof}
For $\epsilon>0$, $N\in\mathbb{N}$, and $\beta\in(0,1)$, by taking smooth test functions in \eqref{eq:kinetic equation test with time} which converge to $\tilde{\alpha}_{\epsilon,T}\zeta_{N}\varphi_{\beta}$ where
\begin{equation}\label{r-2}
\tilde{\alpha}_{\epsilon,T}(s):=\begin{cases}
1,&0\leq s\leq T-\epsilon,\\
({T-s})/{\epsilon},&T-\epsilon\leq s\leq T,\\
0,&s\geq T,
\end{cases}
\end{equation}
taking the limit $\epsilon\to0$, and using the Lebesgue's dominated convergence theorem, we have
\begin{align}
 & \mathbb{E}\Big[2\beta^{-1}q(\mathbb{T}^{d}\times[\beta/2,\beta]\times[0,T))\Big]\label{eq:kinetic eqn} \\
 & =-\mathbb{E}\Bigg[\int_{x,\xi}\chi^{-}(x,\xi,T)\zeta_{N}(\xi)\varphi_{\beta}(\xi)\Bigg]
 +\mathbb{E}\Bigg[\int_{x,\xi}\bar{\chi}(u_{\mathrm{init}},\xi)\zeta_{N}(\xi)\varphi_{\beta}(\xi)\Bigg]\nonumber\\
 &\quad-\mathbb{E}\Bigg[\int_{x,t}\zeta_{N}(u)\varphi_{\beta}(u)\nabla\cdot g(u)\Bigg]\nonumber\\
 & \quad+\mathbb{E}\Bigg[\int_{x,t}\Big(\beta^{-1}\mathbf{1}_{\{\beta/2\leq u\leq\beta\}}-\frac{1}{2}\mathbf{1}_{\{N\leq u\leq N+1\}}\Big)\Big(\sigma(u)\sigma^{\prime}(u)\nabla u\cdot F_{2}+\sigma^{2}(u)F_{3}\Big)\Bigg]\nonumber\\
 & \quad+\mathbb{E}\Big[q(\mathbb{T}^{d}\times[N,N+1]\times[0,T))\Big]-\mathbb{E}\Bigg[\int_{[0,T)}\int_{\mathbb{T}^{d}}\zeta_{N}(\psi)\varphi_{\beta}(\psi)\mathrm{d}\nu(x,t)\Bigg].\nonumber
\end{align}
We now estimate each term on the righthand side of \eqref{eq:kinetic eqn}.
Using Assumption \ref{assu:assum for F} and Assumption \ref{assu:assu for unique} and following \cite[the proof of Proposition 4.5]{fehrman2024well},
we have
\begin{align*}
\lim_{\beta\to0}\beta^{-1}\mathbb{E}\Bigg[\int_{x,t}\mathbf{1}_{\{\beta/2\leq u\leq\beta\}}\Big(\sigma(u)\sigma^{\prime}(u)\nabla u\cdot F_{2}+\sigma^{2}(u)F_{3}\Big)\Bigg] & =0,
\end{align*}
and
\begin{align*}
 & \lim_{N\to\infty}\mathbb{E}\Bigg[\int_{x,t}\mathbf{1}_{\{N\leq u\leq N+1\}}\Big(\sigma(u)\sigma^{\prime}(u)\nabla u\cdot F_{2}+\sigma^{2}(u)F_{3}\Big)\Bigg]=0.
\end{align*}
Note that
\begin{align*}
\int_{x}\zeta_{N}(u)\varphi_{\beta}(u)\nabla\cdot g(u)=\int_{x}\nabla\cdot\int_0^u\zeta_{N}(r)\varphi_{\beta}(r) g(r)\mathrm{d}r=0.
\end{align*}
The Definition \ref{def:def of stochastic kinetic solution} (vii) yields
$
\lim_{N\to\infty}\mathbb{E}\big[q(\mathbb{T}^{d}\times[N,N+1]\times[0,T))\big]=0.
$
The remaining three terms satisfy
\begin{align*}
 & \lim_{\beta\to0}\lim_{N\to\infty}\mathbb{E}\Bigg[\int_{[0,T)}\int_{\mathbb{T}^{d}}\zeta_{N}(\psi)\varphi_{\beta}(\psi)\mathrm{d}\nu(x,t)\Bigg]\\
 & +\lim_{\beta\to0}\lim_{N\to\infty}\mathbb{E}\Bigg[\int_{x,\xi}\chi^{-}(x,\xi,T)\zeta_{N}(\xi)\varphi_{\beta}(\xi)-\int_{x,\xi}\bar{\chi}(u_{\mathrm{init}},\xi)\zeta_{N}(\xi)\varphi_{\beta}(\xi)\Bigg]\\
 & =\mathbb{E}\big[\nu(\mathbb{T}^d\times [0,T))\big]+\mathbb{E}\Bigg[\int_{x,\xi} \big(\chi^{-}(x,\xi,T)-\bar{\chi}(u_{\mathrm{init}},\xi)\big)\Bigg].
\end{align*}
Taking smooth test functions in \eqref{eq:L1 perservation} which converge to $\tilde{\alpha}_{\epsilon,T}$, using the fact that $\chi$ is the kinetic function of $u$, and taking the limit $\epsilon\to 0$, by the dominated convergence theorem, we reach
\begin{align*}
    \int_{x,\xi}\chi^{-}(x,\xi,T)+\nu(\mathbb{T}^d\times [0,T))=\int_{x,\xi}\bar{\chi}(u_{\mathrm{init}},\xi),
\end{align*}
which implies that
\begin{align*}
 & \lim_{\beta\to0}\lim_{N\to\infty}\mathbb{E}\Bigg[\int_{[0,T)}\int_{\mathbb{T}^{d}}\zeta_{N}(\psi)\varphi_{\beta}(\psi)\mathrm{d}\nu(x,t)\Bigg]\\
 & +\lim_{\beta\to0}\lim_{N\to\infty}\mathbb{E}\Bigg[\int_{x,\xi}\big(\chi^{-}(x,\xi,T)-\bar{\chi}(u_{\mathrm{init}},\xi)\big)\zeta_{N}(\xi)\varphi_{\beta}(\xi)\Bigg]=0.
\end{align*}
Combining all these estimates with \eqref{eq:kinetic eqn}, we have
\begin{align*}
  \lim_{\beta\to0}\mathbb{E}\Big[2\beta^{-1}q(\mathbb{T}^{d}\times[\beta/2,\beta]\times[0,T))\Big]=0.
\end{align*}
The desired result is completed by using Fatou's lemma.
\end{proof}
\begin{thm}
\label{thm:stability L1}Let Assumption \ref{assu:assum for F} and Assumption \ref{assu:assu for unique} hold for $\xi^F, \psi, \Phi, \sigma$ and two initial functions $u_{1,\mathrm{init}}$ and $u_{2,\mathrm{init}}$.
Let $(u_{1},\nu_{1})$ and $(u_{2},\nu_{2})$ be stochastic kinetic solutions of the obstacle problem for \eqref{eq:main} in the sense of Definition \ref{def:def of stochastic kinetic solution} with initial data $u_{1,\mathrm{init}}$ and $u_{2,\mathrm{init}}$, respectively.
Let $\chi_{1}$ and $\chi_{2}$ be the kinetic functions of $u_{1}$ and $u_{2}$.
Then, almost surely, for almost every $s,t\in(0,T)$ with $s<t$,
\[
\int_{x,\xi} \big|\chi_{1}(x,\xi,t)-\chi_{2}(x,\xi,t)\big|^2\leq\int_{x,\xi} \big|\chi_{1}(x,\xi,s)-\chi_{2}(x,\xi,s)\big|^2.
\]
\end{thm}
\begin{proof}
Based on Proposition \ref{prop:atom point in t}, we consider the case that $s$ and $t$ in the dense set
\[
\mathcal{T}:=\{\tau\in(0,T)|\tilde{\chi}_i^-(\tau)=\tilde{\chi}_i^+(\tau)=\chi_i(\tau), \text{as a distribution on }\mathbb{T}^d\times{\mathbb{R}}\text{ for }i=1,2\}
\]
in $(0,T)$. Clearly, $\chi_{i}^{+}$ is c\`{a}dl\`{a}g. Recall that $\kappa^{\varsigma,\delta}$ is defined by \eqref{rr-1}.
For every $\varsigma,\delta\in(0,1)$ and $i\in\{1,2\}$, let
\begin{align*}
\chi_{i,t}^{\varsigma,\delta}(y,\eta)&:=\int_{x,\xi}\chi_{i}(x,\xi,t)\kappa^{\varsigma,\delta}(x,y,\xi,\eta),\\
 \chi_{i,t}^{\varsigma,\delta,\pm}(y,\eta)&:=\int_{x,\xi}\chi_{i}^{\pm}(x,\xi,t)\kappa^{\varsigma,\delta}(x,y,\xi,\eta),\\
\bar{\kappa}_{i,t}^{\varsigma,\delta}(x,y,\eta)&:=\kappa^{\varsigma,\delta}(x,y,u_{i}(x,t),\eta).
\end{align*}
With the aid of the distributional equalities $\partial_{\xi}\chi_{i}=\delta_0(\xi)-\delta_0(\xi-u_i)$ and $\partial_{x_j}\chi_{i}=\delta_0(\xi-u_i)\partial_{x_j} u_i$ for $j\in\{1,2,\ldots,d\}$, by the property of $\kappa^{\varsigma,\delta}$, we reach
\begin{align}
\partial_{\eta}\chi_{i,t}^{\varsigma,\delta}(y,\eta) 
 & =-\int_{x,\xi}\chi_{i}(x,\xi,t)\partial_{\xi}\kappa^{\varsigma,\delta}(x,y,\xi,\eta)
  =\int_{x,\xi}\partial_{\xi}\chi_{i}(x,\xi,t)\kappa^{\varsigma,\delta}(x,y,\xi,\eta)\label{eq:partial eta}\\
 & =\int_{x}\kappa^{\varsigma,\delta}(x,y,0,\eta)-\int_{x}\bar{\kappa}_{i,t}^{\varsigma,\delta}(x,y,\eta),\nonumber
\end{align}
and
\begin{align*}
\partial_{y_j}\chi_{i,t}^{\varsigma,\delta}(y,\eta)
 & =-\int_{x,\xi}\chi_{i}(x,\xi,t)\partial_{x_j}\kappa^{\varsigma,\delta}(x,y,\xi,\eta) =\int_{x,\xi}\partial_{x_j}\chi_{i}(x,\xi,t)\kappa^{\varsigma,\delta}(x,y,\xi,\eta)\\
 & =\int_{x}\partial_{x_j}u_{i}(x,t)\bar{\kappa}_{i,t}^{\varsigma,\delta}(x,y,\eta).
\end{align*}
Then, by taking the test function $\phi(x,\xi)=\kappa^{\varsigma,\delta}(x,y,\xi,\eta)$ in \eqref{eq:kinetic equation chi+}, based on the definition of the set $\mathcal{T}$, there exists a subset of full probability such that, for every $i\in\{1,2\}$ and $(y,\eta)\in\mathbb{T}^{d}\times(\delta/2,\infty)$,
\begin{align*}
 & \chi_{i,t}^{\varsigma,\delta,+}(y,\eta)-\chi_{i,s}^{\varsigma,\delta}(y,\eta)\\
 & =\nabla_{y}\cdot\bigg(\int_{s}^{t}\int_{x}\Phi^{\prime}(u_{i})\nabla_x u_{i}\bar{\kappa}_{i,\tau}^{\varsigma,\delta}(x,y,\eta)\mathrm{d}\tau\bigg)\\
 & \quad+\frac{1}{2}\nabla_{y}\cdot\bigg(\int_{s}^{t}\int_{x}\big(F_{1}[\sigma^{\prime}(u_{i})]^{2}\nabla_x u_{i}+\sigma(u_{i})\sigma^{\prime}(u_{i})F_{2}\big)\bar{\kappa}_{i,\tau}^{\varsigma,\delta}(x,y,\eta)\mathrm{d}\tau\bigg)\\
 & \quad+\partial_{\eta}\bigg(\int_{s}^{t}\int_{\mathbb{R}}\int_{\mathbb{T}^{d}}\kappa^{\varsigma,\delta}(x,y,\xi,\eta)\mathrm{d}q_{i}(x,\xi,\tau)\bigg)-\int_s^t\int_x\bar\kappa_{i,\tau}^{\varsigma,\delta}(x,y,\eta)\nabla_x\cdot g(u)\mathrm{d}\tau\\
 & \quad-\frac{1}{2}\partial_{\eta}\bigg(\int_{s}^{t}\int_{x}\Big(\sigma(u_{i})\sigma^{\prime}(u_{i})\nabla_x u_{i}\cdot F_{2}+\sigma^{2}(u_{i})F_{3}\Big)\bar{\kappa}_{i,\tau}^{\varsigma,\delta}(x,y,\eta)\mathrm{d}\tau\bigg)\\
 & \quad-\int_{s}^{t}\int_{\mathbb{T}^{d}}\kappa^{\varsigma,\delta}(x,y,\psi,\eta)\mathrm{d}\nu_{i}(x,\tau)-\int_{s}^{t}\int_{x}\bar{\kappa}_{i,\tau}^{\varsigma,\delta}(x,y,\eta)\nabla_x\cdot(\sigma(u_{i})f_{k})\mathrm{d}B_{\tau}^{k},
\end{align*}
where $u_{i}=u_{i}(x,\tau)$, $\psi=\psi(x,\tau)$ and $f_{k}=f_k(x)$. Then,
for every $\varsigma,\beta\in(0,1)$, $\delta\in(0,\beta/4)$, $i\in\{1,2\}$ and $N\in\mathbb{N}$, we have
\begin{align}\label{eq:first order term}
\int_{y,\eta}\chi_{i,t}^{\varsigma,\delta,+}(y,\eta)\varphi_{\beta}(\eta)\zeta_{N}(\eta)-\int_{y,\eta}\chi_{i,s}^{\varsigma,\delta}(y,\eta)\varphi_{\beta}(\eta)\zeta_{N}(\eta) =I_{i,s,t}^{\mathrm{obs}}+I_{i,s,t}^{\mathrm{cut}}+I_{i,s,t}^{\mathrm{mart}}+I_{i,s,t}^{\mathrm{cons}},
\end{align}
where the obstacle term is
\begin{align*}
I_{i,s,t}^{\mathrm{obs}} & :=-\int_{y,\eta}\varphi_{\beta}(\eta)\zeta_{N}(\eta)\int_{s}^{t}\int_{\mathbb{T}^{d}}\kappa^{\varsigma,\delta}(x,y,\psi,\eta)\mathrm{d}\nu_{i}(x,\tau),
\end{align*}
the cutoff term is
\begin{align*}
I_{i,s,t}^{\mathrm{cut}} & :=
-\int_{y,\eta}\partial_{\eta}\big(\varphi_{\beta}(\eta)\zeta_{N}(\eta)\big)\int_{s}^{t}\int_{\mathbb{R}}\int_{\mathbb{T}^{d}}\kappa^{\varsigma,\delta}(x,y,\xi,\eta)\mathrm{d}q_{i}(x,\xi,\tau)\\
 & \quad+\frac{1}{2}\int_{y,\eta,x}\partial_{\eta}\big(\varphi_{\beta}(\eta)\zeta_{N}(\eta)\big)\int_{s}^{t}\Big(\sigma(u_{i})\sigma^{\prime}(u_{i})\nabla_x u_{i}\cdot F_{2}(x)+\sigma^{2}(u_{i})F_{3}(x)\Big)\bar{\kappa}_{i,\tau}^{\varsigma,\delta}(x,y,\eta)\mathrm{d}\tau,
\end{align*}
the martingale term is
\begin{align*}
I_{i,s,t}^{\mathrm{mart}} & :=-\int_{y,\eta,x}\varphi_{\beta}(\eta)\zeta_{N}(\eta)\int_{s}^{t}\bar{\kappa}_{i,\tau}^{\varsigma,\delta}(x,y,\eta)\nabla\cdot(\sigma(u_{i})f_{k})\mathrm{d}B_{\tau}^{k},
\end{align*}
and the conservative term is
\begin{align*}
I_{i,s,t}^{\mathrm{cons}} & :=-\int_{y,\eta}\varphi_\beta(\eta)\zeta_N(\eta)\int_s^t\int_x\bar\kappa_{i,\tau}^{\varsigma,\delta}(x,y,\eta)\nabla_x\cdot g(u)\mathrm{d}\tau.
\end{align*}
In the following, we write
 $u_{1}=u_{1}(x,\tau)$, $u_{2}=u_{2}(z,\tau)$,
$\bar{\kappa}_{1,\tau}^{\varsigma,\delta}=\bar{\kappa}_{1,\tau}^{\varsigma,\delta}(x,y,\eta)=\kappa^{\varsigma,\delta}(x,y,u_{1}(x,\tau),\eta)$,
and $\bar{\kappa}_{2,\tau}^{\varsigma,\delta}=\bar{\kappa}_{2,\tau}^{\varsigma,\delta}(z,y,\eta)=\kappa^{\varsigma,\delta}(z,y,u_{2}(z,\tau),\eta)$
for simplicity.

Using the It\^{o}'s formula, by Lemmas \ref{lem:intergrating by part}, \eqref{eq:partial eta} and integration by parts for two functions of finite variations from \cite[Chapter 0, Proposition 4.5]{revuz1999continuous}, we have almost surely
\begin{align}
 & \int_{y,\eta}\chi_{1,t}^{\varsigma,\delta,+}(y,\eta)\chi_{2,t}^{\varsigma,\delta,+}(y,\eta)\varphi_{\beta}(\eta)\zeta_{N}(\eta)-\int_{y,\eta}\chi_{1,s}^{\varsigma,\delta}(y,\eta)\chi_{2,s}^{\varsigma,\delta}(y,\eta)\varphi_{\beta}(\eta)\zeta_{N}(\eta) \label{eq:second order term}\\
 & =\int_{s}^{t}\int_{y,\eta}\Big[\chi_{2,\tau}^{\varsigma,\delta,+}(y,\eta)\mathrm{d}\chi_{1,\tau}^{\varsigma,\delta,+}(y,\eta)+\chi_{1,\tau}^{\varsigma,\delta,-}(y,\eta)\mathrm{d}\chi_{2,\tau}^{\varsigma,\delta,+}(y,\eta)+\mathrm{d}\langle\chi_{2,\cdot}^{\varsigma,\delta,+},\chi_{1,\cdot}^{\varsigma,\delta,+}\rangle_{\tau}(y,\eta)\Big]\varphi_{\beta}(\eta)\zeta_{N}(\eta).\nonumber
\end{align}
Here, $\mathrm{d}\langle\chi_{2,\cdot}^{\varsigma,\delta,+},\chi_{1,\cdot}^{\varsigma,\delta,+}\rangle_{\tau}(y,\eta)$ denotes the predictable quadratic covariation of the martingale parts of $\chi_{2}^{\varsigma,\delta,+}$ and $\chi_{1}^{\varsigma,\delta,+}$.
In particular, although the measures $q$ and $\nu$ may produce time discontinuities, the bracket term above refers only to the It\^{o} correction coming from the stochastic integral parts.

For every $\varsigma,\beta\in(0,1)$, $\delta\in(0,\beta/4)$,
and $N\in\mathbb{N}$, using the fact that $\chi_{2,\tau}^{\varsigma,\delta,+}=\chi_{2,\tau}^{\varsigma,\delta}$ almost everywhere, we have
\[
\int_{s}^{t}\int_{y,\eta}\chi_{2,\tau}^{\varsigma,\delta,+}(y,\eta)\mathrm{d}\chi_{1,\tau}^{\varsigma,\delta,+}(y,\eta)\varphi_{\beta}(\eta)\zeta_{N}(\eta)=I_{2,1,s,t}^{\mathrm{obs}}+I_{2,1,s,t}^{\mathrm{err}}+I_{2,1,s,t}^{\mathrm{meas}}+I_{2,1,s,t}^{\mathrm{cut}}+I_{2,1,s,t}^{\mathrm{mart}}+I_{2,1,s,t}^{\mathrm{cons}},
\]
where the obstacle term is
\begin{align*}
I_{2,1,s,t}^{\mathrm{obs}} & :=
-\int_{y,\eta}\int_{s}^{t}\int_{\mathbb{T}^{d}}\chi_{2,\tau}^{\varsigma,\delta,+}(y,\eta)\varphi_{\beta}(\eta)\zeta_{N}(\eta)\kappa^{\varsigma,\delta}(x,y,\psi,\eta)\mathrm{d}\nu_{1}(x,\tau),
\end{align*}
the error term is
\begin{align*}
I_{2,1,s,t}^{\mathrm{err}} & :=-\int_{y,\eta,x,z}\int_{s}^{t}\Phi^{\prime}(u_{1})\nabla_{x}u_{1}\cdot\nabla_{z}u_{2}\bar{\kappa}_{2,\tau}^{\varsigma,\delta}\bar{\kappa}_{1,\tau}^{\varsigma,\delta}\varphi_{\beta}(\eta)\zeta_{N}(\eta)\mathrm{d}\tau\\
 & \quad+\int_{y,\eta,x,z}\int_{s}^{t}|\Phi^{\prime}(u_{1})|^{\frac{1}{2}}|\Phi^{\prime}(u_{2})|^{\frac{1}{2}}\nabla_{x}u_{1}\cdot\nabla_{z}u_{2}\bar{\kappa}_{2,\tau}^{\varsigma,\delta}\bar{\kappa}_{1,\tau}^{\varsigma,\delta}\varphi_{\beta}(\eta)\zeta_{N}(\eta)\mathrm{d}\tau\\
 & \quad-\frac{1}{2}\int_{y,\eta,x,z}\int_{s}^{t}\nabla_{z}u_{2}\cdot\big(F_{1}(x)[\sigma^{\prime}(u_{1})]^{2}\nabla_{x}u_{1}+\sigma(u_{1})\sigma^{\prime}(u_{1})F_{2}(x)\big)\bar{\kappa}_{1,\tau}^{\varsigma,\delta}\bar{\kappa}_{2,\tau}^{\varsigma,\delta}\varphi_{\beta}(\eta)\zeta_{N}(\eta)\mathrm{d}\tau\\
 & \quad-\frac{1}{2}\int_{y,\eta,x,z}\int_{s}^{t}\Big(\sigma(u_{1})\sigma^{\prime}(u_{1})\nabla_{x}u_{1}\cdot F_{2}(x)+\sigma^{2}(u_{1})F_{3}(x)\Big)\bar{\kappa}_{1,\tau}^{\varsigma,\delta}\bar{\kappa}_{2,\tau}^{\varsigma,\delta}\varphi_{\beta}(\eta)\zeta_{N}(\eta)\mathrm{d}\tau,
\end{align*}
the measure term is
\begin{align*}
I_{2,1,s,t}^{\mathrm{meas}} & :=-\int_{y,\eta}\int_{s}^{t}\partial_\eta\chi_{2,\tau}^{\varsigma,\delta,+}(y,\eta)\varphi_{\beta}(\eta)\zeta_{N}(\eta)\int_{\mathbb{R}}\int_{\mathbb{T}^{d}}\kappa^{\varsigma,\delta}(x,y,\xi,\eta)\mathrm{d}q_{1}(x,\xi,\tau)\\
 & \quad-\int_{y,\eta,x,z}\int_{s}^{t}|\Phi^{\prime}(u_{1})|^{\frac{1}{2}}|\Phi^{\prime}(u_{2})|^{\frac{1}{2}}\nabla_{x}u_{1}\cdot\nabla_{z}u_{2}\bar{\kappa}_{2,\tau}^{\varsigma,\delta}\bar{\kappa}_{1,\tau}^{\varsigma,\delta}\varphi_{\beta}(\eta)\zeta_{N}(\eta)\mathrm{d}\tau,
\end{align*}
the cutoff term is
\begin{align*}
I_{2,1,s,t}^{\mathrm{cut}} & :=-\int_{y,\eta}\int_{s}^{t}\int_{\mathbb{R}}\int_{\mathbb{T}^{d}}\chi_{2,\tau}^{\varsigma,\delta,+}(y,\eta)\partial_{\eta}\big(\varphi_{\beta}(\eta)\zeta_{N}(\eta)\big)\kappa^{\varsigma,\delta}(x,y,\xi,\eta)\mathrm{d}q_{1}(x,\xi,\tau)\\
 & \quad+\frac{1}{2}\int_{y,\eta,x}\int_{s}^{t}\chi_{2,\tau}^{\varsigma,\delta}(y,\eta)\partial_{\eta}\big(\varphi_{\beta}(\eta)\zeta_{N}(\eta)\big)\Big(\sigma(u_{1})\sigma^{\prime}(u_{1})\nabla_{x}u_{1}\cdot F_{2}(x)+\sigma^{2}(u_{1})F_{3}(x)\Big)\bar{\kappa}_{1,\tau}^{\varsigma,\delta}\mathrm{d}\tau,
\end{align*}
the martingale term is
\[
I_{2,1,s,t}^{\mathrm{mart}}:=-\int_{y,\eta,x}\int_{s}^{t}\chi_{2,\tau}^{\varsigma,\delta}(y,\eta)\varphi_{\beta}(\eta)\zeta_{N}(\eta)\bar{\kappa}_{1,\tau}^{\varsigma,\delta}\nabla_{x}\cdot(\sigma(u_{1})f_{k}(x))\mathrm{d}B_{\tau}^{k},
\]
and the conservative term is
\[
I_{2,1,s,t}^{\mathrm{cons}}:=-\int_{y,\eta,x}\int_s^t\varphi_\beta(\eta)\zeta_N(\eta)\chi^{\varsigma,\delta}_{2,\tau}(y,\eta)\bar\kappa_{1,\tau}^{\varsigma,\delta}\nabla_x\cdot g(u)\mathrm{d}\tau.
\]
An analogous formula holds for the second term in \eqref{eq:second order term}, which means
\[
\int_{s}^{t}\int_{y,\eta}\chi_{1,\tau}^{\varsigma,\delta,-}(y,\eta)\mathrm{d}\chi_{2,\tau}^{\varsigma,\delta,+}(y,\eta)\varphi_{\beta}(\eta)\zeta_{N}(\eta)=I_{1,2,s,t}^{\mathrm{obs}}+I_{1,2,s,t}^{\mathrm{err}}+I_{1,2,s,t}^{\mathrm{meas}}+I_{1,2,s,t}^{\mathrm{cut}}+I_{1,2,s,t}^{\mathrm{mart}}+I_{1,2,s,t}^{\mathrm{cons}},
\]
where the integrals involving the measures $\nu_2$ and $q_2$  contain $\chi^{\varsigma,\delta,-}_{1,\tau}$ instead of $\chi^{\varsigma,\delta,+}_{1,\tau}$.

For the cross term in \eqref{eq:second order term}, using the fact that $\chi_{i,\tau}^{\varsigma,\delta,+}=\chi_{i,\tau}^{\varsigma,\delta}$ almost everywhere, $i\in\{1,2\}$, we have
\begin{align}
 & \int_{s}^{t}\int_{y,\eta}\mathrm{d}\langle\chi_{2,\cdot}^{\varsigma,\delta,+},\chi_{1,\cdot}^{\varsigma,\delta,+}\rangle_{\tau}(y,\eta)\varphi_{\beta}(\eta)\zeta_{N}(\eta)\label{eq:corss term} \\
 & =\int_{y,\eta,x,z}\int_{s}^{t}f_{k}(x)f_{k}(z)\sigma^{\prime}(u_{1})\sigma^{\prime}(u_{2})\nabla_{x}u_{1}\cdot\nabla_{z}u_{2}\varphi_{\beta}(\eta)\zeta_{N}(\eta)\bar{\kappa}_{1,\tau}^{\varsigma,\delta}\bar{\kappa}_{2,\tau}^{\varsigma,\delta}\mathrm{d}\tau\nonumber \\
 & \quad+\int_{y,\eta,x,z}\int_{s}^{t}f_{k}(x)\nabla_{z}f_{k}(z)\cdot\nabla_{x}u_{1}\sigma^{\prime}(u_{1})\sigma(u_{2})\varphi_{\beta}(\eta)\zeta_{N}(\eta)\bar{\kappa}_{1,\tau}^{\varsigma,\delta}\bar{\kappa}_{2,\tau}^{\varsigma,\delta}\mathrm{d}\tau\nonumber\\
 & \quad+\int_{y,\eta,x,z}\int_{s}^{t}f_{k}(z)\nabla_{x}f_{k}(x)\cdot\nabla_{z}u_{2}\sigma(u_{1})\sigma^{\prime}(u_{2})\varphi_{\beta}(\eta)\zeta_{N}(\eta)\bar{\kappa}_{1,\tau}^{\varsigma,\delta}\bar{\kappa}_{2,\tau}^{\varsigma,\delta}\mathrm{d}\tau\nonumber \\
 & \quad+\int_{y,\eta,x,z}\int_{s}^{t}\nabla_{x}f_{k}(x)\cdot\nabla_{z}f_{k}(z)\sigma(u_{1})\sigma(u_{2})\varphi_{\beta}(\eta)\zeta_{N}(\eta)\bar{\kappa}_{1,\tau}^{\varsigma,\delta}\bar{\kappa}_{2,\tau}^{\varsigma,\delta}\mathrm{d}\tau.\nonumber
\end{align}
Therefore, we have
\begin{align}
&\int_{y,\eta}\chi_{1,t}^{\varsigma,\delta,+}(y,\eta)\chi_{2,t}^{\varsigma,\delta,+}(y,\eta)\varphi_{\beta}(\eta)\zeta_{N}(\eta)-\int_{y,\eta}\chi_{1,s}^{\varsigma,\delta}(y,\eta)\chi_{2,s}^{\varsigma,\delta}(y,\eta)\varphi_{\beta}(\eta)\zeta_{N}(\eta)\label{eq:second term all}\\
&=I_{\mathrm{mix},s,t}^{\mathrm{obs}}+I_{\mathrm{mix},s,t}^{\mathrm{err}}+I_{\mathrm{mix},s,t}^{\mathrm{meas}}+I_{\mathrm{mix},s,t}^{\mathrm{cut}}+I_{\mathrm{mix},s,t}^{\mathrm{mart}}+I_{\mathrm{mix},s,t}^{\mathrm{cons}},\nonumber
\end{align}
where
\[
I_{\mathrm{mix},s,t}^{\mathrm{fun}}=I_{2,1,s,t}^{\mathrm{fun}}+I_{1,2,s,t}^{\mathrm{fun}},
\]
with ``fun'' taking values in $\{\mathrm{obs,meas,cut,mart,cons}\}$. For the error terms, which combine $I_{2,1,s,t}^{\mathrm{err}}$, $I_{1,2,s,t}^{\mathrm{err}}$
and \eqref{eq:corss term}, are
\begin{align*}
 & I_{\mathrm{mix},s,t}^{\mathrm{err}}\\
 & :=-\int_{y,\eta,x,z}\int_{s}^{t}\big[|\Phi^{\prime}(u_{1})|^{\frac{1}{2}}-|\Phi^{\prime}(u_{2})|^{\frac{1}{2}}\big]^{2}\nabla_{x}u_{1}\cdot\nabla_{z}u_{2}\bar{\kappa}_{2,\tau}^{\varsigma,\delta}\bar{\kappa}_{1,\tau}^{\varsigma,\delta}\varphi_{\beta}\zeta_{N}\mathrm{d}\tau\\
 & -\frac{1}{2}\int_{y,\eta,x,z}\int_{s}^{t}\nabla_{z}u_{2}\cdot\nabla_{x}u_{1}\big(F_{1}(x)[\sigma^{\prime}(u_{1})]^{2}+F_{1}(z)[\sigma^{\prime}(u_{2})]^{2}-2f_{k}(x)f_{k}(z)\sigma^{\prime}(u_{1})\sigma^{\prime}(u_{2})\big)\bar{\kappa}_{1,\tau}^{\varsigma,\delta}\bar{\kappa}_{2,\tau}^{\varsigma,\delta}\varphi_{\beta}\zeta_{N}\mathrm{d}\tau\\
 & -\frac{1}{2}\int_{y,\eta,x,z}\int_{s}^{t}\nabla_{z}u_{2}\cdot\big(F_{2}(x)\sigma(u_{1})\sigma^{\prime}(u_{1})+F_{2}(z)\sigma(u_{2})\sigma^{\prime}(u_{2})-2f_{k}(z)\nabla_{x}f_{k}(x)\sigma(u_{1})\sigma^{\prime}(u_{2})\big)\bar{\kappa}_{1,\tau}^{\varsigma,\delta}\bar{\kappa}_{2,\tau}^{\varsigma,\delta}\varphi_{\beta}\zeta_{N}\mathrm{d}\tau\\
 & -\frac{1}{2}\int_{y,\eta,x,z}\int_{s}^{t}\nabla_{x}u_{1}\cdot\big(F_{2}(x)\sigma(u_{1})\sigma^{\prime}(u_{1})+F_{2}(z)\sigma(u_{2})\sigma^{\prime}(u_{2})-2f_{k}(x)\nabla_{z}f_{k}(z)\sigma^{\prime}(u_{1})\sigma(u_{2})\big)\bar{\kappa}_{1,\tau}^{\varsigma,\delta}\bar{\kappa}_{2,\tau}^{\varsigma,\delta}\varphi_{\beta}\zeta_{N}\mathrm{d}\tau\\
 & -\frac{1}{2}\int_{y,\eta,x,z}\int_{s}^{t}\big(F_{3}(x)\sigma^{2}(u_{1})+F_{3}(z)\sigma^{2}(u_{2})-2\nabla_{x}f_{k}(x)\cdot\nabla_{z}f_{k}(z)\sigma(u_{1})\sigma(u_{2})\big)\bar{\kappa}_{1,\tau}^{\varsigma,\delta}\bar{\kappa}_{2,\tau}^{\varsigma,\delta}\varphi_{\beta}\zeta_{N}\mathrm{d}\tau.
\end{align*}
Combining \eqref{eq:first order term} and \eqref{eq:second term all}, we have
\begin{align*}
&\int_{y,\eta}\big(\chi_{1,t}^{\varsigma,\delta,+}+\chi_{2,t}^{\varsigma,\delta,+}-2\chi_{1,t}^{\varsigma,\delta,+}\chi_{2,t}^{\varsigma,\delta,+}\big)\varphi_{\beta}\zeta_{N}-\int_{y,\eta}\big(\chi_{1,s}^{\varsigma,\delta}+\chi_{2,s}^{\varsigma,\delta}-2\chi_{1,s}^{\varsigma,\delta}\chi_{2,s}^{\varsigma,\delta}\big)\varphi_{\beta}\zeta_{N}\\
 & =I_{s,t}^{\mathrm{obs}}+I_{s,t}^{\mathrm{err}}+I_{s,t}^{\mathrm{meas}}+I_{s,t}^{\mathrm{cut}}+I_{s,t}^{\mathrm{mart}}+I_{s,t}^{\mathrm{cons}},
\end{align*}
where $I_{s,t}^{\mathrm{fun}}:=I_{1,s,t}^{\mathrm{fun}}+I_{2,s,t}^{\mathrm{fun}}-2(I_{\mathrm{mix},s,t}^{\mathrm{fun}})$ with ``fun'' taking values in $\{\mathrm{obs,err,meas,cut,mart,cons}\}$, where $I_{i,s,t}^{\mathrm{err}},I_{i,s,t}^{\mathrm{meas}}:=0$ for $i\in\{1,2\}$.

The terms $I_{s,t}^{\mathrm{err}}$, $I_{s,t}^{\mathrm{cut}}$, $I_{s,t}^{\mathrm{mart}}$, and $I_{s,t}^{\mathrm{cons}}$ do not involve integration with respect to the measures $q_i$ or $\nu_i$. Hence, for these terms, the choice of the right- or left-continuous representatives is immaterial, and we use $\chi$ instead of $\chi^{+}$ or $\chi^{-}$.
Consequently, their estimates can be carried out exactly as in the proof of \cite[Theorem 4.6]{fehrman2024well}.
We thus omit the details and only list the results:
\begin{align*} \limsup_{\delta\to0}\limsup_{\varsigma\to0}|I_{s,t}^{\mathrm{err}}|=0,\quad \lim_{N\to\infty}\lim_{\beta\to0}\lim_{\delta\to0}\lim_{\varsigma\to0}I_{s,t}^{\mathrm{cut}}=0, \quad \\\lim_{N\to\infty}\lim_{\beta\to0}\lim_{\delta\to0}\lim_{\varsigma\to0}I_{s,t}^{\mathrm{mart}}=0,\quad
\lim_{N\to\infty}\lim_{\beta\to0}\lim_{\delta\to0}\lim_{\varsigma\to0}I_{s,t}^{\mathrm{mart}}=0.
\end{align*}

In the following, we will handle the remaining terms $I_{s,t}^{\mathrm{obs}}$ and $I_{s,t}^{\mathrm{meas}}$,
which involve the measure $\nu_{i}$ associated to the obstacle problem and the kinetic measure $q_i$  for $i\in\{1,2\}$.
Let us start with the obstacle term $I_{s,t}^{\mathrm{obs}}$ given by
\begin{align}
I_{s,t}^{\mathrm{obs}}&=
-\int_{y,\eta}\varphi_{\beta}(\eta)\zeta_{N}(\eta)\int_{s}^{t}\int_{\mathbb{T}^{d}}\kappa^{\varsigma,\delta}(x,y,\psi(x,\tau),\eta)\mathrm{d}\nu_{1}(x,\tau)  \label{eq:obstacle terms estimates}\\ \notag
 & \quad-\int_{y,\eta}\varphi_{\beta}(\eta)\zeta_{N}(\eta)\int_{s}^{t}\int_{\mathbb{T}^{d}}\kappa^{\varsigma,\delta}(z,y,\psi(z,\tau),\eta)\mathrm{d}\nu_{2}(z,\tau) \\
 & \quad+2\int_{y,\eta}\int_{s}^{t}\int_{\mathbb{T}^{d}}\chi_{2,\tau}^{\varsigma,\delta,+}(y,\eta)\varphi_{\beta}(\eta)\zeta_{N}(\eta)\kappa^{\varsigma,\delta}(x,y,\psi(x,\tau),\eta)\mathrm{d}\nu_{1}(x,\tau)\nonumber\\ \notag
 & \quad+2\int_{y,\eta}\int_{s}^{t}\int_{\mathbb{T}^{d}}\chi_{1,\tau}^{\varsigma,\delta,-}(y,\eta)\varphi_{\beta}(\eta)\zeta_{N}(\eta)\kappa^{\varsigma,\delta}(z,y,\psi(z,\tau),\eta)\mathrm{d}\nu_{2}(z,\tau).
\end{align}
To deal with the last two terms on the righthand side of \eqref{eq:obstacle terms estimates}, we define
\begin{equation}\label{eq:definition for chi psi}
\chi_\psi(x,\xi,\tau):=\mathbf{1}_{\{0<\xi<\psi(x,\tau)\}},\quad \chi_{\psi,\tau}^{\varsigma,\delta}(y,\eta):=\int_{x,\xi}\chi_\psi(x,\xi,\tau)\kappa^{\varsigma,\delta}(x,y,\xi,\eta).
\end{equation}
From Definition \ref{def:def of stochastic kinetic solution} (i), we have $\chi_{i}\leq\chi_\psi$ on $\mathbb{T}^d\times(0,\infty)\times[0,T)$ for $i\in\{1,2\}$.
Fix $i\in\{1,2\}$.
Note that almost surely $\chi^{+}_{i}=\chi^{-}_{i}=\chi_{i}$ almost everywhere on $[0,T]$ as a distribution in $\mathbb{T}^d\times\mathbb{R}$.
Then, for all $\tau\in(0,T)$, there exist sequences $\{\tau_k^{-}\}_{k\in{\mathbb{N}}}$ and $\{\tau_k^{+}\}_{k\in{\mathbb{N}}}$ of positive numbers satisfying $\tau_k^{-}<\tau$, $\tau_k^{+}<T-\tau$ for every $k\in\mathbb{N}$, and $\tau_k^{-},\tau_k^{+}\to0$ as $k\to\infty$ such that $\chi^{\pm}_{i}(x,\xi,\tau\pm \tau_k^{\pm})=\chi_{i}(x,\xi,\tau\pm \tau_k^{\pm})$ for every $k\in\mathbb{N}$.
Owing to the time-continuity of $\psi$ and the definition of $\chi_i^{+}$, by the Fatou lemma, we obtain almost surely for all nonnegative $\phi\in C_c^\infty(\mathbb{T}^d\times (0,\infty))$,
\begin{align*}
\int_{x,\xi}(\chi_{\psi}(x,\xi,\tau)-\chi_{i}^{\pm}(x,\xi,\tau))\phi(x,\xi)&\geq \liminf_{k\to\infty}\int_{x,\xi}(\chi_{\psi}(x,\xi,\tau\pm \tau_k^{\pm})-\chi_{i}(x,\xi,\tau\pm \tau_k^{\pm}))\phi(x,\xi)\geq 0,
\end{align*}
By taking  $\phi(x,\xi)=\kappa^{\varsigma,\delta}(x,\xi,y,\eta)$, we have almost surely
\begin{equation}\label{eq:chi leq chi psi}
\chi_{i,\tau}^{\varsigma,\delta,\pm}(y,\eta)\leq\chi_{\psi,\tau}^{\varsigma,\delta}(y,\eta),\quad\text{for }(y,\eta,\tau)\in\mathbb{T}^d\times(0,\infty)\times(0,T).
\end{equation}
We further define
\begin{align}\label{eq:initegrable}
\llbracket f\rrbracket(r):=\int_{0}^{r}f(\tilde{r})\mathrm{d}\tilde{r},\quad\text{for each }f\in L^{1}(\mathbb{R}).
\end{align}
Then, we have
\begin{equation}\label{rr-2}
\kappa_{1}^{\delta}(\xi-\eta)=\partial_{\xi}\llbracket\kappa_{1}^{\delta}\rrbracket(\xi-\eta).
\end{equation}
With the aid of \eqref{eq:chi leq chi psi} and \eqref{rr-2}, when $\delta$ is small enough, the third
term on the righthand side of \eqref{eq:obstacle terms estimates} can be estimated as
\begin{align}
 & \int_{y,\eta}\int_{s}^{t}\int_{\mathbb{T}^{d}}\chi_{2,\tau}^{\varsigma,\delta,+}(y,\eta)\varphi_{\beta}(\eta)\zeta_{N}(\eta)\kappa^{\varsigma,\delta}(x,y,\psi,\eta)\mathrm{d}\nu_{1}(x,\tau)\label{eq:second order obstacle term estimates} \\
 &\leq \int_{y,\eta}\int_{s}^{t}\int_{\mathbb{T}^{d}}\chi_{\psi,\tau}^{\varsigma,\delta}(y,\eta)\varphi_{\beta}(\eta)\zeta_{N}(\eta)\kappa^{\varsigma,\delta}(x,y,\psi,\eta)\mathrm{d}\nu_{1}(x,\tau)\nonumber \\
 & =\int_{y,\eta,z,\xi}\int_{s}^{t}\int_{\mathbb{T}^{d}}\varphi_{\beta}(\eta)\zeta_{N}(\eta)\chi_{\psi}(z,\xi,\tau)\kappa_{d}^{\varsigma}(z-y)\kappa_{1}^{\delta}(\xi-\eta)\kappa_{d}^{\varsigma}(x-y)\kappa_{1}^{\delta}(\psi(x,\tau)-\eta)\mathrm{d}\nu_{1}(x,\tau)\nonumber \\
 & =\int_{y,\eta,z,\xi}\int_{s}^{t}\int_{\mathbb{T}^{d}}\varphi_{\beta}(\eta)\zeta_{N}(\eta)\chi_{\psi}(z,\xi,\tau)\kappa_{d}^{\varsigma}(z-y)\partial_{\xi}\llbracket\kappa_{1}^{\delta}\rrbracket(\xi-\eta)\kappa_{d}^{\varsigma}(x-y)\kappa_{1}^{\delta}(\psi(x,\tau)-\eta)\mathrm{d}\nu_{1}(x,\tau)\nonumber \\
 & =-\int_{y,\eta,z,\xi}\int_{s}^{t}\int_{\mathbb{T}^{d}}\varphi_{\beta}(\eta)\zeta_{N}(\eta)\partial_{\xi}\chi_{\psi}(z,\xi,\tau)\kappa_{d}^{\varsigma}(z-y)\llbracket\kappa_{1}^{\delta}\rrbracket(\xi-\eta)\kappa_{d}^{\varsigma}(x-y)\kappa_{1}^{\delta}(\psi(x,\tau)-\eta)\mathrm{d}\nu_{1}(x,\tau)\nonumber\\
 & =-\int_{y,\eta}\int_{s}^{t}\int_{\mathbb{T}^{d}}\varphi_{\beta}(\eta)\zeta_{N}(\eta)\llbracket\kappa_{1}^{\delta}\rrbracket(-\eta)\kappa_{d}^{\varsigma}(x-y)\kappa_{1}^{\delta}(\psi(x,\tau)-\eta)\mathrm{d}\nu_{1}(x,\tau)\nonumber\\
 & \quad+\int_{y,\eta,z}\int_{s}^{t}\int_{\mathbb{T}^{d}}\varphi_{\beta}(\eta)\zeta_{N}(\eta)\kappa_{d}^{\varsigma}(z-y)\llbracket\kappa_{1}^{\delta}\rrbracket(\psi(z,\tau)-\eta)\kappa_{d}^{\varsigma}(x-y)\kappa_{1}^{\delta}(\psi(x,\tau)-\eta)\mathrm{d}\nu_{1}(x,\tau)\nonumber \\
 & =\frac{1}{2}\int_{\eta}\int_{s}^{t}\int_{\mathbb{T}^{d}}\varphi_{\beta}(\eta)\zeta_{N}(\eta)\kappa_{1}^{\delta}(\psi(x,\tau)-\eta)\mathrm{d}\nu_{1}(x,\tau)\nonumber \\
 & \quad+\int_{y,\eta,z}\int_{s}^{t}\int_{\mathbb{T}^{d}}\varphi_{\beta}(\eta)\zeta_{N}(\eta)\kappa_{d}^{\varsigma}(z-y)\llbracket\kappa_{1}^{\delta}\rrbracket(\psi(z,\tau)-\eta)\kappa_{d}^{\varsigma}(x-y)\kappa_{1}^{\delta}(\psi(x,\tau)-\eta)\mathrm{d}\nu_{1}(x,\tau),\nonumber
\end{align}
where the last inequality is deduced from the fact that
\[
\llbracket\kappa_{1}^{\delta}\rrbracket(r)=\begin{cases}
{1}/{2}, & \text{if }r\geq{\delta}/{2};\\
-{1}/{2}, & \text{if }r\leq-{\delta}/{2}.
\end{cases}
\]
Furthermore, it follows from the support of $\kappa_{1}^{\delta}$ and the continuity of $\varphi_{\beta}$ and $\zeta_{N}$ that, for fixed $\beta\in(0,1)$ and $N\in\mathbb{N}$, we have
\begin{align*}
 & \int_{y,\eta,z}\int_{s}^{t}\int_{\mathbb{T}^{d}}\bigg[\kappa_{d}^{\varsigma}(z-y)\llbracket\kappa_{1}^{\delta}\rrbracket(\psi(z,\tau)-\eta)\kappa_{d}^{\varsigma}(x-y)\kappa_{1}^{\delta}(\psi(x,\tau)-\eta)\\
 & \quad\cdot\big|\varphi_{\beta}(\eta)\zeta_{N}(\eta)-\varphi_{\beta}(\psi(x,\tau))\zeta_{N}(\psi(x,\tau))\big|\bigg]\mathrm{d}\nu_{1}(x,\tau)\\
 & \leq C(\beta)\delta\nu_{1}(Q_{T}).
\end{align*}
Therefore, letting $\tilde{\eta}=\eta-\psi(x,\tau)$ for each $(x,\tau)\in Q_{T}$, we have
\begin{align*}
 & \int_{y,\eta,z}\int_{s}^{t}\int_{\mathbb{T}^{d}}\varphi_{\beta}(\eta)\zeta_{N}(\eta)\kappa_{d}^{\varsigma}(z-y)\llbracket\kappa_{1}^{\delta}\rrbracket(\psi(z,\tau)-\eta)\kappa_{d}^{\varsigma}(x-y)\kappa_{1}^{\delta}(\psi(x,\tau)-\eta)\mathrm{d}\nu_{1}(x,\tau)\\
 & \leq\int_{y,\eta,z}\int_{s}^{t}\int_{\mathbb{T}^{d}}\varphi_{\beta}(\psi(x,\tau))\zeta_{N}(\psi(x,\tau))\kappa_{d}^{\varsigma}(z-y)\llbracket\kappa_{1}^{\delta}\rrbracket(\psi(z,\tau)-\eta)\kappa_{d}^{\varsigma}(x-y)\kappa_{1}^{\delta}(\psi(x,\tau)-\eta)\mathrm{d}\nu_{1}(x,\tau)\\
 & \quad+C(\beta)\delta\nu_{1}(Q_{T})\\
 & =\int_{y,\tilde{\eta},z}\int_{s}^{t}\int_{\mathbb{T}^{d}}\bigg[\varphi_{\beta}(\psi(x,\tau))\zeta_{N}(\psi(x,\tau))\kappa_{d}^{\varsigma}(z-y)\llbracket\kappa_{1}^{\delta}\rrbracket(\psi(z,\tau)-\psi(x,\tau)-\tilde{\eta})\\
 & \quad\quad\cdot\kappa_{d}^{\varsigma}(x-y)\kappa_{1}^{\delta}(-\tilde{\eta})\bigg]\mathrm{d}\nu_{1}(x,\tau)+C(\beta)\delta\nu_{1}(Q_{T}).
\end{align*}
Then, taking the limit supremum $\varsigma\to0$ on \eqref{eq:second order obstacle term estimates}, by the time-continuity property of $\psi$, we have
\begin{align*}
 & \limsup_{\varsigma\to0}\Bigg(\int_{y,\eta}\int_{s}^{t}\int_{\mathbb{T}^{d}}\chi_{2,\tau}^{\varsigma,\delta,+}(y,\eta)\varphi_{\beta}(\eta)\zeta_{N}(\eta)\kappa^{\varsigma,\delta}(x,y,\psi,\eta)\mathrm{d}\nu_{1}(x,\tau)\Bigg)\\
 & \leq\frac{1}{2}\int_{\eta}\int_{s}^{t}\int_{\mathbb{T}^{d}}\varphi_{\beta}(\eta)\zeta_{N}(\eta)\kappa_{1}^{\delta}(\psi(x,\tau)-\eta)\mathrm{d}\nu_{1}(x,\tau)\\
 & \quad+\int_{\tilde{\eta}}\llbracket\kappa_{1}^{\delta}\rrbracket(-\tilde{\eta})\kappa_{1}^{\delta}(-\tilde{\eta})\cdot\int_{s}^{t}\int_{\mathbb{T}^{d}}\varphi_{\beta}(\psi(x,\tau))\zeta_{N}(\psi(x,\tau))\mathrm{d}\nu_{1}(x,\tau)+C(\beta)\delta\nu_{1}(Q_{T}).
\end{align*}
Since $\kappa_{1}^{\delta}$ is an even function and $\llbracket\kappa_{1}^{\delta}\rrbracket$ is an odd function, it follows that
\[\int_{\tilde{\eta}}\llbracket\kappa_{1}^{\delta}\rrbracket(-\tilde{\eta})\kappa_{1}^{\delta}(-\tilde{\eta})=0.\]
Consequently, we reach
\begin{align}
 & \limsup_{\varsigma\to0}\Bigg(\int_{y,\eta}\int_{s}^{t}\int_{\mathbb{T}^{d}}\chi_{2,\tau}^{\varsigma,\delta,+}(y,\eta)\varphi_{\beta}(\eta)\zeta_{N}(\eta)\kappa^{\varsigma,\delta}(x,y,\psi,\eta)\mathrm{d}\nu_{1}(x,\tau)\Bigg)\label{rr-3}\\
 \notag
 & \leq\frac{1}{2}\int_{\eta}\int_{s}^{t}\int_{\mathbb{T}^{d}}\varphi_{\beta}(\eta)\zeta_{N}(\eta)\kappa_{1}^{\delta}(\psi(x,\tau)-\eta)\mathrm{d}\nu_{1}(x,\tau)+C(\beta)\delta\nu_{1}(Q_{T}).
\end{align}
Proceed similarly as above, the last term on the righthand side of \eqref{eq:obstacle terms estimates} can be estimated as
\begin{align}\label{rr-4}
 & \limsup_{\varsigma\to0}\Bigg(\int_{y,\eta}\int_{s}^{t}\int_{\mathbb{T}^{d}}\chi_{1,\tau}^{\varsigma,\delta,-}(y,\eta)\varphi_{\beta}(\eta)\zeta_{N}(\eta)\kappa^{\varsigma,\delta}(z,y,\psi,\eta)\mathrm{d}\nu_{2}(z,\tau)\Bigg)\\
 \notag
 & \leq\frac{1}{2}\int_{\eta}\int_{s}^{t}\int_{\mathbb{T}^{d}}\varphi_{\beta}(\eta)\zeta_{N}(\eta)\kappa_{1}^{\delta}(\psi(z,\tau)-\eta)\mathrm{d}\nu_{2}(z,\tau)+C(\beta)\delta\nu_{2}(Q_{T}).
\end{align}
Combining \eqref{rr-3} with \eqref{rr-4}, and letting $\varsigma\to0$ on \eqref{eq:obstacle terms estimates}, we get
\[
\limsup_{\varsigma\to0}I_{s,t}^{\mathrm{obs}}\leq 2C(\beta)\delta[\nu_{1}(Q_{T})+\nu_{2}(Q_{T})],
\]
and then, it yields almost surely
\[
\lim_{\delta\to0}\limsup_{\varsigma\to0}I_{s,t}^{\mathrm{obs}}\leq0.
\]
Regarding the measure term $I^{\mathrm{meas}}_{s,t}$, it is written as
\begin{align*}
I_{s,t}^{\mathrm{meas}} & =2\int_{y,\eta}\int_{s}^{t}\partial_\eta\chi_{2,\tau}^{\varsigma,\delta,+}(y,\eta)\varphi_{\beta}(\eta)\zeta_{N}(\eta)\int_{\mathbb{R}}\int_{\mathbb{T}^{d}}\kappa^{\varsigma,\delta}(x,y,\xi,\eta)\mathrm{d}q_{1}(x,\xi,\tau)\\
& \quad+2\int_{y,\eta}\int_{s}^{t}\partial_\eta\chi_{1,\tau}^{\varsigma,\delta,-}(y,\eta)\varphi_{\beta}(\eta)\zeta_{N}(\eta)\int_{\mathbb{R}}\int_{\mathbb{T}^{d}}\kappa^{\varsigma,\delta}(z,y,\xi,\eta)\mathrm{d}q_{2}(z,\xi,\tau)\\
 & \quad+4\int_{y,\eta,x,z}\int_{s}^{t}|\Phi^{\prime}(u_{1})|^{\frac{1}{2}}|\Phi^{\prime}(u_{2})|^{\frac{1}{2}}\nabla_{x}u_{1}\cdot\nabla_{z}u_{2}\bar{\kappa}_{2,\tau}^{\varsigma,\delta}\bar{\kappa}_{1,\tau}^{\varsigma,\delta}\varphi_{\beta}(\eta)\zeta_{N}(\eta)\mathrm{d}\tau.
\end{align*}
We first prove a claim that as distribution almost surely
\begin{equation}\label{eq:negative for partial xi chi}
\partial_\xi \chi^{\pm}_{i}\leq0,\quad\text{on }\mathbb{T}^d\times(0,\infty)\times(0,T).
\end{equation}
Analogously to the proof of \eqref{eq:chi leq chi psi}, we will use the fact that $\chi_{i}^{\pm}(x,\xi,\tau)=\lim_{k\to\infty} \chi_{i}(x,\xi,\tau\pm \tau_k^{\pm})$. Since $\chi_i\in \{0,1\}$ and $\partial_\xi\chi_i\leq0$ on $\mathbb{T}^d\times(0,\infty)\times(0,T)$, we apply the dominated convergence theorem to deduce that for all non-negative functions $\alpha\in C_c^\infty((0,\infty))$ and $\phi\in C^\infty(\mathbb{T}^d)$,
\begin{align*}
\int_{x,\xi}\chi_{i}^{\pm}(x,\xi,\tau)\phi(x)\alpha^\prime(\xi)
&=\int_{x,\xi}\lim_{k\to\infty}\chi_{i}(x,\xi,\tau\pm \tau_k^{\pm})\phi(x)\alpha^\prime(\xi)\\
&=\int_{x}\lim_{k\to\infty}\int_{\xi}\chi_{i}(x,\xi,\tau\pm \tau_k^{\pm})\phi(x)\alpha^\prime(\xi)\\
&=\int_{x}\lim_{k\to\infty}\int_{\xi}\Big(-\partial_{\xi}\chi_{i}(x,\xi,\tau\pm \tau_k^{\pm})\Big)\phi(x)\alpha(\xi)\\
&\geq 0.
\end{align*}
Since linear combinations of functions of the form $\alpha(\xi)\phi(x)$ are dense in the set of non-negative functions in $C_c^\infty(\mathbb{T}^d\times(0,\infty))$, the claim \eqref{eq:negative for partial xi chi} follows by integration by parts.
Based on \eqref{eq:negative for partial xi chi}, the definition of $\chi_{i,\tau}^{\varsigma,\delta,\pm}$, and the non-negativity of $\kappa^{\varsigma,\delta}$, we have almost surely
\begin{equation}\label{eq:negative partial eta chi}
\partial_\eta\chi_{i,\tau}^{\varsigma,\delta,\pm}(y,\eta)\leq0,\quad \text{for }(y,\eta,\tau)\in\mathbb{T}^d\times(0,\infty)\times(0,T).
\end{equation}

In view of $\chi^{\varsigma,\delta,\pm}_{i,\tau}=\chi^{\varsigma,\delta}_{i,\tau}$ for $i\in\{1,2\}$ and almost all $\tau\in[0,T]$, using \eqref{eq:negative partial eta chi}, Definition \ref{def:def of stochastic kinetic solution} (vi) and \eqref{eq:partial eta}, for $\delta\in (0,{\beta}/{4})$, we have almost surely
\begin{align*}
I_{s,t}^{\mathrm{meas}} 
&\leq2\int_{x,y,\eta}\int_{s}^{t}\partial_\eta\chi_{2,\tau}^{\varsigma,\delta}(y,\eta)\varphi_{\beta}(\eta)\zeta_{N}(\eta)\bar{\kappa}_{1,\tau}^{\varsigma,\delta}\Phi^{\prime}(u_1)|\nabla_x u_1|^{2}\mathrm{d}\tau\\
& \quad+2\int_{z,y,\eta}\int_{s}^{t}\partial_\eta\chi_{1,\tau}^{\varsigma,\delta}(y,\eta)\varphi_{\beta}(\eta)\zeta_{N}(\eta)\bar{\kappa}_{2,\tau}^{\varsigma,\delta}\Phi^{\prime}(u_2)|\nabla_z u_2|^{2}\mathrm{d}\tau\\
 & \quad+4\int_{y,\eta,x,z}\int_{s}^{t}|\Phi^{\prime}(u_{1})|^{\frac{1}{2}}|\Phi^{\prime}(u_{2})|^{\frac{1}{2}}\nabla_{x}u_{1}\cdot\nabla_{z}u_{2}\bar{\kappa}_{2,\tau}^{\varsigma,\delta}\bar{\kappa}_{1,\tau}^{\varsigma,\delta}\varphi_{\beta}(\eta)\zeta_{N}(\eta)\mathrm{d}\tau\\
 &=-2\int_{x,y,\eta,z}\int_{s}^{t}\varphi_{\beta}(\eta)\zeta_{N}(\eta)\bar{\kappa}_{1,\tau}^{\varsigma,\delta}\bar{\kappa}_{2,\tau}^{\varsigma,\delta}\Phi^{\prime}(u_1)|\nabla_x u_1|^{2}\mathrm{d}\tau\\
 & \quad-2\int_{z,y,\eta,x}\int_{s}^{t}\varphi_{\beta}(\eta)\zeta_{N}(\eta)\bar{\kappa}_{1,\tau}^{\varsigma,\delta}\bar{\kappa}_{2,\tau}^{\varsigma,\delta}\Phi^{\prime}(u_2)|\nabla_z u_2|^{2}\mathrm{d}\tau\\
  & \quad+4\int_{y,\eta,x,z}\int_{s}^{t}|\Phi^{\prime}(u_{1})|^{\frac{1}{2}}|\Phi^{\prime}(u_{2})|^{\frac{1}{2}}\nabla_{x}u_{1}\cdot\nabla_{z}u_{2}\bar{\kappa}_{2,\tau}^{\varsigma,\delta}\bar{\kappa}_{1,\tau}^{\varsigma,\delta}\varphi_{\beta}(\eta)\zeta_{N}(\eta)\mathrm{d}\tau\\
& \leq0.
\end{align*}
Based on all the previous estimates, we conclude that
\begin{align*}
&\int_{x,\xi}\big(\chi_{1,t}^{+}-\chi_{2,t}^{+}\big)^{2}- \int_{x,\xi}\big(\chi_{1,s}-\chi_{2,s}\big)^{2}\\
& =\lim_{N\to \infty}\Bigg(\lim_{\beta\to0}\bigg(\lim_{\delta\to0}\Big(\limsup_{\varsigma\to0}\Big(I_{s,t}^{\mathrm{obs}}+I_{s,t}^{\mathrm{err}}+I_{s,t}^{\mathrm{meas}}+I_{s,t}^{\mathrm{cut}}+I_{s,t}^{\mathrm{mart}}+I_{s,t}^{\mathrm{cons}}\Big)\Big)\bigg)\Bigg)\leq0.
\end{align*}
Using the fact that $t\in\mathcal{T}$, we have $\chi^+_{i,t}=\chi_{i,t}$ for $i\in\{1,2\}$, which completes the proof.
\end{proof}
\begin{rem}\label{rem:measure 0,right initial}
Note that $\chi^{-}(x,\xi,0):=\chi(x,\xi,0)=\bar{\chi}(u_{\mathrm{init}},\xi)$. Using \cite[Proposition 8]{DV10} and \eqref{eq:right left differnt in t}, we have
\begin{align*}
&\int_{\mathbb{T}^d}\int_{\mathbb{R}}\big[\tilde{\chi}^{+}(x,\xi,0)-\tilde{\chi}(x,\xi,0)\big]\phi(x,\xi)\mathrm{d}\xi\mathrm{d}x\nonumber\\
&=-\int_{\mathbb{R}}\int_{\mathbb{T}^{d}}\partial_{\xi}\phi(x,\xi)\mathrm{d}q(x,\xi,\{0\}) -\int_{\mathbb{T}^{d}}\phi(x,\psi(x,0))\mathrm{d}\nu(x,\{0\}).
\end{align*}
Intuitively, one expects that $\nu(x,\{0\})=0$.
Indeed, by the penalization method, the penalization term at the initial time is
$
\varepsilon^{-1}(u_{\mathrm{init}}-\psi(\cdot,0))^{+},
$
which vanishes under the compatibility assumption on the initial data.
However, since the penalization term converges only weakly to $\nu$ in the sense of measures as $\varepsilon\to0$, this property alone does not  directly yield $\nu(x,\{0\})=0$.
Consequently, the identity $\chi^{+}(x,\xi,0)=\chi(x,\xi,0)$ cannot be immediately deduced from \eqref{eq:kinetic equation test with time}.
We thus invoke condition (iii) in Definition \ref{def:def of stochastic kinetic solution}, which ensures that the solution satisfies the initial condition in the time-averaged $L^1$ sense.
\end{rem}
\begin{cor}\label{cor:unique with inital}
Let Assumption \ref{assu:assum for F} and Assumption \ref{assu:assu for unique} hold for two initial functions $u_{1,\mathrm{init}}$ and $u_{2,\mathrm{init}}$.
Let $(u_{1},\nu_{1})$ and $(u_{2},\nu_{2})$ be stochastic kinetic solutions of the obstacle problem for \eqref{eq:main} in the
sense of Definition \ref{def:def of stochastic kinetic solution} with initial data $u_{1,\mathrm{init}}$ and $u_{2,\mathrm{init}}$, respectively.
Let $\chi_{1}$ and $\chi_{2}$ be the kinetic functions of $u_{1}$ and $u_{2}$.
Then, we have, almost surely, for almost every $t\in(0,T)$,
\begin{align}
\|u_1(t)-u_{2}(t)\|_{L^1(\mathbb{T}^d)}
\leq\|u_{1,\mathrm{init}}-u_{2,\mathrm{init}}\|_{L^1(\mathbb{T}^d)}.
\end{align}
\end{cor}
\begin{proof}
Using Theorem \ref{thm:stability L1} and the definition of $\chi$, we have, almost surely, for almost every $s,t\in(0,T)$ satisfying $s<t$,
\[
\int_{x} \big|u_{1}(x,t)-u_{2}(x,t)\big|\leq\int_{x} \big|u_{1}(x,s)-u_{2}(x,s)\big|.
\]
For each $\tau\in(0,t)$, by averaging over $s\in(0,\tau)$, we have
\[
\int_{x} \big|u_{1}(x,t)-u_{2}(x,t)\big|\leq\frac{1}{\tau}\int_0^{\tau}\int_{x} \big|u_{1}(x,s)-u_{2}(x,s)\big|\mathrm{d}s.
\]
Taking the limit $\tau\to0$ and using Definition \ref{def:def of stochastic kinetic solution} (iii), the proof is completed.
\end{proof}

\begin{thm}
\label{thm:uniqueness}Let Assumption \ref{assu:assum for F} and Assumption \ref{assu:assu for unique} hold.  Then the obstacle
problem for \eqref{eq:main} admits at most one stochastic kinetic solution under Definition \ref{def:def of stochastic kinetic solution}.
\end{thm}

\begin{proof}
Let $(u_{1},\nu_{1})$ and $(u_{2},\nu_{2})$ be stochastic kinetic solutions
of the obstacle problem for \eqref{eq:main} in the sense of Definition \ref{def:def of stochastic kinetic solution} with the same initial
data $u_{\mathrm{init}}$. From Corollary \ref{cor:unique with inital}, we have
\begin{align}\label{eq: equ for u1 and u2}
u_{1}=u_{2},\quad\text{a.s. on }\Omega\times Q_{T}.
\end{align}
We now show that there exists a full-measure set $\tilde{\Omega}\subset\Omega$ such that $\nu_{1}(\omega)=\nu_{2}(\omega)$ as Radon measures on $Q_{T}$, for every $\omega\in\tilde{\Omega}$.
Since each $(u_i,\nu_i)$ is a stochastic kinetic solution, for $N\in\mathbb{N}$, $\beta\in(0,1)$, $\rho\in C^{\infty}(\mathbb{T}^{d})$, $\alpha\in C_c^{\infty}([0,T))$, and $B\in\mathcal{F}$, we take test functions in \eqref{eq:kinetic equation test with time} converging to $\zeta_{N}\varphi_{\beta}\rho\alpha$, using the dominated convergence theorem and the fact \eqref{eq: equ for u1 and u2}, we have
\begin{equation}\label{eq:equal for nu q lambda}
\mathcal{G}_t(q_1,\nu_1)=\mathcal{G}_t(q_2,\nu_2),
\end{equation}
where
\begin{align}
 \mathcal{G}_t(q_i,\nu_i)&:= -2\beta^{-1}\mathbb{E}\Bigg[\mathbf{1}_{B}\int_{0}^{T}\int_{\beta/2}^{\beta}\int_{\mathbb{T}^{d}}\rho(x)\alpha(s)\mathrm{d}q_{i}\Bigg]+\mathbb{E}\Bigg[\mathbf{1}_{B}\int_{0}^{T}\int_{N}^{N+1}\int_{\mathbb{T}^{d}}\rho(x)\alpha(s)\mathrm{d}q_{i}\Bigg]\nonumber\\
 & \quad-\mathbb{E}\Bigg[\mathbf{1}_{B}\int_{0}^{T}\int_{\mathbb{T}^{d}}\zeta_{N}(\psi)\varphi_{\beta}(\psi)\rho(x)\alpha(s)\mathrm{d}\nu_{i}\Bigg]\nonumber.
\end{align}
We now consider the limit of function $\mathcal{G}$  when taking $\beta\to0$ and $N\to\infty$.
Using Definition \ref{def:def of stochastic kinetic solution} (vii) and Proposition \ref{prop:proposition for limit measure}, we have
\begin{align}
 & \lim_{\beta\to0}2\beta^{-1}\mathbb{E}\Bigg[\mathbf{1}_{B}\int_{0}^{T}\int_{\beta/2}^{\beta}\int_{\mathbb{T}^{d}}\rho(x)\alpha(s)\mathrm{d}q_{i}\Bigg]+\lim_{N\to\infty}\mathbb{E}\Bigg[\mathbf{1}_{B}\int_{0}^{T}\int_{N}^{N+1}\int_{\mathbb{T}^{d}}\rho(x)\alpha(s)\mathrm{d}q_{i}\Bigg]\label{eq:estimate of q for equ nu}\\
 & \leq C\lim_{\beta\to0}\beta^{-1}\mathbb{E}\Big[q_{i}(\mathbb{T}^{d}\times[\beta/2,\beta]\times[0,T))\Big]+C\lim_{N\to\infty}\mathbb{E}\Big[q_{i}(\mathbb{T}^{d}\times[N,N+1]\times[0,T))\Big]\nonumber\\
 & =0.\nonumber
\end{align}
Therefore, taking the limits $\beta\to0$ and $N\to\infty$ in \eqref{eq:equal for nu q lambda}, by \eqref{eq:estimate of q for equ nu} and the Lebesgue's dominated convergence theorem, we have
\begin{align}\notag
\mathbb{E}\Bigg[\mathbf{1}_{B}\int_{0}^{T}\int_{\mathbb{T}^{d}}\rho(x)\alpha(s)\mathrm{d}\nu_{1}\Bigg]=\mathbb{E}\Bigg[\mathbf{1}_{B}\int_{0}^{T}\int_{\mathbb{T}^{d}}\rho(x)\alpha(s)\mathrm{d}\nu_{1}\Bigg].
\end{align}
Due to the arbitrariness of functions $\rho$, $\alpha$ and the set $B$, we have almost surely
\[
\nu_{1}=\nu_{2}
\]
as a Radon measure on $Q_{T}$, which completes the proof.
\end{proof}

\section{Existence of stochastic kinetic solutions}
\label{sec:exist}
In this section, we will construct a stochastic kinetic solution of the obstacle problem for \eqref{eq:main}.
The main strategy is to approximate the obstacle problem by a family of penalized equations.
More precisely, we introduce the following penalized equation:
\begin{align}
\mathrm{d}u_{\varepsilon} & =\Big[\Delta\Phi(u_{\varepsilon})+\frac{1}{2}\nabla\cdot(F_{1}[\sigma^{\prime}(u_{\varepsilon})]^{2}\nabla u_{\varepsilon}+\sigma(u_{\varepsilon})\sigma^{\prime}(u_{\varepsilon})F_{2})-\nabla\cdot g(u_\varepsilon) \label{eq:penalized eq}\\
 & \quad-\frac{1}{\varepsilon}(u_{\varepsilon}-\psi)^{+}\Big]\mathrm{d}t-\nabla\cdot(\sigma(u_{\varepsilon})\mathrm{d}\xi^{F}),\nonumber
\end{align}
for $\varepsilon\in(0,1)$.
\begin{defn}
\label{def:def of stochastic kinetic solution without obstacle} For fixed $\varepsilon\in(0,1)$, a stochastic kinetic
solution of \eqref{eq:penalized eq} is a nonnegative, almost surely continuous $L^{1}(\mathbb{T}^{d})$-valued $\mathcal{F}_{t}$-predictable function $u_\varepsilon\in L^{1}(\Omega\times[0,T];L^{1}(\mathbb{T}^{d}))$ satisfying following properties.
\begin{itemize}
\item[(i)] Preservation of mass: almost surely for every $\phi\in C_c^{\infty}([0,T))$,
\begin{align}\label{eq:L1 perservation without obstacle}
    -\int_{0}^{T}\int_{\mathbb{T}^d} u_\varepsilon\partial_s\phi(s)\mathrm{d}x\mathrm{d}s+\int_{0}^{T}\int_{\mathbb{T}^d}\phi(s)\varepsilon^{-1}(u_\varepsilon-\psi)^{+}\mathrm{d}x\mathrm{d}s=\int_{\mathbb{T}^d} u_{\mathrm{init}}\phi(0)\mathrm{d}x
\end{align}
\item[(ii)] $\sigma(u_\varepsilon)\in L^{2}(\Omega;L^{2}(Q_T))$,\quad $g(u_\varepsilon)\in L^{1}(\Omega;L^{1}(Q_T;\mathbb{R}^d))$.
\item[(iii)] For every $K\in\mathbb{N}$,
\[
[(u_\varepsilon\land K)\lor(1/K)]\in L^{2}(\Omega\times[0,T];H^{1}(\mathbb{T}^{d}))).
\]
Furthermore, there exists a kinetic measure $m_\varepsilon$ satisfying the following properties.
\item[(iv)] Almost surely as a nonnegative measure, $m_\varepsilon$ satisfies
\[
\delta_{0}(\xi-u_\varepsilon)\Phi^{\prime}(\xi)|\nabla u_\varepsilon|^{2}\leq m_\varepsilon\quad\text{on}\quad\mathbb{T}^{d}\times(0,\infty)\times[0,T].
\]
\item[(v)] (Vanishing)
\[
\lim_{N\to\infty}\mathbb{E}\big[m_\varepsilon(\mathbb{T}^{d}\times[N,N+1]\times[0,T))\big]=0.
\]
\item[(vi)] Preservation of energy: for every $\phi\in C_c^{\infty}([0,T))$,
\begin{align}
&-\mathbb{E}\int_0^T\int_{\mathbb{T}^d}u_\varepsilon^2(x,s)\partial_s\phi(s)\mathrm{d}x\mathrm{d}s+\frac{2}{\varepsilon}\mathbb{E}\int_0^T\int_{\mathbb{T}^d}\phi(s)\psi(x,s)(u_\varepsilon-\psi)^+\mathrm{d}x\mathrm{d}s\label{eq:energy preservation for penalty solution}\\
&+\frac{2}{\varepsilon}\mathbb{E}\int_0^T\int_{\mathbb{T}^d}\phi(s)|(u_\varepsilon-\psi)^+|^2\mathrm{d}x\mathrm{d}s+2\mathbb{E}\int_0^T\phi(s)\mathrm{d}m_\varepsilon(\mathbb{T}^d,[0,\infty),s)\nonumber\\
&=\mathbb{E}\int_{\mathbb{T}^d}u^2_{\mathrm{init}}(x)\phi(0)\mathrm{d}x-\frac{1}{2}\mathbb{E}\int_{0}^{T}\int_{\mathbb{T}^d}\phi(s)\sigma^2(u_{\varepsilon})(\nabla_{x}\cdot F_{2}-2F_3)\mathrm{d}x\mathrm{d}s.\nonumber
\end{align}
\item[(vii)] For every $\phi\in C_{c}^{1}(\mathbb{T}^{d}\times(0,\infty)\times[0,T))$, the kinetic function $\chi_\varepsilon$ of $u_\varepsilon$ almost surely satisfies
\begin{align}
& -\int_{0}^{T}\int_{\mathbb{R}}\int_{\mathbb{T}^d}\chi_\varepsilon(x,\xi,s)\partial_s\phi(x,\xi,s)\mathrm{d}x\mathrm{d}\xi\mathrm{d}s\label{eq:kinetic equation test with time without obstacle}\\
 & =\int_{\mathbb{R}}\int_{\mathbb{T}^{d}}\bar{\chi}(u_{\mathrm{init}},\xi)\phi(x,\xi,0)\mathrm{d}x\mathrm{d}\xi-\int_{0}^{T}\int_{\mathbb{T}^{d}}\Phi^{\prime}(u_\varepsilon)\nabla_x u_\varepsilon\cdot(\nabla_x\phi)(x,u_\varepsilon,s)\mathrm{d}x\mathrm{d}s\nonumber \\
 & \quad-\frac{1}{2}\int_{0}^{T}\int_{\mathbb{T}^{d}}\big(F_{1}(x)[\sigma^{\prime}(u_\varepsilon)]^{2}\nabla_x u_\varepsilon+\sigma(u_\varepsilon)\sigma^{\prime}(u_\varepsilon)F_{2}(x)\big)\cdot(\nabla_x\phi)(x,u_\varepsilon,s)\mathrm{d}x\mathrm{d}s\nonumber\\
 & \quad-\int_{0}^{T}\int_{\mathbb{R}}\int_{\mathbb{T}^{d}}\partial_{\xi}\phi(x,\xi,s)\mathrm{d}m_\varepsilon(x,\xi,s)-\int_0^{T}\int_{\mathbb{T}^d}\phi(x,u_\varepsilon,s)\nabla_x\cdot g(u_\varepsilon)\mathrm{d}x\mathrm{d}s\nonumber \\
 & \quad+\frac{1}{2}\int_{0}^{T}\int_{\mathbb{T}^{d}}\Big(\sigma(u_\varepsilon)\sigma^{\prime}(u_\varepsilon)\nabla_x u_\varepsilon\cdot F_{2}(x)+\sigma^{2}(u_\varepsilon)F_{3}(x)\Big)(\partial_{\xi}\phi)(x,u_\varepsilon,s)\mathrm{d}x\mathrm{d}s\nonumber \\
 & \quad-\int_{0}^{T}\int_{\mathbb{T}^{d}}\phi(x,u_\varepsilon,s)\frac{1}{\varepsilon}(u_\varepsilon-\psi)^+\mathrm{d}x\mathrm{d}s-\int_{0}^{T}\int_{\mathbb{T}^{d}}\phi(x,u_\varepsilon,s)\nabla_x\cdot(\sigma(u_\varepsilon)f_{k})\mathrm{d}x\mathrm{d}B_{s}^{k}.\nonumber
\end{align}
\end{itemize}
\end{defn}
\begin{rem}\label{rem:finite for kinetic measure}
Actually, since we consider the upper obstacle case with the constraint $u\leq\psi$, the kinetic measure $m_\varepsilon$ is finite due to \cite[Remark 5.26]{fehrman2024well} and the fact that $0\leq u_{\mathrm{init}}\leq \psi\leq M$.
For the lower obstacle case with the constraint $u\geq\psi$ (instead of $u\leq\psi$), we can first deal with the initial data $u_{\mathrm{init}}\in L^p(\Omega;L^p(\mathbb{T}^d))\cap L^{m+p-1}(\Omega;L^1(\mathbb{T}^d))$ to guarantee the finiteness of the kinetic measure $m_\varepsilon$. Then, following the proof of \cite[Corollary 5.28]{fehrman2024well}, we can generalize the initial data to $u_{\mathrm{init}}\in L^1(\Omega;L^1(\mathbb{T}^d))$.
\end{rem}
Due to the degeneracy of the equation and the lower regularity of $\sigma$,  we first deal with the approximating penalized equations:
\begin{align}
\mathrm{d}u_{\alpha,n,\varepsilon} & =\Big[\Delta\Phi(u_{\alpha,n,\varepsilon})+\alpha\Delta u_{\alpha,n,\varepsilon}+\frac{1}{2}\nabla_x\cdot(F_{1}[\sigma_n^{\prime}(u_{\alpha,n,\varepsilon})]^{2}\nabla_x u_{\alpha,n,\varepsilon}+\sigma_n(u_{\alpha,n,\varepsilon})\sigma_n^{\prime}(u_{\alpha,n,\varepsilon})F_{2})\nonumber \\
 & \quad-\nabla_x\cdot g(u_{\alpha,n,\varepsilon})-\frac{1}{\varepsilon}(u_{\alpha,n,\varepsilon}-\psi)^{+}\Big]\mathrm{d}t-\nabla_x\cdot(\sigma_{n}(u_{\alpha,n,\varepsilon})\mathrm{d}\xi^{F}),\label{eq:appro penalized eq}
\end{align}
for $n\in\mathbb{N}$ and $\alpha\in(0,1)$, where $\sigma_n\in C^\infty(\mathbb{R})$ is a smooth approximation of $\sigma$ in $C^1_{\mathrm{loc}}((0,\infty))$ and satisfies Assumption \ref{assu:assu for existence} below uniformly in $n\in\mathbb{N}$.
\begin{defn}
\label{def:Weak solution approximating penalized}A solution to the approximating penalized equation \eqref{eq:appro penalized eq} with initial data $u_{\alpha,n,\varepsilon}(0)=u_{\mathrm{init}}$ is a continuous $L^{2}(\mathbb{T}^{d})$-valued, $\mathcal{F}_{t}$-predictable process such that almost surely $u_{\alpha,n,\varepsilon}$ and $\llbracket\sqrt{\Phi^{\prime}}\rrbracket(u_{\alpha,n,\varepsilon})$ are in $L^{2}(0,T;H^{1}(\mathbb{T}^{d}))$ ($\llbracket\cdot\rrbracket$ is defined by \eqref{eq:initegrable}), and such that for all $\phi\in C^{\infty}(\mathbb{T}^{d})$, almost surely for every $t\in[0,T]$,
\begin{align}
\int_{x}u_{\alpha,n,\varepsilon}(x,t)\phi & =\int_{x}u_{\mathrm{init}}\phi-\int_{0}^{t}\int_{x}\Phi^{\prime}(u_{\alpha,n,\varepsilon})\nabla_{x}u_{\alpha,n,\varepsilon}\cdot\nabla_{x}\phi\mathrm{d}s \label{eq: weak solution for nondegenerate equ}\\
 & \quad-\alpha\int_{0}^{t}\int_{x}\nabla_{x}u_{\alpha,n,\varepsilon}\cdot\nabla_{x}\phi\mathrm{d}s+\int_0^t\int_x\nabla_x\phi \cdot g(u_{\alpha,n,\varepsilon})\mathrm{d}s\nonumber\\
 & \quad-\frac{1}{2}\int_{0}^{t}\int_{x}(F_{1}[\sigma_{n}^{\prime}(u_{\alpha,n,\varepsilon})]^{2}\nabla_{x}u_{\alpha,n,\varepsilon}+\sigma_{n}(u_{\alpha,n,\varepsilon})\sigma_{n}^{\prime}(u_{\alpha,n,\varepsilon})F_{2})\cdot\nabla_{x}\phi\mathrm{d}s\nonumber \\
 & \quad-\frac{1}{\varepsilon}\int_{0}^{t}\int_{x}(u_{\alpha,n,\varepsilon}-\psi)^{+}\phi\mathrm{d}s+\int_{0}^{t}\int_{x}\sigma_{n}(u_{\alpha,n,\varepsilon})\nabla_{x}\phi\cdot\mathrm{d}\xi^{F}.\nonumber
\end{align}
\end{defn}
The following assumptions on $\Phi$, $\sigma$ and $g$ are exactly those in \cite[Assumption 4.1 and Assumption 5.2]{fehrman2024well}.
We emphasize that no additional structural condition is imposed in order to treat the obstacle constraint.
Although these assumptions are technical in form, they cover the main model equations of interest.
In particular, they allow for degenerate nonlinear diffusion of porous-medium or fast-diffusion type, such as $\Phi(\xi)=\xi^{m}$, for $m>0$, together with $\sigma=\Phi^{1/2}$.
Thus, our existence result applies to the same broad class of generalized Dean--Kawasaki equations as in \cite{fehrman2024well}, while extending the theory to the obstacle setting.
\begin{assumption}\label{assu:assu for existence}
Let $\Phi,\sigma\in C([0,\infty))\cap C_{\mathrm{loc}}^1((0,\infty))$, $g\in C([0,\infty);\mathbb{R}^d)\cap C^{1}_{\mathrm{loc}}((0,\infty);\mathbb{R}^d)$, and $p\in[2,\infty)$ satisfy the followings.
\begin{itemize}
\item[1.] $\Phi(0)=\sigma(0)=0$ and $\Phi^\prime>0$ on $(0,\infty)$.
\item[2.] There exist $m\in[1,\infty)$ and $c\in(0,\infty)$ such that
\[
\Phi(\xi)\leq c(1+\xi^m)\quad {\rm{for}}\ \xi\in[0,\infty).
\]
\item[3.] There exists $c\in(0,\infty)$ such that
\[
|g(\xi)|+\Phi^\prime(\xi)\leq c(1+\xi+\big|\llbracket |\cdot|^{(p-2)/2}\sqrt{\Phi^\prime}\rrbracket(\xi)\big|^2)\quad {\rm{for}}\ \xi\in(0,\infty).
\]
\item[4.] Either there exists $c\in(0,\infty)$ and $\theta\in[0,1/2]$ such that
\begin{equation*}
\xi^{-\frac{p-2}{2}}|\Phi^\prime(\xi)|^{-\frac{1}{2}}\leq c\xi^\theta\quad {\rm{for}}\ \xi\in(0,\infty),
\end{equation*}
or there exists $c\in(0,\infty)$ and $\tilde{p}\in[1,\infty)$ such that
\begin{equation*}
|\xi-\tilde{\xi}|^{\tilde{p}}\leq c\Big|\llbracket |\cdot|^{(p-2)/2}\sqrt{\Phi^\prime}\rrbracket(\xi)-\llbracket |\cdot|^{(p-2)/2}\sqrt{\Phi^\prime}\rrbracket(\tilde{\xi})\Big|^2\quad {\rm{for}}\ \xi,\tilde{\xi}\in[0,\infty).
\end{equation*}
\item[5.] There exists $c\in(0,\infty)$ such that
\begin{align*}
\sigma^2(\xi)&\leq c(1+\xi+\big|\llbracket\sqrt{\Phi^\prime}\rrbracket(\xi)\big|^2)\quad {{and}}\\
\xi^{p-2}\sigma^2(\xi)&\leq c(1+\xi+\big|\llbracket |\cdot|^{(p-2)/2}\sqrt{\Phi^\prime}\rrbracket(\xi)\big|^2)\quad {\rm{for}}\ \xi\in[0,\infty).
\end{align*}
\item[6.] Either $\nabla\cdot F_2=0$ or there exists $c\in(0,\infty)$ such that
\begin{equation*}
\big|\llbracket|\cdot|^{p-2}\sigma\sigma^{\prime}\rrbracket(\xi)\big|\leq c(1+\xi+\big|\llbracket |\cdot|^{(p-2)/2}\sqrt{\Phi^\prime}\rrbracket(\xi)\big|^2)\quad {\rm{for}}\ \xi\in[0,\infty).
\end{equation*}
\item[7.] For every $\delta\in(0,1)$, there exists $c_\delta\in(0,\infty)$ such that
\begin{equation*}
\frac{|\sigma^\prime(\xi)|^4}{\Phi^\prime(\xi)}+|\sigma\sigma^\prime(\xi)|+\Phi^\prime(\xi)\leq c_\delta(1+\xi+\big|\llbracket |\cdot|^{(p-2)/2}\sqrt{\Phi^\prime}\rrbracket(\xi)\big|^2)\quad {\rm{for}}\ \xi\in(\delta,\infty).
\end{equation*}
\item[8.] There exists $c\in (0,\infty)$ such that
$$
\limsup_{\xi\rightarrow 0^+}\frac{\sigma^2(\xi)}{\xi}\leq c.
$$
\end{itemize}
\end{assumption}
\begin{rem}

The existence assumptions in Assumption \ref{assu:assu for existence} do not explicitly include the additional tail assumptions used for uniqueness, namely conditions 5 and 6 of Assumption \ref{assu:assu for unique}.
However, these assumptions are used in the uniqueness argument only to guarantee the corresponding tail estimates; see \cite[Remark 4.2]{fehrman2024well}.
In the present application, we do not need these tail estimates for arbitrary admissible functions. It is enough to verify them for the approximating solutions constructed in the existence proof and for their limits. This can be done directly from the a priori estimates under Assumption \ref{assu:assu for existence}.

Specifically, by Chebyshev's inequality, if
\[
u\in L^\infty(0,T;L^1(\mathbb T^d)),\qquad \sigma(u)\in L^2(0,T;L^2(\mathbb T^d)),
\]
then
\begin{align}\label{eq:tail condition for sigma}
\lim_{M\to\infty}\Big(\sup_{\xi\in[M,(M+1)\land u]}|\sigma(\xi)|\mathbf{1}_{\{u>M\}}\Big)=0
\quad\text{strongly in }L^2(Q_T).
\end{align}
Similarly, if
\[
u\in L^\infty(0,T;L^1(\mathbb T^d)),\qquad g(u)\in L^1(0,T;L^1(\mathbb T^d)),
\]
then
\begin{align}\label{eq:tail condition for g}
\lim_{M\to\infty}\Big(\sup_{\xi\in[M,(M+1)\land u]}|g(\xi)|\mathbf{1}_{\{u>M\}}\Big)=0
\quad\text{strongly in }L^1(Q_T).
\end{align}

These estimates are available under the existence hypotheses (Assumption \ref{assu:assu for existence}). Indeed, the a priori estimates from Lemma \ref{lem:priori-estimates} yield, almost surely,
\[
    u_{\alpha,n,\varepsilon}\in L^\infty(0,T;L^1(\mathbb T^d)),
    \qquad
    \llbracket\sqrt{\Phi^\prime}\rrbracket(u_{\alpha,n,\varepsilon})\in L^2(Q_T),
    \qquad
    \llbracket|\cdot|^{(p-2)/2}\sqrt{\Phi^\prime}\rrbracket(u_{\alpha,n,\varepsilon})\in L^2(Q_T).
\]
Together with the growth bounds
\[
    \sigma^2(\xi)\le C\bigl(1+\xi+\big|\llbracket\sqrt{\Phi^\prime}\rrbracket(\xi)\big|^2\bigr),
    \qquad
    |g(\xi)|\le c\bigl(1+\xi+\big|\llbracket |\cdot|^{(p-2)/2}\sqrt{\Phi^\prime}\rrbracket(\xi)\big|^2\bigr),
\]
and the monotonicity of
\[
\llbracket\sqrt{\Phi^\prime}\rrbracket
\quad\text{and}\quad
\llbracket |\cdot|^{(p-2)/2}\sqrt{\Phi^\prime}\rrbracket,
\]
we obtain \eqref{eq:tail condition for sigma} and \eqref{eq:tail condition for g} for $u_{\alpha,n,\varepsilon}$, since the right-hand sides have vanishing $L^1$ tails.

The same conclusion holds after passing to the limits $\alpha\to0$, $n\to\infty$, and $\varepsilon\to0$. Hence the estimates actually used in the uniqueness proof are valid under Assumption \ref{assu:assu for existence}. Consequently, the uniqueness result applies to the solutions constructed under the existence hypotheses; see the proof of Theorem \ref{thm:existence}, even though these hypotheses do not explicitly contain the additional uniqueness assumptions.

\end{rem}
\begin{prop}
(Non-negativity) \label{prop:(Nonnegativity)}Let Assumption \ref{assu:assum for F} and Assumption \ref{assu:assu for existence} hold for some $p\in[2,\infty)$.
For $n\in\mathbb{N}$ and $\alpha,\varepsilon\in(0,1)$, let $u_{\alpha,n,\varepsilon}$ be a solution of \eqref{eq:appro penalized eq} with initial data $u_{\mathrm{init}}$, then we have almost surely
\[
u_{\alpha,n,\varepsilon}\geq0,\quad a.e.\  (x,t)\in Q_T.
\]
\end{prop}

\begin{proof}
For the sake of simplicity, denote by $u=u_{\alpha,n,\varepsilon}$.
Let $u^-$ be the negative part of $u$, i.e. $u^-=\max(-u,0)$. We aim to show that $u^-=0$, almost surely on $\Omega\times Q_T$. We introduce a non-negative smooth function $\rho:\mathbb{R}\to[0,\infty)$ such that $\text{supp}\ \rho \subset(0,1)$, $\rho\leq2$ and $\int_\mathbb{R}\rho(r)\mathrm{d}r=1$.
For any $\delta>0$, we define the standard mollifier $\rho_\delta(r):=\delta^{-1}\rho(\delta^{-1}r)$ and a function $\eta_{\delta}\in C^2(\mathbb{R})$ satisfying
$$
\eta_\delta(0)=\eta_{\delta}^\prime(0)=0,\quad \eta_\delta^{\prime\prime}(r)=\rho_\delta(r).
$$
Applying the  It\^{o}'s formula (cf. \cite[Theorem 3.1]{krylov2013relatively}), we have almost surely for every $t\in[0,T]$,
\begin{align*}
 & \int_{x}\eta_{\delta}(-u(x,t))\\
 & =\int_{x}\eta_{\delta}(-u_{\mathrm{init}})-\int_{0}^{t}\int_{x}\eta_{\delta}^{\prime\prime}(-u)\Phi^{\prime}(u)\nabla_{x}u\cdot\nabla_{x}u\mathrm{d}s\\
 & \quad-\alpha\int_{0}^{t}\int_{x}\eta_{\delta}^{\prime\prime}(-u)\nabla_{x}u\cdot\nabla_{x}u\mathrm{d}s\\
 & \quad-\frac{1}{2}\int_{0}^{t}\int_{x}\eta_{\delta}^{\prime\prime}(-u)(F_{1}[\sigma_{n}^{\prime}(u)]^{2}\nabla_{x}u+\sigma_{n}(u)\sigma_{n}^{\prime}(u)F_{2})\cdot\nabla_{x}u\mathrm{d}s\\
 &\quad+\int_{0}^{t}\int_{x}\eta_{\delta}^{\prime\prime}(-u)\nabla_{x}u\cdot g(u)\mathrm{d}s\\
 & \quad +\frac{1}{\varepsilon}\int_{0}^{t}\int_{x}(u-\psi)^{+}\eta_{\delta}^{\prime}(-u)\mathrm{d}s+\int_{0}^{t}\int_{x}\sigma_{n}(u)\eta_{\delta}^{\prime\prime}(-u)\nabla_{x}u\cdot\mathrm{d}\xi^{F}\\
 & \quad+\frac{1}{2}\int_{0}^{t}\sum_{k=1}^{\infty}\int_{x}\eta_{\delta}^{\prime\prime}(-u)|\sigma_{n}^{\prime}(u)\nabla_{x}uf_{k}+\sigma_{n}(u)\nabla_{x}f_{k}|^{2}\mathrm{d}s.
\end{align*}

Since $\sigma_n(r)=0$ when $r\leq0$ and $\text{supp}\ \eta_\delta^{\prime\prime}\subset(0,\infty)$, we have
\begin{align*}
&-\frac{1}{2}\int_x\eta_{\delta}^{\prime\prime}(-u)(F_{1}[\sigma_{n}^{\prime}(u)]^{2}\nabla_{x}u+\sigma_{n}(u)\sigma_{n}^{\prime}(u)F_{2})\cdot\nabla_{x}u\\
&+\int_x\frac{1}{2}\sum_{k=1}^{\infty}\eta_\delta^{\prime\prime}(-u)|\sigma_{n}^{\prime}(u)\nabla_{x}uf_{k}+\sigma_{n}(u)\nabla_{x}f_{k}|^{2}
\\
&=\frac{1}{2}\int_x F_2\cdot\nabla_x\int_0^u\eta_\delta^{\prime\prime}(-r)\sigma_{n}(r)\sigma_n^\prime(r)\mathrm{d}r +\frac{1}{2}\int_x\eta_\delta^{\prime\prime}(-u)|\sigma_{n}(u)|^2F_3=0.
\end{align*}
Note that
\[
\int_{x}\eta_{\delta}^{\prime\prime}(-u)\nabla_{x}u\cdot g(u)\mathrm{d}s=\int_{x}\nabla_x\cdot\int_0^{u}\eta_{\delta}^{\prime\prime}(-r)g(r)\mathrm{d}r\mathrm{d}s=0.
\]
Moreover, using the non-negativity of $\psi$ and $u_{\mathrm{init}}$, together with $\text{supp}\ \eta_\delta, \text{supp}\ \eta^\prime_\delta\subset (0,\infty)$, we obtain almost surely
\[
\frac{1}{\varepsilon}\int_{0}^{t}\int_{x}(u-\psi)^{+}\eta_{\delta}^{\prime}(-u)\mathrm{d}s=0,\quad\text{and}\quad
\eta_{\delta}(-u_{\mathrm{init}})=0, \quad{\text{on}\ \ }\mathbb{T}^d.
\]
Combining the previous estimates, and taking the expectation, we obtain
\[
\mathbb{E}\int_{x}\eta_\delta(-u(x,t))\leq0,\quad\text{a.e. }t\in[0,T].
\]
Taking the limits $\delta\to 0$ and using the Lebesgue dominated theorem, we get the desired result.
\end{proof}

\begin{lem}
\label{lem:priori-estimates}(A priori estimates) Let Assumption \ref{assu:assum for F} and Assumption \ref{assu:assu for existence} hold for some $p\in[2,\infty)$.
For any $n\in\mathbb{N}$ and $\alpha,\varepsilon\in(0,1)$, let $u_{\alpha,n,\varepsilon}$ be a solution of \eqref{eq:appro penalized eq}, then we have
\begin{description}
  \item[(1)] almost surely for every $t\in[0,T]$,
\begin{equation}
\Vert u_{\alpha,n,\varepsilon}(t)\Vert_{L^{1}(\mathbb{T}^{d})}=\Vert u_{\mathrm{init}}\Vert_{L^{1}(\mathbb{T}^{d})}-\int^t_0\Vert\varepsilon^{-1}(u_{\alpha,n,\varepsilon}-\psi)^{+}\Vert_{L^{1}(\mathbb{T}^{d})}\mathrm{d}s,\label{eq:L1 identity}
\end{equation}
  \item[(2)] for any $\phi\in C_c^\infty([0,T))$,
\begin{align}
 & -\mathbb{E}\int_{0}^{T}\int_{x}\partial_t\phi(t)|u_{\alpha,n,\varepsilon}|^2\mathrm{d}t \label{eq: energy perservation for appro solution}\\
 & =\mathbb{E}\int_{x}\phi(0)|u_{\mathrm{init}}|^{2}-2\mathbb{E}\int_{0}^{T}\int_{x}\phi(t)\big(\Phi^{\prime}(u_{\alpha,n,\varepsilon})|\nabla_{x}u_{\alpha,n,\varepsilon}|^2+\alpha|\nabla_{x}u_{\alpha,n,\varepsilon}|^2\big)\mathrm{d}t \nonumber\\
 & \quad-\frac{2}{\varepsilon}\mathbb{E}\int_{0}^{T}\int_{x}\phi(t)|(u_{\alpha,n,\varepsilon}-\psi)^{+}|^{2}\mathrm{d}t-\frac{1}{2}\mathbb{E}\int_{0}^{T}\int_{x}\phi(t)\sigma^2_{n}(u_{\alpha,n,\varepsilon})(\nabla_{x}\cdot F_{2}-2F_3)\mathrm{d}t\nonumber \\
 & \quad-\frac{2}{\varepsilon}\mathbb{E}\int_{0}^{T}\int_{x}\phi(t)(u_{\alpha,n,\varepsilon}-\psi)^{+}\psi\mathrm{d}t,\nonumber
\end{align}
  \item[(3)] there exists a constant $C$ independent of $n$, $\alpha$, and $\varepsilon$ such that
\begin{align}
 & \sup_{t\in[0,T]}\mathbb{E}\Vert u_{\alpha,n,\varepsilon}(t)\Vert_{L^{p}(\mathbb{T}^{d})}^{p}+\mathbb{E}\Vert|u_{\alpha,n,\varepsilon}+1|^{\frac{p-2}{2}}\nabla_{x}\llbracket\sqrt{\Phi^{\prime}}\rrbracket(u_{\alpha,n,\varepsilon})\Vert_{L^{2}(Q_{T})}^{2}\label{eq=00FF1Apriori p} \\
 & +\alpha\mathbb{E}\Vert|u_{\alpha,n,\varepsilon}+1|^{\frac{p-2}{2}}\nabla_{x}u_{\alpha,n,\varepsilon}\Vert_{L^{2}(Q_{T})}^{2} +\varepsilon^{-1}\mathbb{E}\Vert|u_{\alpha,n,\varepsilon}+1|^{\frac{p-2}{2}}(u_{\alpha,n,\varepsilon}-\psi)^{+}\Vert_{L^{2}(Q_{T})}^{2}\nonumber\\
 & +\mathbb{E}\Vert|u_{\alpha,n,\varepsilon}+1|^{p-2}\varepsilon^{-1}(u_{\alpha,n,\varepsilon}-\psi)^{+}\Vert_{L^{1}(Q_{T})}\nonumber \\
 & \leq C\Big(1+\mathbb{E}\Vert u_{\mathrm{init}}\Vert_{L^{p}(\mathbb{T}^{d})}^{p}+\mathbb{E}\Vert u_{\mathrm{init}}\Vert_{L^{1}(\mathbb{T}^{d})}^{m+p-1}\Big).\nonumber
\end{align}
\end{description}
The estimate (\ref{eq=00FF1Apriori p}) also holds when $p=2$, that is
\begin{align}
 & \sup_{t\in[0,T]}\mathbb{E}\Vert u_{\alpha,n,\varepsilon}(t)\Vert_{L^{2}(\mathbb{T}^{d})}^{2}+\mathbb{E}\Vert\nabla_{x}\llbracket\sqrt{\Phi^{\prime}}\rrbracket(u_{\alpha,n,\varepsilon})\Vert_{L^{2}(Q_{T})}^{2}+\alpha\mathbb{E}\Vert\nabla_x u_{\alpha,n,\varepsilon}\Vert_{L^{2}(Q_{T})}^{2}\label{eq:priori} \\
 & +\varepsilon^{-1}\mathbb{E}\Vert(u_{\alpha,n,\varepsilon}-\psi)^{+}\Vert_{L^{2}(Q_{T})}^{2}+\mathbb{E}\Vert\varepsilon^{-1}(u_{\alpha,n,\varepsilon}-\psi)^{+}\Vert_{L^{1}(Q_{T})}\nonumber\\
 & \leq C\Big(1+\mathbb{E}\Vert u_{\mathrm{init}}\Vert_{L^{2}(\mathbb{T}^{d})}^{2}+\mathbb{E}\Vert u_{\mathrm{init}}\Vert_{L^{1}(\mathbb{T}^{d})}^{m+1}\Big).\nonumber
\end{align}
\end{lem}

\begin{proof}
First, taking the test function $\phi\equiv1$ in \eqref{eq: weak solution for nondegenerate equ} and using the non-negativity of $u_{\alpha,n,\varepsilon}$ from Proposition \ref{prop:(Nonnegativity)},
we obtain \eqref{eq:L1 identity}.

For \eqref{eq: energy perservation for appro solution}, using the It\^{o}'s formula (cf. \cite[Theorem 3.1]{krylov2013relatively}), we have almost surely for every $t\in[0,T]$ and $\phi\in C_c^\infty([0,T))$,
\begin{align}
 & -\int_{0}^{T}\partial_t\phi(t)|u_{\alpha,n,\varepsilon}|^2\mathrm{d}t\label{eq: energy perservation for appro eqn} \\
 & =\int_{x}\phi(0)|u_{\mathrm{init}}|^{2}-2\int_{0}^{T}\int_{x}\phi(t)\big(\Phi^{\prime}(u_{\alpha,n,\varepsilon})|\nabla_{x}u_{\alpha,n,\varepsilon}|^2+\alpha|\nabla_{x}u_{\alpha,n,\varepsilon}|^2\big)\mathrm{d}t \nonumber\\
 & \quad-\int_{0}^{T}\int_{x}\phi(t)(F_{1}[\sigma_{n}^{\prime}(u_{\alpha,n,\varepsilon})]^{2}\nabla_{x}u_{\alpha,n,\varepsilon}+\sigma_{n}(u_{\alpha,n,\varepsilon})\sigma_{n}^{\prime}(u_{\alpha,n,\varepsilon})F_{2})\cdot\nabla_{x}u_{\alpha,n,\varepsilon}\mathrm{d}t\nonumber \\
 &\quad+2\int_{0}^{T}\int_{x}\phi(t)\nabla_x u_{\alpha,n,\varepsilon}\cdot g(u_{\alpha,n,\varepsilon})\mathrm{d}t\nonumber\\
 & \quad-\frac{2}{\varepsilon}\int_{0}^{T}\int_{x}\phi(t)(u_{\alpha,n,\varepsilon}-\psi)^{+}u_{\alpha,n,\varepsilon}\mathrm{d}t +2\int_{0}^{T}\int_{x}\phi(t)\sigma_{n}(u_{\alpha,n,\varepsilon})\nabla_{x}u_{\alpha,n,\varepsilon}\cdot\mathrm{d}\xi^{F}\nonumber \\
 & \quad+\int_{0}^{T}\sum_{k=1}^{\infty}\int_{x}\phi(t)|\sigma_{n}^{\prime}(u_{\alpha,n,\varepsilon})\nabla_{x}u_{\alpha,n,\varepsilon}f_{k}+\sigma_{n}(u_{\alpha,n,\varepsilon})\nabla_{x}f_{k}|^{2}\mathrm{d}t.\nonumber
\end{align}
Note that
\[
\int_x\nabla_x u_{\alpha,n,\varepsilon}\cdot g(u_{\alpha,n,\varepsilon})=\int_x\nabla_x\cdot\llbracket g\rrbracket(u_{\alpha,n,\varepsilon})=0,
\]
and the obstacle term can be written as
\begin{align*}
 & -\frac{2}{\varepsilon}\int_{0}^{T}\int_{x}\phi(t)(u_{\alpha,n,\varepsilon}-\psi)^{+}u_{\alpha,n,\varepsilon}\mathrm{d}t\\
 & =-\frac{2}{\varepsilon}\int_{0}^{T}\int_{x}\phi(t)|(u_{\alpha,n,\varepsilon}-\psi)^{+}|^{2}\mathrm{d}t-\frac{2}{\varepsilon}\int_{0}^{T}\int_{x}\phi(t)(u_{\alpha,n,\varepsilon}-\psi)^{+}\psi\mathrm{d}t.
\end{align*}
Moreover, with the definitions of $F_i$ given by (\ref{noise}), it holds that
\begin{align*}
 & \sum_{k=1}^{\infty}\int_{x}|\sigma_{n}^{\prime}(u_{\alpha,n,\varepsilon})\nabla_{x}u_{\alpha,n,\varepsilon}f_{k}+\sigma_{n}(u_{\alpha,n,\varepsilon})\nabla_{x}f_{k}|^{2}\\
 & =\int_{x}|\sigma_{n}^{\prime}(u_{\alpha,n,\varepsilon})\nabla_{x}u_{\alpha,n,\varepsilon}|^{2}F_{1}+2\int_{x}\sigma_{n}^{\prime}(u_{\alpha,n,\varepsilon})\sigma_{n}(u_{\alpha,n,\varepsilon})\nabla_{x}u_{\alpha,n,\varepsilon}\cdot F_{2} +\int_{x}|\sigma_{n}(u_{\alpha,n,\varepsilon})|^{2}F_{3},
\end{align*}
 we then have the cancellation between the It\^{o} and It\^{o}-to-Stratonovich corrections.
Then, taking expectation and using the definition of $\llbracket\sqrt{\Phi^{\prime}}\rrbracket$, we have \eqref{eq: energy perservation for appro solution}.

Regarding \eqref{eq=00FF1Apriori p}, applying the It\^{o}'s formula (cf. \cite[Theorem 3.1]{krylov2013relatively}), and using the cancellation between the It\^{o} and It\^{o}-to-Stratonovich corrections, we have almost surely,
\begin{align}
 & \frac{1}{p(p-1)}\int_{x}|u_{\alpha,n,\varepsilon}(x,T)+1|^{p}\label{eq: weak solution for nondegenerate equ-1-1} \\
 & =\int_{x}|u_{\mathrm{init}}+1|^{p}-\int_{0}^{T}\int_{x}|u_{\alpha,n,\varepsilon}+1|^{p-2}\Phi^{\prime}(u_{\alpha,n,\varepsilon})\nabla_{x}u_{\alpha,n,\varepsilon}\cdot\nabla_{x}u_{\alpha,n,\varepsilon}\mathrm{d}s\nonumber \\
 & \quad-\alpha\int_{0}^{T}\int_{x}|u_{\alpha,n,\varepsilon}+1|^{p-2}\nabla_{x}u_{\alpha,n,\varepsilon}\cdot\nabla_{x}u_{\alpha,n,\varepsilon}\mathrm{d}s+\int_0^{T}\int_{x}|u_{\alpha,n,\varepsilon}+1|^{p-2}\nabla_{x}u_{\alpha,n,\varepsilon}\cdot g(u_{\alpha,n,\varepsilon})\mathrm{d}s\nonumber\\
 & \quad-\frac{p-1}{\varepsilon}\int_{0}^{T}\int_{x}|u_{\alpha,n,\varepsilon}+1|^{p-2}(u_{\alpha,n,\varepsilon}-\psi)^{+}(u_{\alpha,n,\varepsilon}+1)\mathrm{d}s\nonumber
  \\ \notag
 & \quad-\frac{1}{2}\int_{0}^{T}\int_{x}\llbracket|\cdot+1|^{p-2}\sigma_{n}^{\prime}\sigma_{n}\rrbracket(u_{\alpha,n,\varepsilon})\nabla_{x}\cdot F_{2}\mathrm{d}s+\frac{1}{2}\int_{0}^{T}\int_{x}|u_{\alpha,n,\varepsilon}+1|^{p-2}|\sigma_{n}(u_{\alpha,n,\varepsilon})|^{2}F_{3}\mathrm{d}s\\
\nonumber
& \quad +\int_{0}^{T}\int_{x}|u_{\alpha,n,\varepsilon}+1|^{p-2}\sigma_{n}(u_{\alpha,n,\varepsilon})\nabla_{x}u_{\alpha,n,\varepsilon}\cdot\mathrm{d}\xi^{F}.
\end{align}
Note that
\[
\int_{x}|u_{\alpha,n,\varepsilon}+1|^{p-2}\nabla_{x}u_{\alpha,n,\varepsilon}\cdot g(u_{\alpha,n,\varepsilon})=\int_{x}\nabla_x\cdot\int_0^{u_{\alpha,n,\varepsilon}}|r+1|^{p-2} g(r)\mathrm{d}r=0,
\]
and
\[
|u_{\alpha,n,\varepsilon}+1|^{p-2}|\nabla_x\llbracket\sqrt{\Phi^\prime}\rrbracket(u_{\alpha,n,\varepsilon})|^2=|\nabla_x\llbracket|\cdot+1|^{(p-2)/2}\sqrt{\Phi^\prime}\rrbracket(u_{\alpha,n,\varepsilon})|^2.
\]
By using Assumption \ref{assu:assu for existence} for $\sigma$ (when $p=2$, we only use Assumption \ref{assu:assu for existence} 6), Lemma \ref{lem:estimates of sigma u}, Assumption \ref{assu:assum for F} for $F_i$, and \cite[Lemma 5.4]{fehrman2021wellarxiv}, we obtain for any
$\varsigma\in(0,1)$,
\begin{align*}
&-\frac{1}{2}\int_{0}^{T}\int_{x}\llbracket|\cdot+1|^{p-2}\sigma_{n}^{\prime}\sigma_{n}\rrbracket(u_{\alpha,n,\varepsilon})\nabla_{x}\cdot F_{2}\mathrm{d}s+\frac{1}{2}\int_{0}^{T}\int_{x}|u_{\alpha,n,\varepsilon}+1|^{p-2}|\sigma_{n}(u_{\alpha,n,\varepsilon})|^{2}F_{3}\mathrm{d}s\\
 & \leq C\int_{0}^{T}\int_{x}(1+|u_{\alpha,n,\varepsilon}|^{p-1}+\llbracket|\cdot+1|^{(p-2)/2}\sqrt{\Phi^{\prime}}\rrbracket^{2}(u_{\alpha,n,\varepsilon}))\mathrm{d}s\\
& \leq C(\varsigma)\int_0^T\big(1+\Vert u_{\alpha,n,\varepsilon}\Vert_{L^{p-1}(\mathbb{T}^d)}^{p-1}+\Vert u_{\alpha,n,\varepsilon}\Vert_{L^1(\mathbb{T}^d)}^{m+p-1})+\varsigma\int_0^T\int_x |u_{\alpha,n,\varepsilon}+1|^{p-2}|\nabla_{x}\llbracket\sqrt{\Phi^{\prime}}\rrbracket(u_{\alpha,n,\varepsilon})|^{2}\mathrm{d}s.
\end{align*}
For the obstacle term, it follows from $\psi\geq 0$ that
\begin{align*}
 & -\frac{p-1}{\varepsilon}\int_{0}^{T}\int_{x}|u_{\alpha,n,\varepsilon}+1|^{p-2}(u_{\alpha,n,\varepsilon}-\psi)^{+}(u_{\alpha,n,\varepsilon}+1)\mathrm{d}s\\
 & \leq-\frac{p-1}{\varepsilon}\int_{0}^{T}\int_{x}|u_{\alpha,n,\varepsilon}+1|^{p-2}|(u_{\alpha,n,\varepsilon}-\psi)^{+}|^{2}\mathrm{d}s\\
 & \quad-\frac{p-1}{\varepsilon}\int_{0}^{T}\int_{x}|u_{\alpha,n,\varepsilon}+1|^{p-2}(u_{\alpha,n,\varepsilon}-\psi)^{+}\mathrm{d}s.
\end{align*}
Based on the above results, by taking expectation and using the definition of $\llbracket\sqrt{\Phi^{\prime}}\rrbracket$ and \eqref{eq:L1 identity}, we reach
\begin{align*}
 & \mathbb{E}\Vert u_{\alpha,n,\varepsilon}(t)\Vert_{L^{p}(\mathbb{T}^{d})}^{p}+\mathbb{E}\int_{0}^{t}\int_{x}|u_{\alpha,n,\varepsilon}+1|^{p-2}|\nabla_{x}\llbracket\sqrt{\Phi^{\prime}}\rrbracket(u_{\alpha,n,\varepsilon})|^{2}\mathrm{d}s\\
 & +\alpha\mathbb{E}\int_{0}^{t}\int_{x}|u_{\alpha,n,\varepsilon}+1|^{p-2}|\nabla_{x}u_{\alpha,n,\varepsilon}|^{2}\mathrm{d}s+\mathbb{E}\int_{0}^{t}\int_{x}|u_{\alpha,n,\varepsilon}+1|^{p-2}\frac{1}{\varepsilon}|(u_{\alpha,n,\varepsilon}-\psi)^{+}|^{2}\mathrm{d}s\\
 & +\mathbb{E}\int_{0}^{t}\int_{x}|u_{\alpha,n,\varepsilon}+1|^{p-2}\frac{1}{\varepsilon}(u_{\alpha,n,\varepsilon}-\psi)^{+}\mathrm{d}s\\
 & \leq C\Big(1+\mathbb{E}\Vert u_{\mathrm{init}}\Vert_{L^{p}(\mathbb{T}^{d})}^{p}+\mathbb{E}\Vert u_{\mathrm{init}}\Vert_{L^{1}(\mathbb{T}^{d})}^{m+p-1}\Big).
\end{align*}
Then, by taking the supreme up to a time $T$, we complete the proof
of \eqref{eq=00FF1Apriori p}.
For \eqref{eq:priori}, the proof is similar, and we omit it here.

\end{proof}
\begin{rem}
\label{rem:nabla p-2 bracket}Combining \eqref{eq=00FF1Apriori p}, \eqref{eq:priori},  and \cite[Remark 3.1]{dareiotis2019entropy},
we have
\begin{align*}
 & \mathbb{E}\Vert\nabla_{x}\llbracket|\cdot|^{\frac{p-2}{2}}\sqrt{\Phi^{\prime}}\rrbracket(u_{\alpha,n,\varepsilon})\Vert_{L^{2}(Q_{T})}^{2}\\
 & =\mathbb{E}\Vert|u_{\alpha,n,\varepsilon}|^{\frac{p-2}{2}}\nabla_{x}\llbracket\sqrt{\Phi^{\prime}}\rrbracket(u_{\alpha,n,\varepsilon})\Vert_{L^{2}(Q_{T})}^{2}\\
 & \leq C\mathbb{E}\Vert|u_{\alpha,n,\varepsilon}+1|^{\frac{p-2}{2}}\nabla_{x}\llbracket\sqrt{\Phi^{\prime}}\rrbracket(u_{\alpha,n,\varepsilon})\Vert_{L^{2}(Q_{T})}^{2}+C\mathbb{E}\Vert\nabla_{x}\llbracket\sqrt{\Phi^{\prime}}\rrbracket(u_{\alpha,n,\varepsilon})\Vert_{L^{2}(Q_{T})}^{2}\\
 & \leq C\Big(1+\mathbb{E}\Vert u_{\mathrm{init}}(t)\Vert_{L^{p}(\mathbb{T}^{d})}^{p}+C\mathbb{E}\Vert u_{\mathrm{init}}\Vert_{L^{1}(\mathbb{T}^{d})}^{m+p-1}\Big).
\end{align*}
\end{rem}

\begin{cor}\label{cor-1}
For the lower obstacle problem, in which we require $u\geq\psi$ on $Q_T$ rather than $u\leq\psi$, the corresponding approximating penalized equation \eqref{eq:appro penalized eq} becomes
\begin{align}
\mathrm{d}u_{\alpha,n,\varepsilon} & =\Big[\Delta\Phi(u_{\alpha,n,\varepsilon})+\alpha\Delta u_{\alpha,n,\varepsilon}+\frac{1}{2}\nabla\cdot(F_{1}[\sigma_n^{\prime}(u_{\alpha,n,\varepsilon})]^{2}\nabla u_{\alpha,n,\varepsilon}+\sigma_n(u_{\alpha,n,\varepsilon})\sigma_n^{\prime}(u_{\alpha,n,\varepsilon})F_{2})\nonumber \\
 & \quad-\nabla \cdot g(u_{\alpha,n,\varepsilon})+\frac{1}{\varepsilon}(u_{\alpha,n,\varepsilon}-\psi)^{-}\Big]\mathrm{d}t-\nabla\cdot(\sigma_{n}(u_{\alpha,n,\varepsilon})\mathrm{d}\xi^{F}).\label{eq:appro penalized eq lower obstacle}
\end{align}
A priori estimates for the solution $u_{\alpha,n,\varepsilon}$ to \eqref{eq:appro penalized eq lower obstacle} remain valid under additional assumptions that $|\sigma(r)|\leq C(1+|r|)$ and $p=2$ using Gronwall inequality instead of \cite[Lemma 5.4]{fehrman2021wellarxiv}.
This includes the important Dean--Kawasaki equation
\begin{align*}
\mathrm{d}u & =\Delta\Phi(u)\mathrm{d}t+\nabla\cdot\sqrt{\Phi(u)}\circ\mathrm{d}\xi^{F},
\end{align*}
provided that $|\Phi(r)|\leq C(1+|r|^{2})$ for all $r\in\mathbb{R}$, and thus in particular covers the case $\Phi(r)=r$.
\end{cor}
To verify that the kinetic measure vanishes at infinity, we need the following a priori estimates.
\begin{lem}
\label{lem:(priori-estimates-for kinetic measure}(A priori estimates
for the kinetic measure) Let Assumption \ref{assu:assu for existence} and Assumption \ref{assu:assum for F} hold.
For any $n\in\mathbb{N}$
and $\alpha,\varepsilon\in(0,1)$, let $u_{\alpha,n,\varepsilon}$
be a solution to \eqref{eq:appro penalized eq}, then for all $N_{1},N_{2}\in(0,\infty)$
such that $N_{1}<N_{2}$, there exists a constant $C$ independent
of $n$, $\alpha$, $\varepsilon$, $N_{1}$ and $N_{2}$ such that
\begin{align}
 & \mathbb{E}\int_{x,t}\mathbf{1}_{\{u_{\alpha,n,\varepsilon}\in(N_{1},N_{2})\}}\Phi^{\prime}(u_{\alpha,n,\varepsilon})|\nabla_{x}u_{\alpha,n,\varepsilon}|^2+\alpha\mathbb{E}\int_{x,t}\mathbf{1}_{\{u_{\alpha,n,\varepsilon}\in(N_{1},N_{2})\}}|\nabla_{x}u_{\alpha,n,\varepsilon}|^2\label{eq:priori kinetic}\\
 & +\frac{1}{\varepsilon}\mathbb{E}\int_{x,t}\int_\theta|(u_{\alpha,n,\varepsilon}-\psi)^+|^2 \mathbf{1}_{\{u_{\alpha,n,\varepsilon}\in(\theta^{-1}(N_{1}-\psi)+\psi,\theta^{-1}(N_{2}-\psi)+\psi)\}}\nonumber\\
 & \leq C(N_2-N_1)\mathbb{E}\int_x (u_{\mathrm{init}}-N_1)^{+}
+C\mathbb{E}\int_{x,t}\mathbf{1}_{\{u_{\alpha,n,\varepsilon}\geq N_{1}\}}|\sigma_{n}(u_{\alpha,n,\varepsilon}\land N_{2})|^{2}.\nonumber
\end{align}
In particular, we have
\begin{align}
 & \lim_{N\rightarrow \infty}\Big[\mathbb{E}\int_{x,t}\mathbf{1}_{\{u_{\alpha,n,\varepsilon}\in(N,N+1)\}}\Phi^{\prime}(u_{\alpha,n,\varepsilon})|\nabla_{x}u_{\alpha,n,\varepsilon}|^2 +\alpha\mathbb{E}\int_{x,t}\mathbf{1}_{\{u_{\alpha,n,\varepsilon}\in(N,N+1)\}}|\nabla_{x}u_{\alpha,n,\varepsilon}|^2\nonumber\\
 & +\frac{1}{\varepsilon}\mathbb{E}\int_{x,t}\int_\theta|(u_{\alpha,n,\varepsilon}-\psi)^+|^2 \mathbf{1}_{\{u_{\alpha,n,\varepsilon}\in(\theta^{-1}(N-\psi)+\psi,\theta^{-1}(N+1-\psi)+\psi)\}}\label{eq:priori kinetic-1}\Big]=0.
 \end{align}
\end{lem}

\begin{proof}
Applying the It\^{o}'s formula (cf. \cite[Theorem 3.1]{krylov2013relatively}
or \cite[Proposition 5.7]{fehrman2024well}) to the function $S_{N}:[0,\infty)\to[0,\infty)$
satisfying $S_{N}^{\prime\prime}(\xi)=\mathbf{1}_{\{\xi\in(N_{1},N_{2})\}}$,
$S^{\prime}(0)=0$ and $S(0)=0$, by using the cancellation between the It\^{o} and It\^{o}-to-Stratonovich corrections, we have
\begin{align*}
 & \mathbb{E}\int_{x}S_{N}(u_{\alpha,n,\varepsilon}(x,T))\\
 & =\mathbb{E}\int_{x}S_{N}(u_{\mathrm{init}})-\mathbb{E}\int_{x,t}\mathbf{1}_{\{u_{\alpha,n,\varepsilon}\in(N_{1},N_{2})\}}\Phi^{\prime}(u_{\alpha,n,\varepsilon})|\nabla_{x}u_{\alpha,n,\varepsilon}|^2\\
 & \quad-\alpha\mathbb{E}\int_{x,t}\mathbf{1}_{\{u_{\alpha,n,\varepsilon}\in(N_{1},N_{2})\}}|\nabla_{x}u_{\alpha,n,\varepsilon}|^2+\mathbb{E}\int_{x,t}S_N^{\prime\prime}(u_{\alpha,n,\varepsilon})\nabla_x u_{\alpha,n,\varepsilon}\cdot g(u_{\alpha,n,\varepsilon})\\
 & \quad+\frac{1}{2}\mathbb{E}\int_{x,t}\mathbf{1}_{\{u_{\alpha,n,\varepsilon}\in(N_{1},N_{2})\}}[\sigma_{n}(u_{\alpha,n,\varepsilon})\sigma_{n}^{\prime}(u_{\alpha,n,\varepsilon})\nabla_{x}u_{\alpha,n,\varepsilon}F_2+\sigma^2_{n}(u_{\alpha,n,\varepsilon})F_3]\\
 & \quad-\frac{1}{\varepsilon}\mathbb{E}\int_{x,t}(u_{\alpha,n,\varepsilon}-\psi)^{+}S_{N}^{\prime}(u_{\alpha,n,\varepsilon}).
\end{align*}
Note that
\[
\int_{x,t}S_N^{\prime\prime}(u_{\alpha,n,\varepsilon})\nabla_x u_{\alpha,n,\varepsilon}\cdot g(u_{\alpha,n,\varepsilon})=\int_{x,t}\nabla_x\cdot\int_0^{u_{\alpha,n,\varepsilon}}S_N^{\prime\prime}(r)\cdot g(r)\mathrm{d}r=0.
\]
Since $S^\prime_N\geq0$, the obstacle term can be estimated as
\begin{align*}
 & -\frac{1}{\varepsilon}\int_{x,t}(u_{\alpha,n,\varepsilon}-\psi)^{+}S_{N}^{\prime}(u_{\alpha,n,\varepsilon})\\
 &=-\frac{1}{\varepsilon}\int_{x,t}(u_{\alpha,n,\varepsilon}-\psi)^{+}S_{N}^{\prime}(\psi)-\frac{1}{\varepsilon}\int_{x,t}(u_{\alpha,n,\varepsilon}-\psi)^{+}\big[S_{N}^{\prime}(u_{\alpha,n,\varepsilon})-S_N^\prime(\psi)\big]\\
 &\leq -\frac{1}{\varepsilon}\int_{x,t,\theta}|(u_{\alpha,n,\varepsilon}-\psi)^+|^2 S_N^{\prime\prime}(\psi+\theta (u_{\alpha,n,\varepsilon}-\psi))\\
  &= -\frac{1}{\varepsilon}\int_{x,t,\theta}|(u_{\alpha,n,\varepsilon}-\psi)^+|^2 \mathbf{1}_{\{u_{\alpha,n,\varepsilon}\in(\theta^{-1}(N_{1}-\psi)+\psi,\theta^{-1}(N_{2}-\psi)+\psi)\}}.
 \end{align*}
This, together with the distributional equality
\[
\mathbf{1}_{\{u_{\alpha,n,\varepsilon}\in(N_{1},N_{2})\}}\sigma_{n}(u_{\alpha,n,\varepsilon})\sigma_{n}^{\prime}(u_{\alpha,n,\varepsilon})\cdot\nabla_{x}u_{\alpha,n,\varepsilon}=\frac{1}{2}\nabla_{x}\Big(\sigma_{n}^{2}(N_{1}\lor(u_{\alpha,n,\varepsilon}\land N_{2}))-\sigma_{n}^{2}(N_{1})\Big),
\]
yields
\begin{align*}
 & \mathbb{E}\int_{x}S_{N}(u_{\alpha,n,\varepsilon}(x,T))+\frac{1}{\varepsilon}\mathbb{E}\int_{x,t}\int_\theta|(u_{\alpha,n,\varepsilon}-\psi)^+|^2 \mathbf{1}_{\{u_{\alpha,n,\varepsilon}\in(\theta^{-1}(N_{1}-\psi)+\psi,\theta^{-1}(N_{2}-\psi)+\psi)\}}\\
 & +\mathbb{E}\int_{x,t}\mathbf{1}_{\{u_{\alpha,n,\varepsilon}\in(N_{1},N_{2})\}}\Phi^{\prime}(u_{\alpha,n,\varepsilon})|\nabla_{x}u_{\alpha,n,\varepsilon}|^2+\alpha\mathbb{E}\int_{x,t}\mathbf{1}_{\{u_{\alpha,n,\varepsilon}\in(N_{1},N_{2})\}}|\nabla_{x}u_{\alpha,n,\varepsilon}|^2\\
 & \leq\mathbb{E}\int_{x}S_{N}(u_{\mathrm{init}})-\frac{1}{4}\int_{x,t}\Big(\sigma_{n}^{2}(N_{1}\lor(u_{\alpha,n,\varepsilon}\land N_{2}))-\sigma_{n}^{2}(N_{1})\Big)\nabla_{x}\cdot F_{2}\\
 & \quad+\frac{1}{2}\int_{x,t}\mathbf{1}_{\{u_{\alpha,n,\varepsilon}\in(N_{1},N_{2})\}}|\sigma_{n}(u_{\alpha,n,\varepsilon})|^{2}F_{3}.
\end{align*}
Owing to $S_N(\xi)\leq (N_2-N_1)(\xi-N_1)^+$,
the boundedness of $\nabla_{x}\cdot F_{2}$ and $F_{3}$, and $u_{\mathrm{init}}\leq\psi(\cdot,0)\leq M$,
we obtain \eqref{eq:priori kinetic}, which completes the proof.
\end{proof}
For each fixed $\varepsilon\in(0,1)$, the penalty term $\varepsilon^{-1}(u_{\alpha,n,\varepsilon}-\psi)^+$ is a Lipschitz term with linear growth on $\mathbb R_+$, and hence fits into the framework of \cite[Section 6]{fehrman2024well}.
We now introduce the following two existence theorems from \cite[Proposition 5.17 and Theorem 6.12]{fehrman2024well}, which will be adapted to our setting.
\begin{lem}
\label{lem:(Existence-of-solution for approximated smooth }(Existence of solutions to the approximated penalized equation)
Let Assumption \ref{assu:assu for existence} and Assumption \ref{assu:assum for F} hold.
For $n\in\mathbb{N}$ and $\alpha,\varepsilon\in(0,1)$, there exists a solution $u_{\alpha,n,\varepsilon}$ of \eqref{eq:appro penalized eq}, defined on the original probability space $\Omega$, in the sense of Definition \ref{def:Weak solution approximating penalized}.
Furthermore, let $\chi_{\alpha,n,\varepsilon}(x,\xi,t):=\mathbf{1}_{\{0<\xi<u_{\alpha,n,\varepsilon}(x,t)\}}$ be the kinetic function of $u_{\alpha,n,\varepsilon}$.
Then, $u_{\alpha,n,\varepsilon}$ satisfies that, almost surely for every $\phi\in C_{c}^{\infty}(\mathbb{T}^{d}\times(0,\infty)\times[0,T))$,
\begin{align}
& -\int_{0}^{T}\int_{\mathbb{R}}\int_{\mathbb{T}^d}\chi_{\alpha,n,\varepsilon}(x,\xi,s)\partial_s\phi(x,\xi,s)\mathrm{d}x\mathrm{d}\xi\mathrm{d}s\label{eq:kinetic for approximated penalty}\\
 & =\int_{\mathbb{R}}\int_{\mathbb{T}^{d}}\bar{\chi}(u_{\mathrm{init}},\xi)\phi(x,\xi,0)\mathrm{d}x\mathrm{d}\xi-\int_{0}^{T}\int_{\mathbb{T}^{d}}\Phi^{\prime}(u_{\alpha,n,\varepsilon})\nabla_x u_{\alpha,n,\varepsilon}\cdot(\nabla_x\phi)(x,u_{\alpha,n,\varepsilon},s)\mathrm{d}x\mathrm{d}s\nonumber \\
  & \quad-\alpha\int_{0}^{T}\int_{\mathbb{T}^d}\nabla_{x}u_{\alpha,n,\varepsilon}\cdot(\nabla_{x}\phi)(x,u_{\alpha,n,\varepsilon},s)\mathrm{d}x\mathrm{d}s\nonumber\\
 & \quad-\frac{1}{2}\int_{0}^{T}\int_{\mathbb{T}^{d}}\big(F_{1}(x)[\sigma_n^{\prime}(u_{\alpha,n,\varepsilon})]^{2}\nabla_x u_{\alpha,n,\varepsilon}+\sigma_n(u_{\alpha,n,\varepsilon})\sigma_n^{\prime}(u_{\alpha,n,\varepsilon})F_{2}(x)\big)\cdot(\nabla_x\phi)(x,u_{\alpha,n,\varepsilon},s)\mathrm{d}x\mathrm{d}s\nonumber\\
 & \quad-\int_{0}^{T}\int_{\mathbb{R}}\int_{\mathbb{T}^{d}}\partial_{\xi}\phi(x,\xi,s)\mathrm{d}m_{\alpha,n,\varepsilon}(x,\xi,s)-\int_0^{T}\int_{\mathbb{T}^d}\phi(x,u_{\alpha,n,\varepsilon},s)\nabla_x\cdot g(u_{\alpha,n,\varepsilon})\mathrm{d}x\mathrm{d}s\nonumber \\
 & \quad+\frac{1}{2}\int_{0}^{T}\int_{\mathbb{T}^{d}}\Big(\sigma_n(u_{\alpha,n,\varepsilon})\sigma_n^{\prime}(u_{\alpha,n,\varepsilon})\nabla_x u_{\alpha,n,\varepsilon}\cdot F_{2}(x)+\sigma_n^{2}(u_{\alpha,n,\varepsilon})F_{3}(x)\Big)(\partial_{\xi}\phi)(x,u_{\alpha,n,\varepsilon},s)\mathrm{d}x\mathrm{d}s\nonumber \\
 & \quad-\int_{0}^{T}\int_{\mathbb{T}^{d}}\phi(x,u_{\alpha,n,\varepsilon},s)\frac{1}{\varepsilon}(u_{\alpha,n,\varepsilon}-\psi)^+\mathrm{d}x\mathrm{d}s-\int_{0}^{T}\int_{\mathbb{T}^{d}}\phi(x,u_{\alpha,n,\varepsilon},s)\nabla_x\cdot(\sigma_n(u_{\alpha,n,\varepsilon})f_{k})\mathrm{d}x\mathrm{d}B_{s}^{k},\nonumber
\end{align}
where the kinetic measure $m_{\alpha,n,\varepsilon}$ is
\begin{align}\notag
m_{\alpha,n,\varepsilon}&:=\delta_{0}(\xi-u_{\alpha,n,\varepsilon})\Big(\Phi^{\prime}(\xi)|\nabla u_{\alpha,n,\varepsilon}|^{2}+\alpha|\nabla u_{\alpha,n,\varepsilon}|^{2}\Big).
\end{align}
It follows from  (\ref{eq:priori}) that the $\{m_{\alpha,n,\varepsilon}\}_{(\alpha,n)\in(0,1)\times\mathbb{N}}$ are finite kinetic measures.
\end{lem}
\begin{proof}
The proof is based on the standard Galerkin method (see e.g. \cite[Proposition 5.4]{dareiotis2020nonlinear}) and a smooth approximation of nonlinear function $\Phi$.
\end{proof}

\begin{lem}
\label{lem:(Existence-of-solution for penalized eqn }(Existence of
solutions to the penalized equation)\label{lem: existence of penalized eqn}
Let Assumption \ref{assu:assu for existence} and Assumption \ref{assu:assum for F} hold for some $p\in [2,\infty)$.
For fixed $\varepsilon\in(0,1)$, as $\alpha\to0$ and $n\to\infty$,
there exists a subsequence of $\{u_{\alpha,n,\varepsilon}\}_{(\alpha,n)\in(0,1)\times\mathbb{N}}$ that converges almost surely on the original probability space $(\Omega,\mathcal{F},\mathbb{P})$ in $L^{1}(\mathbb{T}^{d}\times[0,T])$ to some limit $u_{\varepsilon}$.
This limit $u_{\varepsilon}$ is a stochastic kinetic equation
of \eqref{eq:penalized eq} in the sense of Definition \ref{def:def of stochastic kinetic solution without obstacle}, meaning that the kinetic function $\chi_{\varepsilon}(x,\xi,t):=\mathbf{1}_{\{0<\xi<u_{\varepsilon}(x,t)\}}$ satisfies almost surely, for every $\phi\in C_{c}^{\infty}(\mathbb{T}^{d}\times(0,\infty)\times[0,T))$,
\begin{align}
& -\int_{0}^{T}\int_{\mathbb{R}}\int_{\mathbb{T}^d}\chi_{\varepsilon}(x,\xi,s)\partial_s\phi(x,\xi,s)\mathrm{d}x\mathrm{d}\xi\mathrm{d}s\label{eq:kinetic for penalty}\\
 & =\int_{\mathbb{R}}\int_{\mathbb{T}^{d}}\bar{\chi}(u_{\mathrm{init}},\xi)\phi(x,\xi,0)\mathrm{d}x\mathrm{d}\xi-\int_{0}^{T}\int_{\mathbb{T}^{d}}\Phi^{\prime}(u_{\varepsilon})\nabla_x u_{\varepsilon}\cdot(\nabla_x\phi)(x,u_{\varepsilon},s)\mathrm{d}x\mathrm{d}s\nonumber \\
 & \quad-\frac{1}{2}\int_{0}^{T}\int_{\mathbb{T}^{d}}\big(F_{1}(x)[\sigma^{\prime}(u_{\varepsilon})]^{2}\nabla_x u_\varepsilon+\sigma(u_{\varepsilon})\sigma^{\prime}(u_{\varepsilon})F_{2}(x)\big)\cdot(\nabla_x\phi)(x,u_{\varepsilon},s)\mathrm{d}x\mathrm{d}s\nonumber\\
 & \quad-\int_{0}^{T}\int_{\mathbb{R}}\int_{\mathbb{T}^{d}}\partial_{\xi}\phi(x,\xi,s)\mathrm{d}m_\varepsilon(x,\xi,s)-\int_0^{T}\int_{\mathbb{T}^d}\phi(x,u_\varepsilon,s)\nabla_x \cdot g(u_\varepsilon)\mathrm{d}x\mathrm{d}s\nonumber \\
 & \quad+\frac{1}{2}\int_{0}^{T}\int_{\mathbb{T}^{d}}\Big(\sigma(u_{\varepsilon})\sigma^{\prime}(u_{\varepsilon})\nabla_x u_\varepsilon\cdot F_{2}(x)+\sigma^{2}(u_{\varepsilon})F_{3}(x)\Big)(\partial_{\xi}\phi)(x,u_{\varepsilon},s)\mathrm{d}x\mathrm{d}s\nonumber \\
 & \quad-\int_{0}^{T}\int_{\mathbb{T}^{d}}\phi(x,u_{\varepsilon},s)\frac{1}{\varepsilon}(u_{\varepsilon}-\psi)^+\mathrm{d}x\mathrm{d}s-\int_{0}^{T}\int_{\mathbb{T}^{d}}\phi(x,u_{\varepsilon},s)\nabla_x\cdot(\sigma(u_{\varepsilon})f_{k})\mathrm{d}x\mathrm{d}B_{s}^{k},\nonumber
\end{align}
where the kinetic measure $m_{\varepsilon}\geq \delta_{0}(\xi-u_{\varepsilon})\Phi^{\prime}(\xi)|\nabla u_{\varepsilon}|^{2}$ is finite. Moreover, the following a priori estimates hold.
  \begin{align}
 & \sup_{t\in[0,T]}\mathbb{E}\Vert u_{\varepsilon}(t)\Vert_{L^{p}(\mathbb{T}^{d})}^{p}+\mathbb{E}\Vert|u_{\varepsilon}+1|^{\frac{p-2}{2}}\nabla_{x}\llbracket\sqrt{\Phi^{\prime}}\rrbracket(u_{\varepsilon})\Vert_{L^{2}(Q_{T})}^{2}\label{eq=00FF1Apriori p varepsilon} \\
 & +\mathbb{E}\Vert|u_{\varepsilon}+1|^{\frac{p-2}{2}}\varepsilon^{-1}(u_{\varepsilon}-\psi)^{+}\Vert_{L^{2}(Q_{T})}^{2}+\mathbb{E}\Vert|u_{\varepsilon}+1|^{p-2}\varepsilon^{-1}(u_{\varepsilon}-\psi)^{+}\Vert_{L^{1}(Q_{T})} \nonumber\\
 & \leq C\Big(1+\mathbb{E}\Vert u_{\mathrm{init}}\Vert_{L^{p}(\mathbb{T}^{d})}^{p}+\mathbb{E}\Vert u_{\mathrm{init}}\Vert_{L^{1}(\mathbb{T}^{d})}^{m+p-1}\Big),\nonumber
\end{align}
  \begin{align}
 & \sup_{t\in[0,T]}\mathbb{E}\Vert u_{\varepsilon}(t)\Vert_{L^{2}(\mathbb{T}^{d})}^{2}+\mathbb{E}\Vert\nabla_{x}\llbracket\sqrt{\Phi^{\prime}}\rrbracket(u_{\varepsilon})\Vert_{L^{2}(Q_{T})}^{2}+\mathbb{E}m_\varepsilon(\mathbb{T}^d\times\mathbb{R}\times[0,T]) \label{eq:priori for u epsilon}\\
 & \quad+\varepsilon^{-1}\mathbb{E}\Vert(u_{\varepsilon}-\psi)^{+}\Vert_{L^{2}(Q_{T})}^{2}+\mathbb{E}\Vert\varepsilon^{-1}(u_{\varepsilon}-\psi)^{+}\Vert_{L^{1}(Q_{T})}\nonumber\\
 & \leq C\Big(1+\mathbb{E}\Vert u_{\mathrm{init}}\Vert_{L^{2}(\mathbb{T}^{d})}^{2}+\mathbb{E}\Vert u_{\mathrm{init}}\Vert_{L^{1}(\mathbb{T}^{d})}^{m+1}\Big),\nonumber
\end{align}
and
  \begin{align}
 & \mathbb{E}m_\varepsilon(\mathbb{T}^d\times(N_1,N_1+1)\times[0,T])+\frac{1}{\varepsilon}\mathbb{E}\int_{x,t}\int_\theta|(u_{\varepsilon}-\psi)^+|^2 \mathbf{1}_{\{u_{\varepsilon}\in(\theta^{-1}(N_{1}-\psi)+\psi,\theta^{-1}(N_{2}-\psi)+\psi)\}}  \nonumber\\
 &\leq C(N_2-N_1)\mathbb
 {E}\int_x(u_{\mathrm{init}}-N_1)^++C\mathbb{E}\int_{x,t}\mathbf{1}_{\{u_{\varepsilon}\geq N_{1}\}}|\sigma(u_{\varepsilon}\land N_2)|^{2}.\label{eq:penalty kinetic priori esti}
 \end{align}
\end{lem}
\begin{proof}
For fixed $\varepsilon\in(0,1)$, the penalty term $\varepsilon^{-1}(u_{\alpha,n,\varepsilon}-\psi)^+$ of \eqref{eq:penalized eq} is a Lipschitz function satisfying $|\varepsilon^{-1}(\xi-\psi)^+|\leq c\xi$ for $\xi\in\mathbb{R}^+$. Hence, we can adapt the approach from \cite[Section 6]{fehrman2024well} to establish  the existence of stochastic kinetic solutions to our equation. In the sequel, we outline the main details.

The core step in constructing a solution is to prove the tightness of the laws of $\{u_{\alpha,n,\varepsilon}\}_{\alpha,n}$ on the space $L^1(0,T;L^1(\mathbb{T}^d))$.
As discussed in \cite{fehrman2024well}, the singularity at zero in the It\^{o} correction term prevents direct time regularity estimate for $u_{\alpha,n,\varepsilon}$.
Instead, by introducing a function $\Psi_{\delta}$ keeping the solution away from zero, the tightness of the laws of $\{\Psi_\delta(u_{\alpha,n,\varepsilon})\}_{\alpha,n}$ on $L^1(0,T;L^1(\mathbb{T}^d))$ is established by stable $W^{\beta,1}(0,T;H^{-s}(\mathbb{T}^{d})-$estimate and the Aubins-Lions-Simon lemma (see e.g., \cite{simon1986compact}), Lemma \ref{lem:priori-estimates} and \cite[Lemma 5.11]{fehrman2021wellarxiv}.
This implies the tightness of the laws of $\{u_{\alpha,n,\varepsilon}\}_{\alpha,n}$ on a new metric on $L^1(0,T;L^1(\mathbb{T}^d))$ defined via the nonlinear function $\Psi_\delta$ (see \cite[Definition 5.19]{fehrman2024well}).
Furthermore, by \cite[Lemma 5.20]{fehrman2024well}, this new metric topology is equal to the strong topology on $L^1(0,T;L^1(\mathbb{T}^d))$, which yields  the tightness of the laws of $u_{\alpha,n,\varepsilon}$ on $L^1(0,T;L^1(\mathbb{T}^d))$ equipped with the strong topology.
Applying Prokhorov's theorem, Skorokhod representation theorem, and Gy\"{o}ngy-Krylov method, we extract a (non-relabeled) subsequence of $u_{\alpha,n,\varepsilon}$ that converges to $u_\varepsilon$ in $L^1(0,T;L^1(\mathbb{T}^d))$ as $\alpha\to0$ and $n\to\infty$.
The desired estimates (\ref{eq=00FF1Apriori p varepsilon})-(\ref{eq:penalty kinetic priori esti}) and $m_{\varepsilon}\geq \delta_{0}(\xi-u_{\varepsilon})\Phi^{\prime}(\xi)|\nabla u_{\varepsilon}|^{2}$ follow from Lemma \ref{lem:priori-estimates}, Lemma \ref{lem:(priori-estimates-for kinetic measure}, the weak lower-semicontinuity of the Sobolev norms, and the strong convergence of $u_{\alpha,n,\varepsilon}$ to $u_\varepsilon$.
By passing to a further non-relabeled subsequence, we obtain the almost sure convergence of $u_{\alpha,n,\varepsilon}$. Testing \eqref{eq:L1 identity} against $\phi\in C_c^\infty([0,T))$, by integration by parts in time, and then passing to the limit, it yields \eqref{eq:L1 perservation without obstacle} in Definition \ref{def:def of stochastic kinetic solution without obstacle}.
By the almost sure convergence $u_{\alpha,n,\varepsilon}\to u_\varepsilon$ on $\Omega\times Q_T$, the convergence
$\sigma_k\to\sigma$ in $C^1_{\mathrm{loc}}((0,\infty))$,
and the uniform (in $n$) version of condition 8 in Assumption \ref{assu:assu for existence}, we have $\sigma_n^2(u_{\alpha,n,\varepsilon})\to \sigma^2(u_\varepsilon)$ a.e. on $\Omega\times Q_T$.
Moreover, the uniform (in $n$) condition 5 in Assumption \ref{assu:assu for existence}, together with the a priori estimates \eqref{eq=00FF1Apriori p varepsilon} with $p$ sufficiently large yields, for some $\eta>0$,
\[
\sup_{\alpha, n}\mathbb E\int_{Q_T}|\sigma_n^2(u_{\alpha,n,\varepsilon})|^{1+\eta}\mathrm{d}x\mathrm{d}t<\infty.
\]
Hence $\{\sigma_n^2(u_{\alpha,n})\}_{\alpha,n}$ is uniformly integrable in $L^1(\Omega\times Q_T)$. By Vitali's theorem, we have $
\sigma_n^2(u_{\alpha,n,\varepsilon})\to \sigma^2(u_\varepsilon)$ strongly in $L^1(\Omega\times Q_T)$.
Combining the almost sure convergence of $u_{\alpha,n,\varepsilon}$ and the weak convergence $m_{\alpha,n,\varepsilon}$, we have the preservation of energy \eqref{eq:energy preservation for penalty solution}. Since the proofs parallel those of \cite[Proposition 5.17 and Theorem 6.12]{fehrman2024well}, we omit them for brevity.
\end{proof}

Next, we aim to pass the limit $\varepsilon\to 0$. Since the Lipschitz constant of the penalization term ${(u_\varepsilon - \psi)^+}/{\varepsilon}$ diverges to infinity as $\varepsilon\to 0$,  a comparison theorem is further required to validate the convergence of $u_\varepsilon$. This result not only demonstrates that $u_\varepsilon$ is monotone in $\varepsilon$, but also ensures the uniqueness of the solution for each fixed $\varepsilon > 0$.
\begin{lem}
(Monotonicity)\label{lem:(Comparison-theorem)} Let Assumption \ref{assu:assu for existence} and Assumption \ref{assu:assum for F} hold.
For any $n\in\mathbb{N}$
and $0<\varepsilon_{1}\leq\varepsilon_{2}<1$, let $u_{\varepsilon_{1}}$
and $u_{\varepsilon_{2}}$ be solutions of \eqref{eq:penalized eq} in the sense of Definition \ref{def:def of stochastic kinetic solution without obstacle}
with penalized coefficients $\varepsilon_{1}$ and $\varepsilon_{2}$,
respectively. Then, we have almost surely for $(x,t)\in Q_{T}$,
\[
u_{\varepsilon_{1}}\leq u_{\varepsilon_{2}}.
\]
\end{lem}

\begin{proof}
For simplicity, we set $u_{i}:=u_{\varepsilon_{i}}$ and $m_{i}:=m_{\varepsilon_{i}}$,
$i\in\{1,2\}$. Let $\chi_{i}$ be the kinetic function of $u_{i}$, $i=1,2$.
The proof relies on the following identity:
\begin{align*}
\int_{x}(u_{1}-u_{2})^{+}(t) & =\frac{1}{2}\int_{x}|u_{1}-u_{2}|(t)+\frac{1}{2}\int_{x}(u_{1}-u_{2})(t)\\
 & =\frac{1}{2}\int_{x,\xi}|\chi_{1}-\chi_{2}|^{2}(t)+\frac{1}{2}\int_{x,\xi}(\chi_{1}-\chi_{2})(t)\\
 & =\int_{x,\xi}\chi_{1}(1-\chi_{2})(t).
\end{align*}
We now estimate each term separately. The notations $\chi_{i,t}^{\varsigma,\delta}(y,\eta)$ and $\bar{\kappa}_{i,t}^{\varsigma,\delta}(x,y,\eta)$ are adopted from Theorem \ref{thm:stability L1}.
By taking the test function in \eqref{eq:kinetic equation test with time without obstacle}
which converges to $\alpha_{\epsilon,t}(\tau)\phi(x,\xi)$, where $\phi(x,\xi)=\kappa^{\varsigma,\delta}(x,y,\xi,\eta)$, and $\alpha_{\epsilon,t}$ is defined by
\begin{equation*}
\alpha_{\epsilon,t}(\tau):=\begin{cases}
1,&\tau\leq t,\\
1-({\tau-t})/{\epsilon},&t\leq \tau\leq t+\epsilon,\\
0,&\tau\geq t+\epsilon.
\end{cases}
\end{equation*}
Then, taking the limit $\epsilon\to0$ and using the Lebesgue dominated convergence theorem, with the continuity of $u_i$ in $L^1(\mathbb{T}^d)$, there exists a subset of full probability
such that, for every $i\in\{1,2\}$, $(y,\eta)\in\mathbb{T}^{d}\times(\delta/2,\infty)$,
and $t\in[0,T]$,
\begin{align*}
  \left.\chi_{i,s}^{\varsigma,\delta}(y,\eta)\right|_{s=0}^{t}
 & =\nabla_{y}\bigg(\int_{0}^{t}\int_{x}\Phi^{\prime}(u_{i})\nabla_{x}u_{i}\bar{\kappa}_{i,s}^{\varsigma,\delta}(x,y,\eta)\mathrm{d}s\bigg)\\
 & \quad+\frac{1}{2}\nabla_{y}\bigg(\int_{0}^{t}\int_{x}\big(F_{1}[\sigma^{\prime}(u_{i})]^{2}\nabla_{x}u_{i}+\sigma(u_{i})\sigma^{\prime}(u_{i})F_{2}\big)\bar{\kappa}_{i,s}^{\varsigma,\delta}(x,y,\eta)\mathrm{d}s\bigg)\\
 & \quad+\partial_{\eta}\bigg(\int_{0}^{t}\int_{\mathbb{R}}\int_{\mathbb{T}^{d}}\kappa^{\varsigma,\delta}(x,y,\xi,\eta)\mathrm{d}m_{i}(x,\xi,s)\bigg)\\
 &\quad-\int_0^{t}\int_x\bar{\kappa}_{i,s}^{\varsigma,\delta}(x,y,\eta)\nabla_x\cdot g(u_i)\mathrm{d}s\\
 & \quad-\frac{1}{2}\partial_{\eta}\bigg(\int_{0}^{t}\int_{x}\Big(\sigma(u_{i})\sigma^{\prime}(u_{i})\nabla_{x}u_{i}\cdot F_{2}+\sigma^{2}(u_{i})F_{3}\Big)\bar{\kappa}_{i,s}^{\varsigma,\delta}(x,y,\eta)\mathrm{d}s\bigg)\\
 & \quad-\frac{1}{\varepsilon_i}\int_{0}^{t}\int_{x}(u_{i}-\psi)^{+}\bar{\kappa}_{i,s}^{\varsigma,\delta}(x,y,\eta)\mathrm{d}s-\int_{0}^{t}\int_{x}\bar{\kappa}_{i,s}^{\varsigma,\delta}(x,y,\eta)\nabla_x\cdot(\sigma(u_{i})f_{k})\mathrm{d}B_{s}^{k},
\end{align*}
where $u_{i}=u_{i}(x,s)$, $\psi=\psi(x,s)$ and $f_{k}=f_{k}(x)$.
Then, for every $\varsigma,\beta\in(0,1)$, $\delta\in(0,\beta/4)$,
$t\in[0,T]$, and $N\in\mathbb{N}$, we have
\begin{align}
\int_{y,\eta}\chi_{1,s}^{\varsigma,\delta}(y,\eta)\varphi_{\beta}(\eta)\zeta_{N}(\eta)\Big|_{s=0}^{t} & =I_{1,t}^{\mathrm{obs}}+I_{1,t}^{\mathrm{cut}}+I_{1,t}^{\mathrm{mart}}+I_{1,t}^{\mathrm{cons}},\label{eq:first order term-1}
\end{align}
where the obstacle term is given by
\begin{align*}
I_{1,t}^{\mathrm{obs}} & :=-\frac{1}{\varepsilon_{1}}\int_{y,\eta,x}\varphi_{\beta}(\eta)\zeta_{N}(\eta)\int_{0}^{t}(u_{1}-\psi)^{+}\bar{\kappa}_{1,s}^{\varsigma,\delta}(x,y,\eta)\mathrm{d}s,
\end{align*}
the cut term is given by
\begin{align*}
I_{1,t}^{\mathrm{cut}} & :=-\int_{y,\eta}\partial_{\eta}\big(\varphi_{\beta}(\eta)\zeta_{N}(\eta)\big)\int_{0}^{t}\int_{\mathbb{R}}\int_{\mathbb{T}^{d}}\kappa^{\varsigma,\delta}(x,y,\xi,\eta)\mathrm{d}m_{1}(x,\xi,s)\\
 & \quad+\frac{1}{2}\int_{y,\eta,x}\partial_{\eta}\big(\varphi_{\beta}(\eta)\zeta_{N}(\eta)\big)\int_{0}^{t}\Big(\sigma(u_{1})\sigma^{\prime}(u_{1})\nabla_{x}u_{1}\cdot F_{2}+\sigma^{2}(u_{1})F_{3}\Big)\bar{\kappa}_{1,s}^{\varsigma,\delta}(x,y,\eta)\mathrm{d}s,
\end{align*}
the martingale term is given by
\begin{align*}
I_{1,t}^{\mathrm{mart}} & :=-\int_{y,\eta,x}\varphi_{\beta}(\eta)\zeta_{N}(\eta)\int_{0}^{t}\bar{\kappa}_{1,s}^{\varsigma,\delta}(x,y,\eta)\nabla_{x}\cdot(\sigma(u_{1})f_{k})\mathrm{d}B_{s}^{k},
\end{align*}
and the conservative term is given by
\begin{align*}
I_{1,t}^{\mathrm{cons}} & :=-\int_{y,\eta,x}\varphi_{\beta}(\eta)\zeta_{N}(\eta)\int_{0}^{t}\bar{\kappa}_{1,s}^{\varsigma,\delta}(x,y,\eta)\nabla_{x}\cdot g(u_{1})\mathrm{d}s,
\end{align*}
where $u_{1}=u_{1}(x,s)$, $\psi=\psi(x,s)$
and $f_{k}=f_k(x)$. Moreover, using Lemma \ref{lem:intergrating by part}, \eqref{eq:partial eta} and It\^o's formula, we have
\begin{align}
 & \int_{y,\eta}\left.\chi_{1,s}^{\varsigma,\delta}(y,\eta)\chi_{2,s}^{\varsigma,\delta}(y,\eta)\varphi_{\beta}(\eta)\zeta_{N}(\eta)\right|_{s=0}^{t}\label{eq:second order term-1} \\
 & =\int_{0}^{t}\int_{y,\eta}\Big[\chi_{2,s}^{\varsigma,\delta}(y,\eta)\mathrm{d}\chi_{1,s}^{\varsigma,\delta}(y,\eta)+\chi_{1,s}^{\varsigma,\delta}(y,\eta)\mathrm{d}\chi_{2,s}^{\varsigma,\delta}(y,\eta)+\mathrm{d}\langle\chi_{2,\cdot}^{\varsigma,\delta},\chi_{1,\cdot}^{\varsigma,\delta}\rangle_{s}(y,\eta)\Big]\varphi_{\beta}(\eta)\zeta_{N}(\eta).\nonumber
\end{align}
For every $\varsigma,\beta\in(0,1)$, $\delta\in(0,\beta/4)$, $t\in[0,T]$
and $N\in\mathbb{N}$, we have
\[
\int_{0}^{t}\int_{y,\eta}\chi_{2,s}^{\varsigma,\delta}(y,\eta)\mathrm{d}\chi_{1,s}^{\varsigma,\delta}(y,\eta)\varphi_{\beta}(\eta)\zeta_{N}(\eta)=I_{2,1,t}^{\mathrm{obs}}+I_{2,1,t}^{\mathrm{err}}+I_{2,1,t}^{\mathrm{meas}}+I_{2,1,t}^{\mathrm{cut}}+I_{2,1,t}^{\mathrm{mart}}+I_{2,1,t}^{\mathrm{cons}},
\]
where the obstacle term is given by
\begin{align*}
I_{2,1,t}^{\mathrm{obs}} & :=-\frac{1}{\varepsilon_{1}}\int_{y,\eta,x}\int_{0}^{t}\chi_{2,s}^{\varsigma,\delta}(y,\eta)\varphi_{\beta}(\eta)\zeta_{N}(\eta)(u_{1}-\psi)^{+}\bar{\kappa}_{1,s}^{\varsigma,\delta}\mathrm{d}s,
\end{align*}
the error term is given by
\begin{align*}
I_{2,1,t}^{\mathrm{err}} & :=-\int_{y,\eta,x,z}\int_{0}^{t}\Phi^{\prime}(u_{1})\nabla_{x}u_{1}\cdot\nabla_{z}u_{2}\bar{\kappa}_{2,s}^{\varsigma,\delta}\bar{\kappa}_{1,s}^{\varsigma,\delta}\varphi_{\beta}(\eta)\zeta_{N}(\eta)\mathrm{d}s\\
 & \quad+\int_{y,\eta,x,z}\int_{0}^{t}|\Phi^{\prime}(u_{1})|^{\frac{1}{2}}|\Phi^{\prime}(u_{2})|^{\frac{1}{2}}\nabla_{x}u_{1}\cdot\nabla_{z}u_{2}\bar{\kappa}_{2,s}^{\varsigma,\delta}\bar{\kappa}_{1,s}^{\varsigma,\delta}\varphi_{\beta}(\eta)\zeta_{N}(\eta)\mathrm{d}s\\
 & \quad-\frac{1}{2}\int_{y,\eta,x,z}\int_{0}^{t}\nabla_{z}u_{2}\cdot\big(F_{1}[\sigma^{\prime}(u_{1})]^{2}\nabla_{x}u_{1}+\sigma(u_{1})\sigma^{\prime}(u_{1})F_{2}\big)\bar{\kappa}_{1,s}^{\varsigma,\delta}\bar{\kappa}_{2,s}^{\varsigma,\delta}\varphi_{\beta}(\eta)\zeta_{N}(\eta)\mathrm{d}s\\
 & \quad-\frac{1}{2}\int_{y,\eta,x,z}\int_{0}^{t}\Big(\sigma(u_{1})\sigma^{\prime}(u_{1})\nabla_{x}u_{1}\cdot F_{2}+\sigma^{2}(u_{1})F_{3}\Big)\bar{\kappa}_{1,s}^{\varsigma,\delta}\bar{\kappa}_{2,s}^{\varsigma,\delta}\varphi_{\beta}(\eta)\zeta_{N}(\eta)\mathrm{d}s,
\end{align*}
the measure term is given by
\begin{align*}
I_{2,1,t}^{\mathrm{meas}} & :=\int_{y,\eta}\int_{0}^{t}\bar{\kappa}_{2,s}^{\varsigma,\delta}\varphi_{\beta}(\eta)\zeta_{N}(\eta)\int_{\mathbb{R}}\int_{\mathbb{T}^{d}}\kappa^{\varsigma,\delta}(x,y,\xi,\eta)\mathrm{d}m_{1}(x,\xi,s)\\
 & \quad-\int_{y,\eta,x,z}\int_{0}^{t}|\Phi^{\prime}(u_{1})|^{\frac{1}{2}}|\Phi^{\prime}(u_{2})|^{\frac{1}{2}}\nabla_{x}u_{1}\cdot\nabla_{z}u_{2}\bar{\kappa}_{2,s}^{\varsigma,\delta}\bar{\kappa}_{1,s}^{\varsigma,\delta}\varphi_{\beta}(\eta)\zeta_{N}(\eta)\mathrm{d}s,
\end{align*}
the cutoff term is given by
\begin{align*}
I_{2,1,t}^{\mathrm{cut}} & :=-\int_{y,\eta}\int_{0}^{t}\int_{\mathbb{R}}\int_{\mathbb{T}^{d}}\chi_{2,s}^{\varsigma,\delta}(y,\eta)\partial_{\eta}\big(\varphi_{\beta}(\eta)\zeta_{N}(\eta)\big)\kappa^{\varsigma,\delta}(x,y,\xi,\eta)\mathrm{d}m_{1}(x,\xi,s)\\
 & \quad+\frac{1}{2}\int_{y,\eta,x}\int_{0}^{t}\chi_{2,s}^{\varsigma,\delta}(y,\eta)\partial_{\eta}\big(\varphi_{\beta}(\eta)\zeta_{N}(\eta)\big)\Big(\sigma(u_{1})\sigma^{\prime}(u_{1})\nabla_{x}u_{1}\cdot F_{2}+\sigma^{2}(u_{1})F_{3}\Big)\bar{\kappa}_{1,s}^{\varsigma,\delta}\mathrm{d}s,
\end{align*}
the martingale term is given by
\[
I_{2,1,t}^{\mathrm{mart}}:=-\int_{y,\eta,x}\int_{0}^{t}\chi_{2,s}^{\varsigma,\delta}(y,\eta)\varphi_{\beta}(\eta)\zeta_{N}(\eta)\bar{\kappa}_{1,s}^{\varsigma,\delta}\nabla_{x}\cdot(\sigma(u_{1})f_{k})\mathrm{d}B_{s}^{k},
\]
and the conservative term is given by
\[
I_{2,1,t}^{\mathrm{cons}}:=-\int_{y,\eta,x}\int_{0}^{t}\chi_{2,s}^{\varsigma,\delta}(y,\eta)\varphi_{\beta}(\eta)\zeta_{N}(\eta)\bar{\kappa}_{1,s}^{\varsigma,\delta}\nabla_{x} \cdot g(u_{1})\mathrm{d}s,
\]
where we used the notations $u_{1}=u_{1}(x,s)$, $u_{2}=u_{2}(z,s)$,
$\bar{\kappa}_{1,s}^{\varsigma,\delta}=\bar{\kappa}_{1,s}^{\varsigma,\delta}(x,y,\eta)$,
and $\bar{\kappa}_{2,s}^{\varsigma,\delta}=\bar{\kappa}_{2,s}^{\varsigma,\delta}(z,y,\eta)$. A similar equation holds for the term
$
\int_{0}^{t}\int_{y,\eta}\chi_{1,s}^{\varsigma,\delta}(y,\eta)\mathrm{d}\chi_{2,s}^{\varsigma,\delta}(y,\eta)\varphi_{\beta}(\eta)\zeta_{N}(\eta).
$

For the cross term in \eqref{eq:second order term-1}, we have
\begin{align}
 & \int_{0}^{t}\int_{y,\eta}\mathrm{d}\langle\chi_{2,\cdot}^{\varsigma,\delta},\chi_{1,\cdot}^{\varsigma,\delta}\rangle_{s}(y,\eta)\varphi_{\beta}(\eta)\zeta_{N}(\eta) \label{eq:corss term-1}\\
 & =\int_{y,\eta,x,z}\int_{0}^{t}f_{k}(x)f_{k}(z)\sigma^{\prime}(u_{1})\sigma^{\prime}(u_{2})\nabla_{x}u_{1}\cdot\nabla_{z}u_{2}\varphi_{\beta}(\eta)\zeta_{N}(\eta)\bar{\kappa}_{1,s}^{\varsigma,\delta}\bar{\kappa}_{2,s}^{\varsigma,\delta}\mathrm{d}s\nonumber \\
 & \quad+\int_{y,\eta,x,z}\int_{0}^{t}f_{k}(x)\nabla_{z}f_{k}(z)\cdot\nabla_{x}u_{1}\sigma^{\prime}(u_{1})\sigma(u_{2})\varphi_{\beta}(\eta)\zeta_{N}(\eta)\bar{\kappa}_{1,s}^{\varsigma,\delta}\bar{\kappa}_{2,s}^{\varsigma,\delta}\mathrm{d}s\nonumber\\
 & \quad+\int_{y,\eta,x,z}\int_{0}^{t}f_{k}(z)\nabla_{x}f_{k}(x)\cdot\nabla_{z}u_{2}\sigma(u_{1})\sigma^{\prime}(u_{2})\varphi_{\beta}(\eta)\zeta_{N}(\eta)\bar{\kappa}_{1,s}^{\varsigma,\delta}\bar{\kappa}_{2,s}^{\varsigma,\delta}\mathrm{d}s\nonumber \\
 & \quad+\int_{y,\eta,x,z}\int_{0}^{t}\nabla_{x}f_{k}(x)\cdot\nabla_{z}f_{k}(z)\sigma(u_{1})\sigma(u_{2})\varphi_{\beta}(\eta)\zeta_{N}(\eta)\bar{\kappa}_{1,s}^{\varsigma,\delta}\bar{\kappa}_{2,s}^{\varsigma,\delta}\mathrm{d}s.\nonumber
\end{align}
Therefore,
we reach
\begin{equation}
\int_{y,\eta}\left.\chi_{1,s}^{\varsigma,\delta}(y,\eta)\chi_{2,s}^{\varsigma,\delta}(y,\eta)\varphi_{\beta}(\eta)\zeta_{N}(\eta)\right|_{s=0}^{t}=I_{\mathrm{mix},t}^{\mathrm{obs}}+I_{\mathrm{mix},t}^{\mathrm{err}}+I_{\mathrm{mix},t}^{\mathrm{meas}}+I_{\mathrm{mix},t}^{\mathrm{cut}}+I_{\mathrm{mix},t}^{\mathrm{mart}}+I_{\mathrm{mix},t}^{\mathrm{cons}},\label
{eq:second term all-1}
\end{equation}
where
\[
I_{\mathrm{mix},t}^{\mathrm{fun}}=I_{2,1,t}^{\mathrm{fun}}+I_{1,2,t}^{\mathrm{fun}},
\]
for ``fun'' taking values in $\{\mathrm{obs,meas,cut,mart,cons}\}$. For the error
terms, which combine $I_{2,1,t}^{\mathrm{err}}$, $I_{1,2,t}^{\mathrm{err}}$
and \eqref{eq:corss term-1}, are
\begin{align*}
 & I_{\mathrm{mix},t}^{\mathrm{err}}\\
 & :=-\int_{y,\eta,x,z}\int_{0}^{t}\big[|\Phi^{\prime}(u_{1})|^{\frac{1}{2}}-|\Phi^{\prime}(u_{2})|^{\frac{1}{2}}\big]^{2}\nabla_{x}u_{1}\cdot\nabla_{z}u_{2}\bar{\kappa}_{2,s}^{\varsigma,\delta}\bar{\kappa}_{1,s}^{\varsigma,\delta}\varphi_{\beta}\zeta_{N}\mathrm{d}s\\
 & \quad-\frac{1}{2}\int_{y,\eta,x,z}\int_{0}^{t}\nabla_{z}u_{2}\cdot\nabla_{x}u_{1}\big(F_{1}(x)[\sigma^{\prime}(u_{1})]^{2}+F_{2}(z)[\sigma^{\prime}(u_{2})]^{2}\\
 & \qquad\qquad\qquad\qquad-2f_{k}(x)f_{k}(z)\sigma^{\prime}(u_{1})\sigma^{\prime}(u_{2})\big)\bar{\kappa}_{1,s}^{\varsigma,\delta}\bar{\kappa}_{2,s}^{\varsigma,\delta}\varphi_{\beta}\zeta_{N}\mathrm{d}s\\
 & \quad-\frac{1}{2}\int_{y,\eta,x,z}\int_{0}^{t}\nabla_{z}u_{2}\cdot\big(F_{2}(x)\sigma(u_{1})\sigma^{\prime}(u_{1})+F_{2}(z)\sigma(u_{2})\sigma^{\prime}(u_{2})\\
 & \qquad\qquad\qquad\qquad-2f_{k}(z)\nabla_{x}f_{k}(x)\sigma(u_{1})\sigma^{\prime}(u_{2})\big)\bar{\kappa}_{1,s}^{\varsigma,\delta}\bar{\kappa}_{2,s}^{\varsigma,\delta}\varphi_{\beta}\zeta_{N}\mathrm{d}s\\
 & \quad-\frac{1}{2}\int_{y,\eta,x,z}\int_{0}^{t}\nabla_{x}u_{1}\cdot\big(F_{2}(x)\sigma(u_{1})\sigma^{\prime}(u_{1})+F_{2}(z)\sigma(u_{2})\sigma^{\prime}(u_{2})\\
 & \qquad\qquad\qquad\qquad-2f_{k}(x)\nabla_{z}f_{k}(z)\sigma^{\prime}(u_{1})\sigma(u_{2})\big)\bar{\kappa}_{1,s}^{\varsigma,\delta}\bar{\kappa}_{2,s}^{\varsigma,\delta}\varphi_{\beta}\zeta_{N}\mathrm{d}s\\
 & \quad-\frac{1}{2}\int_{y,\eta,x,z}\int_{0}^{t}\big(F_{3}(x)\sigma^{2}(u_{1})+F_{3}(z)\sigma^{2}(u_{2})-2\nabla_{x}f_{k}(x)\cdot\nabla_{z}f_{k}(z)\sigma(u_{1})\sigma(u_{2})\big)\bar{\kappa}_{1,s}^{\varsigma,\delta}\bar{\kappa}_{2,s}^{\varsigma,\delta}\varphi_{\beta}\zeta_{N}\mathrm{d}s.
\end{align*}
Combining (\ref{eq:first order term-1}) with (\ref{eq:second term all-1}), we arrive at
\begin{align*}
 & \int_{y,\eta}\left.\big(\chi_{1,s}^{\varsigma,\delta}(y,\eta)-\chi_{1,s}^{\varsigma,\delta}(y,\eta)\chi_{2,s}^{\varsigma,\delta}(y,\eta)\big)\varphi_{\beta}(\eta)\zeta_{N}(\eta)\right|_{s=0}^{t}\\
 & =I_{t}^{\mathrm{obs}}+I_{t}^{\mathrm{err}}+I_{t}^{\mathrm{meas}}+I_{t}^{\mathrm{cut}}+I_{t}^{\mathrm{mart}}+I_{t}^{\mathrm{cons}},
\end{align*}
where $I_{t}^{\mathrm{fun}}=I_{1,t}^{\mathrm{fun}}-I_{\mathrm{mix},t}^{\mathrm{fun}}$
for ``fun'' taking values in \{obs, err, meas, cut, mart, cons\} with defining
$I_{i,t}^{\mathrm{err}},I_{i,t}^{\mathrm{meas}}:=0$ for $i\in\{1,2\}$.

We first estimate the obstacle term $I_{t}^{\mathrm{obs}}$. By
$\varepsilon_{1}\leq\varepsilon_{2}$ and the definition of the function
$\kappa^{\varsigma,\delta}$, we have
\begin{align}
I_{t}^{\mathrm{obs}} & =-\frac{1}{\varepsilon_{1}}\int_{y,\eta,x}\int_{0}^{t}\big(1-\chi_{2,s}^{\varsigma,\delta}(y,\eta)\big)\varphi_{\beta}(\eta)\zeta_{N}(\eta)(u_{1}-\psi(x,s))^{+}\kappa^{\varsigma,\delta}(x,y,u_{1},\eta)\mathrm{d}s\label{eq:I^obs}\\
 & \quad+\frac{1}{\varepsilon_{2}}\int_{y,\eta,z}\int_{0}^{t}\chi_{1,s}^{\varsigma,\delta}(y,\eta)\varphi_{\beta}(\eta)\zeta_{N}(\eta)(u_{2}-\psi(z,s))^{+}\kappa^{\varsigma,\delta}(z,y,u_{2},\eta)\mathrm{d}s\nonumber \\
 & \leq-\frac{1}{\varepsilon_{1}}\int_{y,\eta,x}\int_{0}^{t}\big(1-\chi_{2,s}^{\varsigma,\delta}(y,\eta)\big)\varphi_{\beta}(u_{1})\zeta_{N}(u_{1})(u_{1}-\psi(x,s))^{+}\kappa^{\varsigma,\delta}(x,y,u_{1},\eta)\mathrm{d}s\nonumber \\
 & \quad+\frac{1}{\varepsilon_{1}}\int_{y,\eta,z}\int_{0}^{t}\chi_{1,s}^{\varsigma,\delta}(y,\eta)\varphi_{\beta}(u_{2})\zeta_{N}(u_{2})(u_{2}-\psi(z,s))^{+}\kappa^{\varsigma,\delta}(z,y,u_{2},\eta)\mathrm{d}s\nonumber \\
 & \quad+\frac{C(\beta,N)\delta}{\varepsilon_{1}},\nonumber
\end{align}
where $u_1=u_1(x,s)$ and $u_2=u_2(z,s)$.
Note that for $i\in\{1,2\}$, when $u_{i}(x,s)>2\delta$, it holds that
\begin{align}
\int_{\eta,\tilde{\xi}}\kappa_{1}^{\delta}(\xi-\eta)\kappa_{1}^{\delta}(\eta-\tilde{\xi})\chi_{i}(x,\tilde{\xi},s) & =\begin{cases}
0, & \text{if }\xi\leq-2\delta\text{ or }\xi\geq u_{i}(x,s)+2\delta;\\
1/2, & \text{if }\xi=0\text{ or }\xi=u_{i}(x,s);\\
1, & \text{if }2\delta<\xi<u_{i}(x,s)-2\delta.
\end{cases}\label{eq: limit of conv}
\end{align}
Then, we have
\begin{align*}
 & \lim_{\delta\to0}\int_{\eta}\kappa_{1}^{\delta}(u_{1}(x,s)-\eta)\bigg(1-\int_{\tilde{\xi}}\kappa_{1}^{\delta}(\eta-\tilde{\xi})\chi_{2}(x,\tilde{\xi},s)\bigg)\varphi_{\beta}(u_{1}(x,s))\\
 & =(1-\mathbf{1}_{\{u_{1}<u_{2}\}}-\frac{1}{2}\cdot\mathbf{1}_{\{u_{1}=u_{2}\}})\varphi_{\beta}(u_{1}(x,s))\\
 & =(\mathbf{1}_{\{u_{2}<u_{1}\}}+\frac{1}{2}\cdot\mathbf{1}_{\{u_{1}=u_{2}\}})\varphi_{\beta}(u_{1}(x,s)),
\end{align*}
and
\begin{align*}
 & \lim_{\delta\to0}\int_{\eta}\kappa_{1}^{\delta}(u_{2}(x,s)-\eta)\int_{\tilde{\xi}}\kappa_{1}^{\delta}(\eta-\tilde{\xi})\chi_{1}(x,\tilde{\xi},s)\varphi_{\beta}(u_{2}(x,s))\\
 & =(\mathbf{1}_{\{u_{2}<u_{1}\}}+\frac{1}{2}\cdot\mathbf{1}_{\{u_{1}=u_{2}\}})\varphi_{\beta}(u_{2}(x,s)).
\end{align*}
Therefore, by taking the limits $\varsigma\to0$ and then $\delta\to0$
to \eqref{eq:I^obs}, with the continuity of translations in $L^{1}$,
we have
\begin{align*}
\lim_{\delta\to0}\lim_{\varsigma\to0}I_{t}^{\mathrm{obs}} & \leq\frac{1}{\varepsilon_{1}}\int_{x}\int_{0}^{t}(\mathbf{1}_{\{u_{2}<u_{1}\}}+\frac{1}{2}\cdot\mathbf{1}_{\{u_{1}=u_{2}\}})\\
 & \quad\quad\cdot\big[(u_{2}(x,s)-\psi(x,s))^{+}\varphi_{\beta}(u_{2})\zeta_{N}(u_{2})-(u_{1}(x,s)-\psi(x,s))^{+}\varphi_{\beta}(u_{1})\zeta_{N}(u_{1})\big]\mathrm{d}s.
\end{align*}
Taking $\beta\to0$, $N\to\infty$, using the dominated convergence theorem and the fact that
$
(u_{2}(x,s)-\psi(x,s))^{+}-(u_{1}(x,s)-\psi(x,s))^{+}\leq(u_{2}(x,s)-u_{1}(x,s))^{+}$ when $\psi\geq0$, we have
\begin{align*}
\lim_{N\to\infty}\lim_{\beta\to0}\lim_{\delta\to0}\lim_{\varsigma\to0}I_{t}^{\mathrm{obs}} & \leq\frac{1}{\varepsilon_{1}}\int_{x}\int_{0}^{t}(u_{2}(x,s)-u_{1}(x,s))^{+}\cdot(\mathbf{1}_{\{u_{2}<u_{1}\}}+\frac{1}{2}\cdot\mathbf{1}_{\{u_{1}=u_{2}\}})\mathrm{d}s\leq0.
\end{align*}
For the cut term $I_{t}^{\mathrm{cut}}$, it takes the form
\begin{align*}
I_{t}^{\mathrm{cut}} & =-\int_{y,\eta}\int_{0}^{t}\int_{\mathbb{R}}\int_{\mathbb{T}^{d}}(1-\chi_{2,s}^{\varsigma,\delta}(y,\eta))\partial_{\eta}\big(\varphi_{\beta}(\eta)\zeta_{N}(\eta)\big)\kappa^{\varsigma,\delta}(x,y,\xi,\eta)\mathrm{d}m_{1}(x,\xi,s)\\
 & \quad+\frac{1}{2}\int_{y,\eta,x}\int_{0}^{t}(1-\chi_{2,s}^{\varsigma,\delta}(y,\eta))\partial_{\eta}\big(\varphi_{\beta}(\eta)\zeta_{N}(\eta)\big)\Big(\sigma(u_{1})\sigma^{\prime}(u_{1})\nabla_{x}u_{1}\cdot F_{2}+\sigma^{2}(u_{1})F_{3}\Big)\bar{\kappa}_{1,s}^{\varsigma,\delta}\mathrm{d}s\\
 & \quad+\int_{y,\eta}\int_{0}^{t}\int_{\mathbb{R}}\int_{\mathbb{T}^{d}}\chi_{1,s}^{\varsigma,\delta}(y,\eta)\partial_{\eta}\big(\varphi_{\beta}(\eta)\zeta_{N}(\eta)\big)\kappa^{\varsigma,\delta}(z,y,\tilde{\xi},\eta)\mathrm{d}m_{2}(z,\tilde{\xi},s)\\
 & \quad-\frac{1}{2}\int_{y,\eta,x}\int_{0}^{t}\chi_{1,s}^{\varsigma,\delta}(y,\eta)\partial_{\eta}\big(\varphi_{\beta}(\eta)\zeta_{N}(\eta)\big)\Big(\sigma(u_{2})\sigma^{\prime}(u_{2})\nabla_{z}u_{2}\cdot F_{2}+\sigma^{2}(u_{2})F_{3}\Big)\bar{\kappa}_{2,s}^{\varsigma,\delta}\mathrm{d}s.
\end{align*}
With the aid of Definition \ref{def:def of stochastic kinetic solution without obstacle} (v) and a property analogous to Proposition \ref{prop:proposition for limit measure}, we have almost surely
\begin{align*}
 & \limsup_{N\to\infty,\beta\to0}\limsup_{\varsigma,\delta\to0}\Bigg|\int_{y,\eta}\int_{0}^{t}\int_{\mathbb{R}}\int_{\mathbb{T}^{d}}(1-\chi_{2,s}^{\varsigma,\delta}(y,\eta))\partial_{\eta}\big(\varphi_{\beta}(\eta)\zeta_{N}(\eta)\big)\kappa^{\varsigma,\delta}(x,y,\xi,\eta)\mathrm{d}m_{1}(x,\xi,s)\Bigg|\\
 & \leq\limsup_{N\to\infty,\beta\to0}C\big(\beta^{-1}m_{1}(\mathbb{T}^{d}\times[\beta/2,\beta]\times[0,T])+m_{1}(\mathbb{T}^{d}\times[N,N+1]\times[0,T])\big)=0,
\end{align*}
and a similar result holds for the term involving $m_{2}$.
For the terms involving
$F_{3}$, based on the boundedness of $F_{3}$, Assumption \ref{assu:assu for existence} and the $L^{2}$-integrability of $\sigma^{2}(u_{1})$,
we have almost surely
\begin{align*}
 & \limsup_{N\to\infty,\beta\to0}\limsup_{\varsigma,\delta\to0}\bigg|\int_{y,\eta,x}\int_{0}^{t}(1-\chi_{2,s}^{\varsigma,\delta}(y,\eta))\partial_{\eta}\big(\varphi_{\beta}(\eta)\zeta_{N}(\eta)\big)\sigma^{2}(u_{1})F_{3}\bar{\kappa}_{1,s}^{\varsigma,\delta}\mathrm{d}s\bigg|\\
&+\limsup_{N\to\infty,\beta\to0}\limsup_{\varsigma,\delta\to0}\bigg|\int_{y,\eta,x}\int_{0}^{t}\chi_{1,s}^{\varsigma,\delta}(y,\eta)\partial_{\eta}\big(\varphi_{\beta}(\eta)\zeta_{N}(\eta)\big)\sigma^{2}(u_{2})F_{3}\bar{\kappa}_{2,s}^{\varsigma,\delta}\mathrm{d}s\bigg|\\
 & \leq C\limsup_{N\to\infty,\beta\to0}\bigg[\bigg|\beta^{-1}\int_{0}^{t}\int_{x}\mathbf{1}_{\{\beta/2\leq u_{1}\leq\beta\}}\sigma^{2}(u_{1})F_{3}\mathrm{d}s\bigg|+\bigg|\int_{0}^{t}\int_{x}\mathbf{1}_{\{N\leq u_{1}\leq N+1\}}\sigma^{2}(u_{1})F_{3}\mathrm{d}s\bigg|\bigg]\\
 & =0.
\end{align*}
For terms involving $F_{2}$, based on the definition of the function
$\kappa^{\varsigma,\delta}$, we have
\begin{align*}
 & \frac{1}{2}\int_{y,\eta,x}\int_{0}^{t}(1-\chi_{2,s}^{\varsigma,\delta}(y,\eta))\partial_{\eta}\big(\varphi_{\beta}(\eta)\zeta_{N}(\eta)\big)\sigma(u_{1})\sigma^{\prime}(u_{1})\nabla_{x}u_{1}\cdot F_{2}\bar{\kappa}_{1,s}^{\varsigma,\delta}\mathrm{d}s\\
 & -\frac{1}{2}\int_{y,\eta,x}\int_{0}^{t}\chi_{1,s}^{\varsigma,\delta}(y,\eta)\partial_{\eta}\big(\varphi_{\beta}(\eta)\zeta_{N}(\eta)\big)\sigma(u_{2})\sigma^{\prime}(u_{2})\nabla_{z}u_{2}\cdot F_{2}\bar{\kappa}_{2,s}^{\varsigma,\delta}\mathrm{d}s\\
 & =\frac{1}{2}\int_{y,\eta,x}\int_{0}^{t}(1-\chi_{2,s}^{\varsigma,\delta}(y,\eta))\big(\varphi_{\beta}^{\prime}(u_{1})+\zeta_{N}^{\prime}(u_{1})\big)\sigma(u_{1})\sigma^{\prime}(u_{1})\nabla_{x}u_{1}\cdot F_{2}\bar{\kappa}_{1,s}^{\varsigma,\delta}\mathrm{d}s\\
 & \quad-\frac{1}{2}\int_{y,\eta,x}\int_{0}^{t}\chi_{1,s}^{\varsigma,\delta}(y,\eta)\big(\varphi_{\beta}^{\prime}(u_{2})+\zeta_{N}^{\prime}(u_{2})\big)\sigma(u_{2})\sigma^{\prime}(u_{2})\nabla_{z}u_{2}\cdot F_{2}\bar{\kappa}_{2,s}^{\varsigma,\delta}\mathrm{d}s+C(\beta,N)\delta\\
 =:& I_{t}^{\mathrm{cut},1}+C(\beta,N)\delta.
\end{align*}
Then, using \eqref{eq: limit of conv} and Stampacchia's lemma \cite[Chapter 5, Exercises 17 and 18]{evans1998partial}, we have
\begin{align*}
  \lim_{\varsigma,\delta\to0}I_{t}^{\mathrm{cut},1}
 & =\frac{1}{2}\int_{0}^{t}\int_{x}\mathbf{1}_{\{u_{1}>u_{2}\}}\big(\varphi_{\beta}^{\prime}(u_{1})+\zeta_{N}^{\prime}(u_{1})\big)\sigma(u_{1})\sigma^{\prime}(u_{1})\nabla_{x}u_{1}\cdot F_{2}\mathrm{d}s\\
 & \quad-\frac{1}{2}\int_{0}^{t}\int_{x}\mathbf{1}_{\{u_{1}>u_{2}\}}\big(\varphi_{\beta}^{\prime}(u_{2})+\zeta_{N}^{\prime}(u_{2})\big)\sigma(u_{2})\sigma^{\prime}(u_{2})\nabla_{x}u_{2}\cdot F_{2}\mathrm{d}s\\
 & =\frac{1}{4\beta}\int_{0}^{t}\int_{x}(\mathrm{sgn}(u_1-u_2)+1)\nabla_{x}\big(\sigma^{2}((u_{1}\land\beta)\vee\beta/2)-\sigma^{2}(\beta/2)\big)\cdot F_{2}\mathrm{d}s\\
 & \quad-\frac{1}{4\beta}\int_{0}^{t}\int_{x}(\mathrm{sgn}(u_1-u_2)+1)\nabla_{x}\big(\sigma^{2}((u_{2}\land\beta)\vee\beta/2)-\sigma^{2}(\beta/2)\big)\cdot F_{2}\mathrm{d}s\\
 & \quad+\frac{1}{8}\int_{0}^{t}\int_{x}(\mathrm{sgn}(u_1-u_2)+1)\nabla_{x}\big(\sigma^{2}((u_{1}\land(N+1))\vee N)-\sigma^{2}(N)\big)\cdot F_{2}\mathrm{d}s\\
 & \quad-\frac{1}{8}\int_{0}^{t}\int_{x}(\mathrm{sgn}(u_1-u_2)+1)\nabla_{x}\big(\sigma^{2}((u_{2}\land(N+1))\vee N)-\sigma^{2}(N)\big)\cdot F_{2}\mathrm{d}s.
\end{align*}
Then, based on Assumption \ref{assu:assu for existence} for $\sigma$ and Assumption \ref{assu:assum for F} for $F_2$,
taking the limits $\beta\to0$, $N\to\infty$, and combining
the estimates for the terms involving $m_{1}$, $m_{2}$ and $F_{3}$, we
have
\[
\lim_{N\to\infty}\lim_{\beta\to0}\lim_{\delta\to0}\lim_{\varsigma\to0}I_{t}^{\mathrm{cut}}=0.
\]
For the martingale term $I_{t}^{\mathrm{mart}}$ and conservative term $I_{t}^{\mathrm{cons}}$, we have
\begin{align*}
I_{t}^{\mathrm{mart}} & =\int_{y,\eta,x}\int_{0}^{t}(\chi_{2,s}^{\varsigma,\delta}(y,\eta)-1)\varphi_{\beta}(\eta)\zeta_{N}(\eta)\bar{\kappa}_{1,s}^{\varsigma,\delta}(x,y,\eta)\nabla_{x}\cdot(\sigma(u_{1})f_{k})\mathrm{d}B_{s}^{k}\\
 & \quad+\int_{y,\eta,z}\int_{0}^{t}\chi_{1,s}^{\varsigma,\delta}(y,\eta)\varphi_{\beta}(\eta)\zeta_{N}(\eta)\bar{\kappa}_{2,s}^{\varsigma,\delta}(z,y,\eta)\nabla_{z}\cdot(\sigma(u_{2})f_{k})\mathrm{d}B_{s}^{k},
\end{align*}
and
\begin{align*}
I_{t}^{\mathrm{cons}} & =\int_{y,\eta,x}\int_{0}^{t}(\chi_{2,s}^{\varsigma,\delta}(y,\eta)-1)\varphi_{\beta}(\eta)\zeta_{N}(\eta)\bar{\kappa}_{1,s}^{\varsigma,\delta}(x,y,\eta)\nabla_{x}\cdot g(u_{1})\mathrm{d}s\\
 & \quad+\int_{y,\eta,z}\int_{0}^{t}\chi_{1,s}^{\varsigma,\delta}(y,\eta)\varphi_{\beta}(\eta)\zeta_{N}(\eta)\bar{\kappa}_{2,s}^{\varsigma,\delta}(z,y,\eta)\nabla_{z}\cdot g(u_{2})\mathrm{d}s.
\end{align*}
Following a similar procedure in estimating $I_{t}^{\mathrm{cut}}$, we can replace $\varphi_{\beta}(\eta)\zeta_{N}(\eta)$
with $\varphi_{\beta}(u_{i})\zeta_{N}(u_{i})$ and apply \eqref{eq: limit of conv}
and Stampacchia's lemma \cite[Chapter 5, Exercises 17, 18]{evans1998partial} to
obtain
\begin{align*}
\lim_{\delta\to0}\lim_{\varsigma\to0}I_{t}^{\mathrm{mart}} & =-\frac{1}{2}\int_{x,\eta}\int_{0}^{t}(\mathrm{sgn}(u_1-u_2)+1)\varphi_{\beta}(u_{1})\zeta_{N}(u_{1})\nabla_{x}\cdot(\sigma(u_{1})f_{k})\mathrm{d}B_{s}^{k}\\
 & \quad+\frac{1}{2}\int_{x,\eta}\int_{0}^{t}(\mathrm{sgn}(u_1-u_2)+1)\varphi_{\beta}(u_{2})\zeta_{N}(u_{2})\nabla_{x}\cdot(\sigma(u_{2})f_{k})\mathrm{d}B_{s}^{k},
\end{align*}
and
\begin{align*}
\lim_{\delta\to0}\lim_{\varsigma\to0}I_{t}^{\mathrm{cons}} & =-\frac{1}{2}\int_{x,\eta}\int_{0}^{t}(\mathrm{sgn}(u_1-u_2)+1)\varphi_{\beta}(u_{1})\zeta_{N}(u_{1})\nabla_{x}\cdot g(u_{1})\mathrm{d}s\\
 & \quad+\frac{1}{2}\int_{x,\eta}\int_{0}^{t}(\mathrm{sgn}(u_1-u_2)+1)\varphi_{\beta}(u_{2})\zeta_{N}(u_{2})\nabla_{x}\cdot g(u_{2})\mathrm{d}s,
\end{align*}
Following \cite[the proof of Theorem 4.6]{fehrman2024well}, using the $L^2$ integrability of $\sigma(u_i)$ and $L^1$ integrability of $g(u_i)$, with dominated convergence theorem, we obtain almost surely for every $t\in [0,T]$,
\[
\lim_{N\to\infty}\lim_{\beta\to0}\lim_{\delta\to0}\lim_{\varsigma\to0}I_{t}^{\mathrm{mart}}=0,\quad\text{and}\quad\lim_{N\to\infty}\lim_{\beta\to0}\lim_{\delta\to0}\lim_{\varsigma\to0}I_{t}^{\mathrm{cons}}=0.
\]

For the estimates of error term $I_{t}^{\mathrm{err}}$ and measure
term $I_{t}^{\mathrm{meas}}$, they can be done by a similar method to \cite[the proof of Theorem 4.6]{fehrman2024well}. Then, we reach a.s. for every $t\in [0,T]$,
\begin{align*}
  \lim_{\delta\to0}\lim_{\varsigma\to0}I_{i,t}^{\mathrm{err}}=0,\quad\text{and}\quad
  I_{i,t}^{\mathrm{meas}}\leq0.
\end{align*}
 Combining
all the above estimates, we arrive at
\begin{align*}
 & \int_{x,\xi}\chi_{1}(x,\xi,s)(1-\chi_{2}(x,\xi,s))\Big|_{s=0}^{t}\\
 & \lim_{N\to\infty}\lim_{\beta\to0}\lim_{\delta\to0}\lim_{\varsigma\to0}\int_{y,\eta}\left.\big(\chi_{1,s}^{\varsigma,\delta}(y,\eta)-\chi_{1,s}^{\varsigma,\delta}(y,\eta)\chi_{2,s}^{\varsigma,\delta}(y,\eta)\big)\varphi_{\beta}(\eta)\zeta_{N}(\eta)\right|_{s=0}^{t}\\
 & =\lim_{N\to\infty}\lim_{\beta\to0}\lim_{\delta\to0}\lim_{\varsigma\to0}\big(I_{t}^{\mathrm{obs}}+I_{t}^{\mathrm{err}}+I_{t}^{\mathrm{meas}}+I_{t}^{\mathrm{cut}}+I_{t}^{\mathrm{mart}}+I_{t}^{\mathrm{cons}}\big)\leq0.
\end{align*}
Therefore, we have
\[
\int_{x}(u_{1}-u_{2})^{+}(t)=\int_{x,\xi}\chi_{1}(x,\xi,t)(1-\chi_{2}(x,\xi,t))\leq0,
\]
which completes the proof.
\end{proof}

\begin{thm}\label{thm:existence}
(Existence of solutions to the obstacle problem for \eqref{eq:main})
Suppose that Assumption \ref{assu:assu for existence} and Assumption \ref{assu:assum for F} hold.
Then, there exists a stochastic kinetic solution $(u,\nu)$ of \eqref{eq:main}
in the sense of Definition \ref{def:def of stochastic kinetic solution}.
\end{thm}
\begin{rem}
The proof of Theorem \ref{thm:existence} is different from \cite[the proof of Theorem 2.4]{du2024well}, owing to a crucial observation: when passing to the limit $\varepsilon\to0$ in \eqref{eq:kinetic for penalty}, it is infeasible to take a test function that is independent of $\xi$ by approximation as the terms involving $\sigma^\prime$ may become unbounded.
Consequently, it is difficult to eliminate the contribution of the kinetic measure $m_\varepsilon$ and to prove the almost sure convergence of the penalty term $\varepsilon^{-1}(u_\varepsilon-\psi)^+$.
Instead,
within the current framework, we can derive the tightness of the laws of the penalization term $\varepsilon^{-1}(u_\varepsilon-\psi)^+$ in the space of Radon measures, by using estimate \eqref{eq:priori for u epsilon} and Prokhorov's theorem.
This necessitates an additional step in the proof, that is, we have to employ the Skorokhod representation theorem to work on a new probability space $\tilde{\Omega}$, where we can obtain the almost sure convergence of both the penalization term and kinetic measure as $\varepsilon\to0$.
Combining this with the uniqueness result (Theorem \ref{thm:uniqueness}) and Gy\"{o}ngy-Krylov method, we then conclude that $(u_\varepsilon,{\varepsilon}^{-1}(u_\varepsilon-\psi)^+)$ converges in probability on the original probability space $\Omega$.
\end{rem}
\begin{proof}[Proof of Theorem \ref{thm:existence}]
Let $\{\varepsilon_{l}\}_{l\in\mathbb{N}}$ be a monotonic sequence
decreasing to $0$.
According to Lemma \ref{lem:(Existence-of-solution for penalized eqn }, the penalized equation \eqref{eq:penalized eq} with penalization coefficient $\varepsilon_{l}$ has a stochastic kinetic solution $u_{\varepsilon_{l}}$ in the sense of Definition \ref{def:Weak solution approximating penalized}.
By Lemma \ref{lem:(Comparison-theorem)} and Proposition \ref{prop:(Nonnegativity)}, the sequence $\{u_{\varepsilon_{l}}\}_{l\in\mathbb{N}}$ is almost surely decreasing as $l\to\infty$, and thus admits a pointwise limit $u$ that is almost surely non-negative. Moreover, combining estimate \eqref{eq:priori for u epsilon} with the dominated convergence theorem, we obtain strong convergence of the sequence in $L^{2}(\Omega\times[0,T];L^{2}(\mathbb{T}^{d}))$.
Meanwhile, \eqref{eq:priori for u epsilon} also indicates
\begin{align}
 & \mathbb{E}\Vert(u_{\varepsilon_l}-\psi)^{+}\Vert_{L^{1}(Q_{T})}\leq C\varepsilon_l\Big(1+\mathbb{E}\Vert u_{\mathrm{init}}\Vert_{L^{2}(\mathbb{T}^{d})}^{2}+\mathbb{E}\Vert u_{\mathrm{init}}\Vert_{L^{1}(\mathbb{T}^{d})}^{m+1}\Big)\to0,\quad\text{when }l\to\infty,\label{eq:priori-1}
\end{align}
and the weak convergence of $\llbracket\sqrt{\Phi^{\prime}}\rrbracket(u_{\varepsilon_{l}})$
in $L^{2}(\Omega\times[0,T];H^{1}(\mathbb{T}^d))$ by taking a subsequence.
Using a test function argument (see \cite[the proof of Theorem 2.7]{dareiotis2020nonlinear}) and the pointwise convergence of $u_{\varepsilon_{l}}$, the weak limit of $\llbracket\sqrt{\Phi^{\prime}}\rrbracket(u_{\varepsilon_{l}})$ is
$\llbracket\sqrt{\Phi^{\prime}}\rrbracket(u)$ when $l\to\infty$, and we have
\begin{align*}
 & \sup_{t\in[0,T]}\mathbb{E}\Vert u(t)\Vert_{L^{2}(\mathbb{T}^{d})}^{2}+\mathbb{E}\Vert\nabla_{x}\llbracket\sqrt{\Phi^{\prime}}\rrbracket(u)\Vert_{L^{2}(Q_{T})}^{2}\leq C\Big(1+\mathbb{E}\Vert u_{\mathrm{init}}\Vert_{L^{2}(\mathbb{T}^{d})}^{2}+\mathbb{E}\Vert u_{\mathrm{init}}\Vert_{L^{1}(\mathbb{T}^{d})}^{m+1}\Big),
\end{align*}
and
\begin{align*}
 & \sup_{t\in[0,T]}\mathbb{E}\Vert u(t)\Vert_{L^{p}(\mathbb{T}^{d})}^{p}\leq C\Big(1+\mathbb{E}\Vert u_{\mathrm{init}}\Vert_{L^{p}(\mathbb{T}^{d})}^{p}+\mathbb{E}\Vert u_{\mathrm{init}}\Vert_{L^{1}(\mathbb{T}^{d})}^{m+p-1}\Big).
\end{align*}
Define the measures $\mathrm{d}\nu_{l}(x,t):=\varepsilon_{l}^{-1}(u_{\varepsilon_{l}}-\psi)^{+}\mathrm{d}x\mathrm{d}t$ and
\[
\mathrm{d}\lambda_{\varepsilon_l}(x,\xi,t):=\int_{0}^{1}\frac{1}{\varepsilon_l}[(u_{\varepsilon_l}-\psi)^{+}]^{2}\delta_{0}(\xi-(\psi+\theta(u_{\varepsilon_l}-\psi)))\mathrm{d}\theta \mathrm{d}x\mathrm{d}\xi\mathrm{d}t.
\]
Since for every $\phi\in C_c^\infty(\mathbb{T}^d\times(0,\infty)\times[0,T))$ and $\varepsilon_{l}\in(0,1)$,
\begin{align*}
 & (\phi(x,u_{\varepsilon_l},t)-\phi(x,\psi,t))\frac{1}{\varepsilon_l}(u_{\varepsilon_l}-\psi)^{+}=\int_{0}^{1}\partial_\xi \phi(x,\psi+\theta(u_{\varepsilon_l}-\psi),t)\mathrm{d}\theta\frac{1}{\varepsilon}[(u_{\varepsilon_l}-\psi)^{+}]^{2},
\end{align*}
we have
\begin{align*}
&\int_0^{T}\int_{\mathbb{T}^d}\phi(x,u_{\varepsilon_l},t)d\nu_l(x,\xi,t)\\
&=
\int_0^{T}\int_{\mathbb{T}^d}\phi(x,\psi,t)d\nu_l(x,\xi,t)
+\int_0^{T}\int_\mathbb{R}\int_{\mathbb{T}^d}\partial_\xi\phi(x,\xi,t)\mathrm{d}\lambda_{\varepsilon_l}(x,\xi,t).
\end{align*}
Define the measure $q_l:=m_{\varepsilon_l}+\lambda_{\varepsilon_l}$.
Since the measure $\lambda_{\varepsilon_l}$ is nonnegative, the kinetic measure $q_l$ also satisfies
\begin{equation}\label{eq:kinetic for ql}
\delta_{0}(\xi-u_{\varepsilon_l})\Phi^{\prime}(\xi)|\nabla u_{\varepsilon_l}|^{2}\leq m_{\varepsilon_l}\leq q_l\quad\text{on}\quad\mathbb{T}^{d}\times(0,\infty)\times[0,T].
\end{equation}
Using \eqref{eq:priori for u epsilon} and \eqref{eq:penalty kinetic priori esti}, we have
\begin{align}
&\mathbb{E}q_{l}(\mathbb{T}^d\times\mathbb{R}\times[0,T])+\mathbb{E}\Vert\varepsilon^{-1}(u_{\varepsilon}-\psi)^{+}\Vert_{L^{1}(Q_{T})}\leq C\Big(1+\mathbb{E}\Vert u_{\mathrm{init}}\Vert_{L^{2}(\mathbb{T}^{d})}^{2}+\mathbb{E}\Vert u_{\mathrm{init}}\Vert_{L^{1}(\mathbb{T}^{d})}^{m+1}\Big),\label{eq:priori for q}
\end{align}
and
\begin{align}
 & \mathbb{E}q_l(\mathbb{T}^d\times(N_1,N_2)\times[0,T])  \label{eq:kinetic priori esti}\\
 &\leq C(N_2-N_1)\mathbb
 {E}\int_x(u_{\mathrm{init}}-N_1)^++C\mathbb{E}\int_{x,t}\mathbf{1}_{\{u_{\varepsilon_l}\geq N_{1}\}}|\sigma(u_{\varepsilon_l}\land N_2)|^{2}.\nonumber
 \end{align}
Now, we focus on the limits of $\nu_{l}$ and $q_{l}$.
Let $s>d/2+1$.
For every $\phi\in C_c^\infty(\mathbb{T}^d\times(0,\infty)\times[0,T))$ and $l\in\mathbb{N}$, we define
\begin{align}
M_{\phi,l}(\tau):=
& \int_{0}^{\tau}\int_{\mathbb{R}}\int_{\mathbb{T}^d}\chi_{\varepsilon_l}(x,\xi,t)\partial_t\phi(x,\xi,t)\mathrm{d}x\mathrm{d}\xi\mathrm{d}t\label{eq:Def for stochastic M}\\
 & +\int_{\mathbb{R}}\int_{\mathbb{T}^{d}}\bar{\chi}(u_{\mathrm{init}},\xi)\phi(x,\xi,0)\mathrm{d}x\mathrm{d}\xi-\int_{0}^{\tau}\int_{\mathbb{T}^{d}}\Phi^{\prime}(u_{\varepsilon_l})\nabla_x u_{\varepsilon_l}\cdot(\nabla_x\phi)(x,u_{\varepsilon_l},t)\mathrm{d}x\mathrm{d}t\nonumber \\
 & \quad-\frac{1}{2}\int_{0}^{\tau}\int_{\mathbb{T}^{d}}\big(F_{1}(x)[\sigma^{\prime}(u_{\varepsilon_l})]^{2}\nabla_x u_{\varepsilon_l}+\sigma(u_{\varepsilon_l})\sigma^{\prime}(u_{\varepsilon_l})F_{2}(x)\big)\cdot(\nabla_x\phi)(x,u_{\varepsilon_l},t)\mathrm{d}x\mathrm{d}t\nonumber\\
 &\quad-\int_0^{\tau}\int_{\mathbb{T}^d}\phi(x,u_{\varepsilon_l},t)\nabla_x\cdot g(u_{\varepsilon_l})\mathrm{d}x\mathrm{d}t\nonumber\\
 & \quad-\int_{0}^{\tau}\int_{\mathbb{R}}\int_{\mathbb{T}^{d}}\partial_{\xi}\phi(x,\xi,t)\mathrm{d}q_l(x,\xi,t)-\int_{0}^{\tau}\int_{\mathbb{T}^{d}}\phi(x,\psi,t)\mathrm{d}\nu_l(x,t)\nonumber \\
 & \quad+\frac{1}{2}\int_{0}^{\tau}\int_{\mathbb{T}^{d}}\Big(\sigma(u_{\varepsilon_l})\sigma^{\prime}(u_{\varepsilon_l})\nabla_x u_{\varepsilon_l}\cdot F_{2}(x)+\sigma^{2}(u_{\varepsilon_l})F_{3}(x)\Big)(\partial_{\xi}\phi)(x,u_{\varepsilon_l},t)\mathrm{d}x\mathrm{d}t.\nonumber
\end{align}
Let $\{\phi_j\}_{j\in\mathbb{N}}$ be a countable sequence which is dense in $C_c^\infty(\mathbb{T}^d\times(0,\infty)\times[0,T))$ in the strong $H^s(\mathbb{T}^d\times(0,\infty)\times[0,T))$-topology. We define the random variables
\[
X_l:=(u_{\varepsilon_l},\nabla\llbracket\sqrt{\Phi^\prime}\rrbracket(u_{\varepsilon_l}),q_l,\nu_l,\{M_{\phi_j,l}\}_{j\in\mathbb{N}}),
\]
taking values in the space
\[
\bar{X}:=L^1(Q_{T})\times L^2(Q_{T};\mathbb{R}^d)\times\mathcal{M}(\mathbb{T}^d\times\mathbb{R}\times[0,T])\times\mathcal{M}(Q_T)\times C([0,T])^{\mathbb{N}},
\]
with the product metric topology induced by the strong topology in $L^1(Q_T)$, the weak topology in $L^2(Q_T;\mathbb{R}^d)$, the weak topology in the space of nonnegative Borel measures $\mathcal{M}(\mathbb{T}^d\times\mathbb{R}\times[0,T])$, the weak topology in the space of nonnegative Borel measures $\mathcal{M}(Q_T)$, and the topology of component-wise convergence in the strong norm in $C([0,T])^{\mathbb{N}}$ induced by the metric
\[
d_C(\{f_j\}_{j\in\mathbb{N}},\{\tilde{f}_j\}_{j\in\mathbb{N}})=\sum_{j=1}^\infty 2^{-j}\frac{\Vert f_j-\tilde{f}_j\Vert_{C([0,T])}}{1+\Vert f_j-\tilde{f}_j\Vert_{C([0,T])}}.
\]
Based on the almost sure convergence of $u_{\varepsilon_l}$ and \eqref{eq:priori for u epsilon}, the laws of $X_l$ are tight on $\bar{X}$.

In the following, our aim is to use \cite[Lemma 1.1]{gyongy2022existence} to give the convergence in probability of $\{\nu_l\}_{l\in\mathbb{N}}$.
Let $\{l_h\}_{h\in\mathbb{N}}$ and $\{\bar{l}_h\}_{h\in\mathbb{N}}$ be two subsequences satisfying $l_h,\bar{l}_h\to\infty$ as $h\to\infty$.
Using Prokhorov's theorem (see e.g. \cite[Chapter 1, Theorem 5.1]{billingsley1999convergence}), there exists a (non-relabelled) subsequence $\{X_{l_h},X_{\bar{l}_h}\}$ such that the laws of $(X_{l_h},X_{\bar{l}_h},B,u_{\mathrm{init}})$ on $\bar{Y}=\bar{X}\times\bar{X}\times C([0,T])^{\mathbb{N}}\times L^1(\mathbb{T}^d)$ are weak convergent, where $B=\{B^k\}_{k\in\mathbb{N}}$.
Note that the separability of the space $\bar{X}$ implies that of the space $\bar{Y}$.
Using Jakubowski-Skorokhod representation theorem (see e.g. \cite[Theorem 2]{jakubowski1998almost}), there exist $\bar{Y}$-valued random variables $(\tilde{Y},\tilde{Z},\tilde{\beta},\tilde{u}_{\mathrm{init}})$ and $(\tilde{Y}_h,\tilde{Z}_h,\tilde{\beta}_h,\tilde{u}_{\mathrm{init},h})$,
 ${h\in\mathbb{N}}$, in a probability space $(\tilde\Omega,\tilde{\mathcal{F}},\tilde{\mathbb{P}})$ such that for every $h\in\mathbb{N}$, as random variables in $\bar{Y}$,
\begin{equation}\label{eq:equal for distribution}
(\tilde{Y}_h,\tilde{Z}_h,\tilde{\beta}_h,\tilde{u}_{\mathrm{init},h})\overset{d}{=}(X_{l_h},X_{\bar{l}_h},B,u_{\mathrm{init}}),
\end{equation}
and as $h\to\infty$, $\tilde{\mathbb{P}}$-almost surely,
\begin{equation}\label{eq:converge of rv}
\tilde{Y}_h\to\tilde{Y},\quad\tilde{Z}_h\to\tilde{Z},\quad\tilde{\beta}_h\to\tilde{\beta},\quad\text{and}\quad\tilde{u}_{\mathrm{init},h}\to\tilde{u}_{\mathrm{init}}\quad  \text{in } \bar{X}\text{, }C([0,T];\mathbb{R}^{2h})\text{ and }L^1(\mathbb{T}^d).
\end{equation}

Based on \eqref{eq:equal for distribution}, for each $h\in\mathbb{N}$, there exist $\tilde{u}_h\in L^1(\tilde{\Omega}\times Q_T)$, $\tilde{G}_h\in L^2(\tilde{\Omega}\times[0,T];L^2(\mathbb{T}^d;\mathbb{R}^d))$, $\tilde{q}_h\in L^1(\tilde{\Omega};\mathcal{M}(\mathbb{T}^d\times{\mathbb{R}}\times[0,T]))$, $\tilde{\nu}_h\in L^1(\tilde{\Omega};\mathcal{M}(Q_T))$, and $\{\tilde{M}_{\phi_j,h}\}_{j\in\mathbb{N}}\in L^1(\tilde{\Omega};C([0,T])^{\mathbb{N}})$ such that
\[
\tilde{Y}_h=(\tilde{u}_h,\tilde{G}_h,\tilde{q}_h,\tilde{\nu}_h,\{\tilde{M}_{\phi_j,h}\}_{j\in\mathbb{N}}).
\]
Further, owing to \eqref{eq:equal for distribution}, equality in law is preserved after applying continuous functions to the corresponding random variables.
We approximate the nonlinear functions appearing in
\[
\nabla\llbracket\sqrt{\Phi'}\rrbracket(u_{\varepsilon_l}),\qquad
\mathrm{d}\nu_l=\varepsilon_l^{-1}(u_{\varepsilon_l}-\psi)^+\mathrm{d}x\mathrm{d}t,
\qquad\text{and}\quad
M_{\phi_j,l}
\]
by smooth continuous functions of $u_{\varepsilon_l}$.
Using the test function argument again, see \cite[the proof of Theorem 5.25]{fehrman2024well}, we identify the corresponding limits. Namely, $\tilde{\mathbb P}$-almost surely,
\[
\tilde G_h=\nabla\llbracket\sqrt{\Phi'}\rrbracket(\tilde u_h),
\qquad
\mathrm{d}\tilde\nu_h=\varepsilon_{l_h}^{-1}(\tilde u_h-\psi)^+\mathrm{d}x\mathrm{d}t,
\]
and, for every $j\in\mathbb N$, the continuous path $\tilde M_{\phi_j,h}$ is given in terms of $\tilde u_h$ and $\tilde u_{\mathrm{init},h}$ by the same formula as in \eqref{eq:Def for stochastic M}.
Similarly, using \eqref{eq:equal for distribution}, \eqref{eq:kinetic for ql}, and \eqref{eq:kinetic priori esti}, the measure $\tilde{q}_h$ is $\tilde{\mathbb{P}}$-almost surely a kinetic measure for $\tilde{u}_h$ that satisfies Definition \ref{def:def of stochastic kinetic solution without obstacle} (iv) and (v).
Combining \eqref{eq:converge of rv}, there exists $\tilde{u}\in L^1(\tilde{\Omega}\times Q_T)$ such that
\[
\tilde{Y}=(\tilde{u},\nabla\llbracket\sqrt{\Phi^\prime}\rrbracket(\tilde{u}),\tilde{q},\tilde{\nu},\{\tilde{M}_{\phi_j}\}_{j\in\mathbb{N}}).
\]
where continuous paths $\{\tilde{M}_{\phi_j}\}_{j\in\mathbb{N}}$ are defined by $\tilde{u}$ and $\tilde{u}_{\mathrm{init}}$ as in \eqref{eq:Def for stochastic M}.

Now, we verify that $(\tilde{u},\tilde{\nu})$ and $\tilde{q}$ satisfy Definition \ref{def:def of stochastic kinetic solution}. Note that
\begin{align*}
&\tilde{\mathbb{E}}\Vert(\tilde{u}_h-\psi)^+\Vert_{L^1(Q_T)}=\mathbb{E}\Vert(u_{\varepsilon_{l_h}}-\psi)^+\Vert_{L^1(Q_T)}\\
&\quad\leq C\varepsilon_{l_h}(1+\mathbb{E}\Vert u_{\mathrm{init}}\Vert_{L^2(\mathbb{T}^d)}^2+\mathbb{E}\Vert u_{\mathrm{init}}\Vert_{L^1(\mathbb{T}^d)}^{m+1} )\to0,\quad \text{as }h\to\infty,
\end{align*}
using \eqref{eq:L1 perservation without obstacle}, \eqref{eq:priori for q}, \eqref{eq:equal for distribution}, the strong convergence of $\tilde{u}_h$ to $\tilde{u}$ in $L^1(Q_T)$, and the weak convergence of $\tilde{\nu}_h$ to $\tilde{\nu}$ in $\mathcal{M}(Q_T)$, we obtain Definition \ref{def:def of stochastic kinetic solution} (i) and (ii).
Combining the $L^2$-integrability of $\nabla\llbracket\sqrt{\Phi^\prime}\rrbracket(\tilde{u})$ and Assumption \ref{assu:assu for existence} for $\sigma$ and $g$, we obtain Definition \ref{def:def of stochastic kinetic solution} (iv) and (v).

For the measure $\tilde{q}$, using \eqref{eq:equal for distribution}, \eqref{eq:converge of rv}, and \eqref{eq:priori for q}, it is a kinetic measure on $\mathbb{T}^d\times\mathbb{R}\times[0,T]$.
The Definition \ref{def:def of stochastic kinetic solution} (vi) which means almost surely as a measure
\[
\delta_0(\xi-\tilde{u})\Phi^\prime(\xi)|\nabla\tilde{u}|^2=\delta_0(\xi-\tilde{u})|\nabla\llbracket\sqrt{\Phi^\prime}\rrbracket(\tilde{u})|^2\leq\tilde{q}\quad\text{ on }\mathbb{T}^d\times(0,\infty)\times[0,T],
\]
follows from the weak lower-semicontinuity of the Sobolev norm, the strong convergence of $\tilde{u}_h$ to $\tilde{u}$ in $L^1(0,T;L^1(\mathbb{T}^d))$, and the weak convergence of  $\nabla\llbracket\sqrt{\Phi^\prime}\rrbracket(\tilde{u}_h)$ to $\nabla\llbracket\sqrt{\Phi^\prime}\rrbracket(\tilde{u})$ in $L^2(0,T;L^2(\mathbb{T}^d;\mathbb{R}^d))$.
Combining \eqref{eq:kinetic priori esti}, the kinetic measure $\tilde{q}$ satisfies Definition \ref{def:def of stochastic kinetic solution} (vii).

In order to verify Definition \ref{def:def of stochastic kinetic solution} (viii), it requires to clarify the specific form of $\{\tilde{M}_{\phi_j}\}_{j\in\mathbb{N}}$. Denote $\tilde{\beta}_h=\{\tilde\beta_{j,h}\}_{j\in\mathbb{N}}$ and $\tilde{\beta}=\{\tilde\beta_{j}\}_{j\in\mathbb{N}}$.
Then, for every $s\leq t\in[0,T]$, $i,j\in\mathbb{N}$, and continuous function $V:\bar{Y}\to\mathbb{R}$,
\begin{align}
&\tilde{\mathbb{E}}\Big[V(\tilde{Y}_h|_{[0,s]},\tilde{Z}_h|_{[0,s]},\tilde{\beta}_h|_{[0,s]})(\tilde{\beta}_{j,h}(t)-\tilde{\beta}_{j,h}(s))\Big]\nonumber\\
&=\mathbb{E}\Big[V(X_{l_h}|_{[0,s]},X_{\bar{l}_h}|_{[0,s]},B|_{[0,s]})(B^j_t-B^{j}_s)\Big]=0,\nonumber
\end{align}
and
\begin{align}
&\tilde{\mathbb{E}}\Big[V(\tilde{Y}_h|_{[0,s]},\tilde{Z}_h|_{[0,s]},\tilde{\beta}_h|_{[0,s]})\big(\tilde{\beta}_{i,h}(t)\tilde{\beta}_{j,h}(t)-\tilde{\beta}_{i,h}(s)\tilde{\beta}_{j,h}(s)-\delta_{ij}(t-s)\big)\Big]\nonumber\\
&=\mathbb{E}\Big[V(X_{l_h}|_{[0,s]},X_{\bar{l}_h}|_{[0,s]},B|_{[0,s]})(B^i_tB^j_t-B^i_sB^{j}_s-\delta_{ij}(t-s))\Big]=0.\nonumber
\end{align}
Taking the limit $h\to\infty$, we have
\begin{align*}
&\tilde{\mathbb{E}}\Big[V(\tilde{Y}|_{[0,s]},\tilde{Z}|_{[0,s]},\tilde{\beta}|_{[0,s]})(\tilde{\beta}_{j}(t)-\tilde{\beta}_{j}(s))\Big]=0,
\end{align*}
and
\begin{align*}
&\tilde{\mathbb{E}}\Big[V(\tilde{Y}|_{[0,s]},\tilde{Z}|_{[0,s]},\tilde{\beta}|_{[0,s]})(\tilde{\beta}_{i}(t)\tilde{\beta}_{j}(t)-\tilde{\beta}_{i}(s)\tilde{\beta}_{j}(s)-\delta_{ij}(t-s))\Big]=0.
\end{align*}
Using L\'{e}vy's characterization theorem (see \cite[Theorem 3.16]{karatzas1988brownian}), for each $j\in\mathbb{N}$, the path $\tilde{\beta}_j$ is a one-dimensional $\tilde{\mathcal{F}}_t$-Brownian motion, where $(\tilde{\mathcal{F}}_t)_{t\in[0,T]}$ is the augmented filtration of
\[
\mathcal{G}_t:=\sigma(\tilde{Y}|_{[0,s]},\tilde{Z}|_{[0,s]},\tilde{\beta}|_{[0,s]};s\leq t).
\]
Therefore, $\tilde{\beta}$ is a Brownian motion with respect to the complete filtration $\tilde{\mathcal{F}}_t$.

Similarly, note that
\begin{align}
&\tilde{\mathbb{E}}\Big[V(\tilde{Y}_h|_{[0,s]},\tilde{Z}_h|_{[0,s]},\tilde{\beta}_h|_{[0,s]})\big(\tilde{M}_{\phi_j,h}(t)-\tilde{M}_{\phi_j,h}(s)\big)\Big]\nonumber\\
&=\mathbb{E}\Big[V(X_{l_h}|_{[0,s]},X_{\bar{l}_h}|_{[0,s]},B|_{[0,s]})(M_{\phi_j,l_h}(t)-M_{\phi_j,l_h}(s))\Big]=0,\nonumber
\end{align}
\begin{align}
&\tilde{\mathbb{E}}\bigg[V(\tilde{Y}_h|_{[0,s]},\tilde{Z}_h|_{[0,s]},\tilde{\beta}_h|_{[0,s]})\nonumber\\
&\quad\quad\cdot\bigg(\tilde{M}_{\phi_j,h}(t)\tilde{\beta}_{i,h}(t)-\tilde{M}_{\phi_j,h}(s)\tilde{\beta}_{i,h}(s)-\int_s^t\int_x\phi_j(x,\tilde{u}_h,\tau)\nabla\cdot(\sigma(\tilde{u}_h)f_i)\mathrm{d}\tau \bigg)\bigg]\nonumber\\
&=\mathbb{E}\bigg[V(X_{l_h}|_{[0,s]},X_{\bar{l}_h}|_{[0,s]},B|_{[0,s]})\nonumber\\
&\quad\quad\cdot\bigg(M_{\phi_j,l_h}(t)B^i_t-M_{\phi_j,l_h}(s)B^{i}_s-\int_s^t\int_x\phi_j(x,\tilde{u}_h,\tau)\nabla\cdot(\sigma(\tilde{u}_h)f_i)\mathrm{d}\tau\bigg)\bigg]=0,\nonumber
\end{align}
and
\begin{align}
&\tilde{\mathbb{E}}\bigg[V(\tilde{Y}_h|_{[0,s]},\tilde{Z}_h|_{[0,s]},\tilde{\beta}_h|_{[0,s]})\nonumber\\
&\quad\quad\cdot\bigg(|\tilde{M}_{\phi_j,h}(t)|^2-|\tilde{M}_{\phi_j,h}(s)|^2-\int_s^t\sum_{i=1}^\infty\Big(\int_x\phi_j(x,\tilde{u}_h,\tau)\nabla\cdot(\sigma(\tilde{u}_h)f_i)\Big)^2\mathrm{d}\tau \bigg)\bigg]\nonumber\\
&=\mathbb{E}\bigg[V(X_{l_h}|_{[0,s]},X_{\bar{l}_h}|_{[0,s]},B|_{[0,s]})\nonumber\\
&\quad\quad\cdot\bigg(|M_{\phi_j,l_h}(t)|^2-|M_{\phi_j,l_h}(s)|^2-\int_s^t\sum_{i=1}^\infty\Big(\int_x\phi_j(x,u_h,\tau)\nabla\cdot(\sigma(u_h)f_i)\Big)^2\mathrm{d}\tau\bigg)\bigg]=0.\nonumber
\end{align}
Taking the limit $h\to\infty$, based on the continuity of the process in time and the uniform integrability, we have
for every $i,j\in\mathbb{N}$, three processes
\begin{align*}
&\tilde{M}_{\phi_j}(t),\\
&\tilde{M}_{\phi_j}(t)\tilde{\beta}_i(t)-\int_0^t\int_x\phi_j(x,\tilde{u},\tau)\nabla\cdot(\sigma(\tilde{u})f_i)\mathrm{d}\tau,\\
&|\tilde{M}_{\phi_j}(t)|^2-\int_0^t\sum_{i=1}^\infty\Big(\int_x\phi_j(x,\tilde{u},\tau)\nabla\cdot(\sigma(\tilde{u})f_i)\Big)^2\mathrm{d}\tau
\end{align*}
are all continuous $\tilde{\mathcal{F}}_t$-martingale. Since $\phi_j\in C_c^\infty(\mathbb{T}^d\times(0,\infty)\times[0,T))$, using \cite[Proposition A.1]{hofmanova2013degenerate}, we have
\begin{equation}\label{eq:form of M}
\tilde{M}_{\phi_j}(t)=\int_0^t\int_x\phi_j(x,\tilde{u},\tau)\nabla\cdot(\sigma(\tilde{u})\mathrm{d}\tilde{\xi}^F),
\end{equation}
where $\tilde{\xi}^F=\sum_{i=1}^{\infty}f_i(x)\tilde{\beta}_i(t)$.
This indicates that $(\tilde{u},\tilde{\nu},\tilde{q})$ satisfies \eqref{eq:kinetic equation test with time}, $\tilde{\mathbb{P}}$-almost surely.
Moreover, since $(\tilde{u}_h,\tilde{\nu}_h,\tilde{q}_h)$ satisfies \eqref{eq:energy preservation for penalty solution}, it follows from \eqref{eq:equal for distribution}, the almost sure convergence of $\tilde{u}_h$ (a non-relabelled subsequence), and the weak convergences of $\tilde{q}_h$ and $\tilde{\nu}_h$ that $(\tilde{u},\tilde{\nu},\tilde{q})$ satisfies \eqref{eq:preservation of energy in def} for any $\phi\in C_c^\infty([0,T))$, with the expectation taken with respect to
\(\tilde{\mathbb P}\), namely under \(\tilde{\mathbb E}\).
Then, by Lemma \ref{lem:converge to u init}, we obtain Definition \ref{def:def of stochastic kinetic solution} (iii).
Therefore, we conclude that $(\tilde{u},\tilde{\nu})$ is a stochastic kinetic solution under Definition \ref{def:def of stochastic kinetic solution} with $\tilde{q},\tilde{u}_{\mathrm{init}}$, Brownian motion $\tilde{\beta}$ and filtration $(\tilde{\mathcal{F}}_t)_{t\in[0,T]}$ in $(\tilde{\Omega},\tilde{\mathcal{F}},\tilde{\mathbb{P}})$.

Then, by Lemma \ref{lem:converge to u init}, we obtain Definition
\ref{def:def of stochastic kinetic solution} (iii). Therefore,
\((\tilde u,\tilde\nu)\) is a stochastic kinetic solution in the sense of
Definition \ref{def:def of stochastic kinetic solution}, with kinetic measure
\(\tilde q\), initial datum \(\tilde u_{\mathrm{init}}\), Brownian motion
\(\tilde\beta\), and filtration \((\tilde{\mathcal F}_t)_{t\in[0,T]}\) on the
stochastic basis
\[
(\tilde\Omega,\tilde{\mathcal F},(\tilde{\mathcal F}_t)_{t\in[0,T]},
\tilde{\mathbb P}).
\]

Following the above arguments, we have
\[
\tilde{Z}=(\bar{u},\nabla\llbracket\sqrt{\Phi^\prime}\rrbracket(\bar{u}),\bar{q},\bar{\nu},\{\bar{M}_{\phi_j}\}_{j\in\mathbb{N}}),
\]
and $(\bar{u},\bar{\nu})$ is a stochastic kinetic solution in the sense of Definition \ref{def:def of stochastic kinetic solution} with $\bar{q},\tilde{u}_{\mathrm{init}}$ and Brownian motion $\tilde{\beta}$ and filtration $(\tilde{\mathcal{F}}_t)_{t\in[0,T]}$ in $(\tilde{\Omega},\tilde{\mathcal{F}},\tilde{\mathbb{P}})$.
It follows from Corollary \ref{cor:unique with inital} and Theorem \ref{thm:uniqueness} that $(\tilde{u},\tilde{\nu})=(\bar{u},\bar{\nu})$ in $L^1(Q_T)\times\mathcal{M}(Q_T)$.
Applying \cite[Lemma 1.1]{gyongy2022existence}, there exists a non-relabelled subsequence $\{u_{\varepsilon_l},\nu_l\}$ in the original probability space $(\Omega,\mathcal{F},\mathbb{P})$ such that $\{u_{\varepsilon_l},\nu_l\}$ almost surely converges to $\{u,\nu\}$ in $L^1(Q_T)\times\mathcal{M}(Q_T)$ as $l\to\infty$.
Moreover, since almost surely for every $\phi\in C_{c}^{\infty}(\mathbb{T}^{d}\times(0,\infty)\times[0,T))$,
\begin{align}
& -\int_{0}^{T}\int_{\mathbb{R}}\int_{\mathbb{T}^d}\chi_{\varepsilon_l}(x,\xi,s)\partial_s\phi(x,\xi,s)\mathrm{d}x\mathrm{d}\xi\mathrm{d}s\label{eq:kinetic for penalty ql}\\
 & =\int_{\mathbb{R}}\int_{\mathbb{T}^{d}}\bar{\chi}(u_{\mathrm{init}},\xi)\phi(x,\xi,0)\mathrm{d}x\mathrm{d}\xi-\int_{0}^{T}\int_{\mathbb{T}^{d}}\Phi^{\prime}(u_{\varepsilon_l})\nabla_x u_{\varepsilon_l}\cdot(\nabla_x\phi)(x,u_{\varepsilon_l},s)\mathrm{d}x\mathrm{d}s\nonumber \\
 & \quad-\frac{1}{2}\int_{0}^{T}\int_{\mathbb{T}^{d}}\big(F_{1}(x)[\sigma^{\prime}(u_{\varepsilon_l})]^{2}\nabla_x u_{\varepsilon_l}+\sigma(u_{\varepsilon_l})\sigma^{\prime}(u_{\varepsilon_l})F_{2}(x)\big)\cdot(\nabla_x\phi)(x,u_{\varepsilon_l},s)\mathrm{d}x\mathrm{d}s\nonumber\\
 & \quad-\int_{0}^{T}\int_{\mathbb{R}}\int_{\mathbb{T}^{d}}\partial_{\xi}\phi(x,\xi,s)\mathrm{d}q_l(x,\xi,s)-\int_{0}^{T}\int_{\mathbb{T}^{d}}\phi(x,u_{\varepsilon_l},s)\nabla_x\cdot g(u_{\varepsilon_l})\mathrm{d}x\mathrm{d}s \nonumber\\
 & \quad+\frac{1}{2}\int_{0}^{T}\int_{\mathbb{T}^{d}}\Big(\sigma(u_{\varepsilon_l})\sigma^{\prime}(u_{\varepsilon_l})\nabla_x u_{\varepsilon_l}\cdot F_{2}(x)+\sigma^{2}(u_{\varepsilon_l})F_{3}(x)\Big)(\partial_{\xi}\phi)(x,u_{\varepsilon_l},s)\mathrm{d}x\mathrm{d}s\nonumber \\
 & \quad-\int_{0}^{T}\int_{\mathbb{T}^{d}}\phi(x,\psi,s)\mathrm{d}\nu_l(x,s)-\int_{0}^{T}\int_{\mathbb{T}^{d}}\phi(x,u_{\varepsilon_l},s)\nabla_x\cdot(\sigma(u_{\varepsilon_l})f_{k})\mathrm{d}x\mathrm{d}B_{s}^{k},\nonumber
\end{align}
the almost sure convergence of $q_l$ is a consequence of the almost sure convergence of other terms in \eqref{eq:kinetic for penalty ql} based on the convergence of $u_{\varepsilon_l}$ and $\nu_l$. Then,
one can verify that $(u,\nu)$ satisfies Definition \ref{def:def of stochastic kinetic solution} following the steps for $(\tilde{u},\tilde{\nu})$ in $(\tilde{\Omega},\tilde{\mathcal{F}},\tilde{\mathbb{P}})$, which completes the proof.
\end{proof}

\begin{rem}
Our result can be extended to a general equation
\begin{equation*}
\mathrm{d}u=\Delta\Phi(u)\mathrm{d}t-\nabla\cdot\big(\sigma(u)\circ \mathrm{d}{\xi}^{F}+g(u)\mathrm{d}t\big)+\phi(u)\mathrm{d}\xi^G+\lambda(u)\mathrm{d}t
\end{equation*}
where $\phi$, $\lambda$ and $\xi^G$ satisfy \cite[Assumption 6.1 and Assumption 6.3]{fehrman2024well}
.
\end{rem}

\begin{rem}
When $u$ is continuous or quasi-continuous, as in \cite{denis2014obstacle}, our definition of a solution to the obstacle problem automatically yields the Skorokhod condition. Indeed, one can directly apply It\^{o}'s formula to $\int_x u^2$ and compare the result with \eqref{eq:preservation of energy in def}. Together with the fact that $u\leq\psi$ almost surely on $\Omega\times Q_T$, this implies the Skorokhod condition.
\end{rem}

\begin{rem}
 Although the proof of existence is focused on the upper obstacle problem, a corresponding existence result is also valid for the lower obstacle problem under the additional assumptions $|\sigma(r)|\leq C(1+|r|)$ and $p=2$. The proof sketch is outlined as follows.
We first establish a counterpart monotonicity lemma via arguments similar to those for Lemma \ref{lem:(Comparison-theorem)}, which means that for any $0 < \varepsilon_1 \leq \varepsilon_2 $, we have almost surely $u_{\varepsilon_1} \geq u_{\varepsilon_2}$.
Passing to the limit as $\varepsilon \to 0$, and leveraging the a priori estimates derived from Corollary \ref{cor-1}, we then adopt the proof strategy of \cite[the proof of Theorem 2.4]{du2024well} to verify the existence of the almost sure limit function $u\in L^1(Q_T)$.
Finally, following the proof of Theorem \ref{thm:existence}, we obtain the existence of a solution $ (u, \nu)$ to the lower obstacle problem.
\end{rem}

\begin{appendix}
\section{Auxiliary lemmas\label{sec:Auxiliary}}
The first result concerns the estimates of $\sigma$, which is utilized to derive the a priori estimates \eqref{eq=00FF1Apriori p}.
\begin{lem}\label{lem:estimates of sigma u}
Let Assumption \ref{assu:assu for existence} hold for some $p\in[2,\infty)$, then we have either $\nabla\cdot F_2=0$ or there exists $c\in(0,\infty)$ such that for $\xi\in[0,\infty)$,
\[
\bigg|\int_0^{\xi}|r+1|^{p-2}\sigma^\prime(r)\sigma(r)\,\mathrm{d}r\bigg|\leq C\Bigg(1+\xi+\bigg|\int_0^{\xi}|r+1|^{(p-2)/2}\sqrt{\Phi^\prime(r)}\mathrm{d}r\bigg|^2\Bigg).
\]
\end{lem}
\begin{proof}
The case $p=2$ is a direct conclusion of Assumption \ref{assu:assu for existence} assertion 5, and we deal with the case $p>2$. Note that for $\alpha\in(-1,\infty)$, we have $(x+1)^{\alpha}\leq c x^{\alpha}$ holds for all $x\in[1,\infty)$ and some $c\in(0,\infty)$.
Integrating by part and using Assumption \ref{assu:assu for existence} assertion 5, we have
\begin{align}\label{eq:integrating by part sigma}
&\bigg|\int_0^{\xi}|r+1|^{p-2}\sigma^\prime(r)\sigma(r)\,\mathrm{d}r\bigg|\\
&=\Bigg|\frac{1}{2}|r+1|^{p-2}\sigma^2(r)\Big|_0^{\xi}-\frac{p-2}{2}\int_0^\xi|r+1|^{p-3}\sigma^2(r)\mathrm{d}r\Bigg|\nonumber\\
&\leq C\bigg((\xi^{p-2}+1)\sigma^2(\xi)+\int_0^{\xi\wedge1}|r+1|^{p-3}\sigma^2(r)\mathrm{d}r+\int_{\xi\wedge1}^\xi r^{p-3}\sigma^2(r)\mathrm{d}r\bigg)\nonumber\\
&\leq C\bigg(1+\xi+\bigg|\int_0^\xi |r+1|^{(p-2)/2}\sqrt{\Phi^\prime(r)}\mathrm{d}r\bigg|^2+\int_{\xi\wedge1}^\xi r^{p-3}\sigma^2(r)\mathrm{d}r\bigg).\nonumber
\end{align}
For the last term on the right hand side of \eqref{eq:integrating by part sigma}, using Assumption \ref{assu:assu for existence} assertion 5 and assertion 6, we have
\begin{align}
\int_{\xi\wedge1}^\xi r^{p-3}\sigma^{2}(r)\mathrm{d}r&\leq \int_{0}^\xi r^{p-3}\sigma^{2}(r)\mathrm{d}r\label{eq:estimates of r^p-3}\\
&=\frac{1}{p-2}r^{p-2}\sigma^2(r)\bigg|_0^\xi-\frac{2}{p-2}\int_0^\xi r^{p-2}\sigma^\prime(r)\sigma(r)\mathrm{d}r\nonumber\\
&\leq C\Big(\xi^{p-2}\sigma^{2}(\xi)+|\llbracket|\cdot|^{p-2}\sigma\sigma^\prime\rrbracket(\xi)|\Big)\nonumber\\
&\leq C\bigg(1+\xi+\bigg|\int_0^\xi r^{(p-2)/2}\sqrt{\Phi^\prime(r)}\mathrm{d}r\bigg|^2\bigg).\nonumber
\end{align}
Combining \eqref{eq:integrating by part sigma} and \eqref{eq:estimates of r^p-3}, we complete the proof.
\end{proof}

To establish the final existence result, namely Theorem \ref{thm:existence}, we need a lemma to verify that the limiting solution satisfies the initial condition (iii) in Definition \ref{def:def of stochastic kinetic solution}.
\begin{lem}\label{lem:converge to u init}
Let Assumption \ref{assu:assum for F} and Assumption \ref{assu:assu for unique} hold for $\xi^F, \psi, u_{\text{init}},\Phi,\sigma$.
Suppose that the pair $(u,\nu)$ fulfills all items in Definition \ref{def:def of stochastic kinetic solution} except the initial condition (iii), and satisfies the following preservation of energy for all $\phi\in C_c^\infty([0,T))$:
\begin{align}
&-\mathbb{E}\int_0^T\int_{\mathbb{T}^d}u^2(x,s)\partial_s\phi(s)\mathrm{d}x\mathrm{d}s+2\mathbb{E}\int_0^T\int_{\mathbb{T}^d}\phi(s)\psi(x,s)\mathrm{d}\nu(x,s)\label{eq:preservation of energy in def}\\
&\ +2\mathbb{E}\int_0^T\phi(s)\mathrm{d}q(\mathbb{T}^d,[0,\infty),s)+\frac{1}{2}\mathbb{E}\int_{0}^{T}\int_{\mathbb{T}^d}\phi(t)\sigma^2(u)(\nabla_{x}\cdot F_{2}-2F_3)\mathrm{d}x\mathrm{d}s\nonumber\\
&=\mathbb{E}\int_{\mathbb{T}^d}u^2_{\mathrm{init}}(x)\phi(0)\mathrm{d}x.\nonumber
\end{align}
Then, we have almost surely
\begin{align}\label{r-4}
\lim_{\tau\to0}\frac{1}{\tau}\int_0^{\tau}\int_{x}|u(x,t)-u_{\mathrm{init}}(x)|^2\mathrm{d}t=0,
\end{align}
which can be viewed as a strong trace result with respect to the initial data.
\end{lem}
\begin{proof}
Based on the definition of $\kappa^\varsigma_d$, we have
\begin{align}
&\frac{1}{\tau}\int_0^\tau\int_{x}|u(x,t)-u_{\text{init}}(x)|^2\mathrm{d}t\label{eq:split the term for uinit}\\
&\leq\frac{2}{\tau}\int_0^\tau\int_{x,y}\int_{\mathbb{T}^d}|u(x,t)-u_{\text{init}}(y)|^2\kappa_{d}^{\varsigma}(x-y)\mathrm{d}t+2\int_{x,y}|u_{\text{init}}(x)-u_{\text{init}}(y)|^2\kappa_{d}^{\varsigma}(x-y).\nonumber
\end{align}
Taking a decreasing non-negative function $\gamma\in C_c^1([0,T))$ such that
\[
\gamma(0)=2,\quad\gamma\leq2\mathbf{1}_{[0,3\tau]},\quad \partial_t\gamma(t)\in[-1/\tau,0],
\]
and
\begin{equation*}
\quad\partial_t\gamma(t)=
\begin{cases}
-1/\tau,&  \text{when  } t\in[0,\tau];\\
0,&\text{when  }t\in[3\tau,\infty).
\end{cases}
\end{equation*}
It follows from Assumption \ref{assu:assum for F} that there exists some $M>0$ such that $u\leq\psi\leq M$ on $Q_T$, it implies that $\chi(x,\xi,t)=0$ for any $(x,\xi,t)\in\mathbb{T}^d\times[M,\infty)\times[0,T]$. Note that
\begin{align}\label{r-5}
|u(x,t)-r|^2-r^2=2\int_{\mathbb{R}}(\xi-r)\chi(x,\xi,t)\mathrm{d}\xi\quad{\text{for }}(x,t,r)\in Q_T\times{\mathbb{R}},
\end{align}
by using the facts that $\chi\in[0,1]$ on $\mathbb{T}^d\times\mathbb{R}\times[0,T]$ and $u_{\text{init}}\geq0$, we have for $\beta\in(0,1)$, $N>M$ and $\varsigma\in(0,1)$,
\begin{align}
&\frac{1}{\tau}\int_0^\tau\int_{x,y}|u(x,t)-u_{\text{init}}(y)|^2\kappa_{d}^{\varsigma}(x-y)\mathrm{d}t\label{eq:split the term for u}\\
&\leq-\int_0^T\int_{x,y}\partial_t\gamma(t)|u(x,t)-u_{\text{init}}(y)|^2\kappa_{d}^{\varsigma}(x-y)\mathrm{d}t\nonumber\\
&=-2\int_0^{T}\int_{x,y,\xi}\partial_t\gamma(t)(\xi-u_{\text{init}}(y))\kappa_{d}^{\varsigma}(x-y)\chi(x,\xi,t)\mathrm{d}t \nonumber\\
&\quad-\int_0^{T}\int_{x,y}\partial_t\gamma(t)|u_{\text{init}}(y)|^2\kappa_{d}^{\varsigma}(x-y)\mathrm{d}t \nonumber\\
&\leq-2\int_0^{T}\int_{x,y,\xi}\partial_t\gamma(t)\varphi_{\beta}(\xi)(\xi-u_{\text{init}}(y))\kappa_{d}^{\varsigma}(x-y)\chi(x,\xi,t)\mathrm{d}t\nonumber\\
&\quad+2\int_0^{T}\int_{x,\xi}|\partial_t\gamma(t)|\xi\mathbf{1}_{\xi\in[0,\beta]}\mathrm{d}t -\int_0^{T}\int_{y}\partial_t\gamma(t)|u_{\text{init}}(y)|^2\mathrm{d}t\nonumber\\
&\leq -2\int_0^{T}\int_{x,y,\xi}\partial_t\gamma(t)\varphi_{\beta}(\xi)\zeta_{N}(\xi)(\xi-u_{\text{init}}(y))\kappa_{d}^{\varsigma}(x-y)\chi(x,\xi,t)\mathrm{d}t\nonumber\\
&\quad+ C\beta^2+\gamma(0)\int_{y}|u_{\text{init}}(y)|^2,\nonumber
\end{align}
where we used $\zeta_N(\xi)\chi(x,\xi,t)=\chi(x,\xi,t)$ for $(x,\xi,t)\in\mathbb{T}^d\times{\mathbb{R}}\times[0,T]$.

In the following, we will handle the first term on the righthand side of (\ref{eq:split the term for u}).
For fixed $y\in\mathbb{T}^d$, by taking the test function approximating $\phi(x,\xi,t)=\gamma(t)\varphi_\beta(\xi)\zeta_{N}(\xi)(\xi-u_{\text{init}}(y))\kappa_{d}^{\varsigma}(x-y)$ in \eqref{eq:kinetic equation test with time}, and integrating $y$ over $\mathbb{T}^d$, we have almost surely
\begin{align}
& -\int_{0}^{T}\int_{x,y,\xi}\partial_t\gamma(t)\varphi_\beta(\xi)(\xi-u_{\text{init}}(y))\kappa_{d}^{\varsigma}(x-y)\chi(x,\xi,t)\mathrm{d}t\label{test for u init}\\
 & =\int_{x,y,\xi}\gamma(0)\bar{\chi}(u_{\mathrm{init}},\xi)\varphi_\beta(\xi)(\xi-u_{\text{init}}(y))\kappa_{d}^{\varsigma}(x-y)\nonumber\\
&\quad -\int_{0}^{3\tau}\int_{x,y}\Phi^{\prime}(u)\nabla_x u\cdot\nabla_x\kappa_{d}^{\varsigma}(x-y)(u- u_{\text{init}}(y))\gamma(t)\varphi_\beta(u)\mathrm{d}t\nonumber \\
 & \quad-\frac{1}{2}\int_{0}^{3\tau}\int_{x,y}\big(F_{1}(x)[\sigma^{\prime}(u)]^{2}\nabla_x u+\sigma(u)\sigma^{\prime}(u)F_{2}(x)\big)\cdot\nabla_x\kappa_{d}^{\varsigma}(x-y)(u- u_{\text{init}}(y))\gamma(t)\varphi_\beta(u)\mathrm{d}t\nonumber\\
 & \quad-\int_{0}^{T}\int_{\mathbb{R}}\int_{\mathbb{T}^{d}}\int_{y}\gamma(t)\partial_\xi[\varphi_{\beta}(\xi)\zeta_N(\xi)(\xi-u_{\text{init}}(y))]\kappa_{d}^{\varsigma}(x-y)\mathrm{d}q(x,\xi,t)\nonumber \\
 &\quad-\int_{0}^{3\tau}\int_{x,y}\gamma(t)\varphi_\beta(u)\zeta_N(u)(u-u_{\mathrm{init}}(y))\kappa^\varsigma_d(x-y)\nabla\cdot g(u)\mathrm{d}t\nonumber\\
 & \quad+\frac{1}{2}\int_{0}^{3\tau}\int_{x,y}\Big(\sigma(u)\sigma^{\prime}(u)\nabla_x u\cdot F_{2}(x)+\sigma^{2}(u)F_{3}(x)\Big)\gamma(t)\kappa_{d}^{\varsigma}(x-y)\partial_\xi[\varphi_{\beta}(\xi)(\xi-u_{\text{init}}(y))]|_{\xi=u}\mathrm{d}t\nonumber \\
 & \quad-\int_{0}^{T}\int_{\mathbb{T}^{d}}\int_{y}\gamma(t)\varphi_\beta(\psi(x,t))(\psi(x,t)-u_{\text{init}}(y))\kappa_{d}^{\varsigma}(x-y)\mathrm{d}\nu(x,t)\nonumber\\
 &\quad-\int_{0}^{3\tau}\int_{x,y}\gamma(t)\varphi_\beta(u)(u-u_{\text{init}}(y))\kappa_{d}^{\varsigma}(x-y)\nabla_x\cdot(\sigma(u)f_{k})\mathrm{d}B_{t}^{k}
 =:\sum_{i=1}^8 I_i,\nonumber
\end{align}
where $u=u(x,t)$ and $\bar{\chi}(u_{\text{init}},\xi)=\bar{\chi}(u_{\text{init}}(x),\xi)$.

Note that the terms $I_2$, $I_3$, $I_5$, $I_6$ and $I_8$ will almost surely converge to 0 when $\tau\to0$. This follows from the support of $\varphi_\beta$, the uniform bound $u\leq M$, and Definition \ref{def:def of stochastic kinetic solution} (v).
For the term $I_1$, using the definition of the function $\bar{\chi}$ and the Lebesgue's dominated convergence theorem, we have
\begin{align}\label{eq:limit of I1}
\lim_{\varsigma\to0}\lim_{\beta\to0}I_1&=\frac{\gamma(0)}{2}\lim_{\varsigma\to0}\int_{\mathbb{T}^{d}}\int_{\mathbb{T}^d}\kappa_{d}^{\varsigma}(x-y)|\xi-u_{\text{init}}(y)|^2\Big|^{\xi=u_{\mathrm{init}}(x)}_{\xi=0}\mathrm{d}x\mathrm{d}y
\\
&=-\frac{\gamma(0)}{2}\int_{y}|u_{\mathrm{init}}(y)|^2.\nonumber
\end{align}
For the term $I_7$, since $u_{\mathrm{init}}\leq\psi(\cdot,0)$ on $\Omega\times\mathbb{T}^d$, we have
\begin{align}
I_7&\leq -\int_{0}^{T}\int_{\mathbb{T}^{d}}\int_{y}\gamma(t)\varphi_\beta(\psi(x,t))(\psi(x,t)-\psi(y,0))\kappa_{d}^{\varsigma}(x-y)\mathrm{d}\nu(x,t)\label{eq:estimate of I6}\\
&\leq C(\mathfrak{C}(\varsigma)+\mathfrak{C}(\tau))\nu(Q_T),\nonumber
\end{align}
where $\mathfrak{C}$ is the modulus of continuity of $\psi$.
On this stage, it remains to deal with the term $I_4$. The chain rule yields
\begin{align}
I_4&=-\int_{0}^{T}\int_{\mathbb{R}}\int_{\mathbb{T}^{d}}\int_{y}\gamma(t)\varphi_{\beta}^{\prime}(\xi)(\xi-u_{\text{init}}(y))\kappa_{d}^{\varsigma}(x-y)\mathrm{d}q(x,\xi,t)\label{eq:split I4}\\
&\quad-\int_{0}^{T}\int_{\mathbb{R}}\int_{\mathbb{T}^{d}}\int_{y}\gamma(t)\zeta_N^\prime(\xi)(\xi-u_{\text{init}}(y))\kappa_{d}^{\varsigma}(x-y)\mathrm{d}q(x,\xi,t)\nonumber\\
&\quad-\int_{0}^{T}\int_{\mathbb{R}}\int_{\mathbb{T}^{d}}\int_{y}\gamma(t)\varphi_{\beta}(\xi)\zeta_N(\xi)\kappa_{d}^{\varsigma}(x-y)\mathrm{d}q(x,\xi,t)=:\sum_{i=1}^3I_{4,i}.\nonumber
\end{align}
Since $q$ is a nonnegative measure, we have $I_{4,3}\leq0$.
For $I_{4,1}$, using Proposition \ref{prop:proposition for limit measure}, the non-negativity of measure $q$ and function $\varphi_\beta^\prime$, and Fatou's lemma, there almost surely exists non-relabelled subsequence $\beta\to0$ such that
\begin{align}
\lim_{\beta\to0}\lim_{\tau\to0}I_{4,1}&\leq\lim_{\beta\to0}\lim_{\tau\to0}\int_{0}^{T}\int_{\mathbb{R}}\int_{\mathbb{T}^{d}}\int_{y}\gamma(t)\varphi_{\beta}^{\prime}(\xi)u_{\text{init}}(y)\kappa_{d}^{\varsigma}(x-y)\mathrm{d}q(x,\xi,t)\label{eq:estimates of I 4,1}\\
&\leq2\lim_{\beta\to0}\int_{0}^{T/2}\int_{\mathbb{R}}\int_{\mathbb{T}^{d}}\int_{y}\varphi_{\beta}^{\prime}(\xi)u_{\text{init}}(y)\kappa_{d}^{\varsigma}(x-y)\mathrm{d}q(x,\xi,t)\nonumber\\
&\leq C\lim_{\beta\to0}\int_{0}^{T/2}\int_{\mathbb{R}}\int_{\mathbb{T}^{d}}\varphi_{\beta}^{\prime}(\xi)\mathrm{d}q(x,\xi,t)\nonumber\\
&\leq C\lim_{\beta\to0}\beta^{-1}q(\mathbb{T}^d\times[\beta/2,\beta]\times[0,T/2])=0.\nonumber
\end{align}
 To estimate $I_{4,2}$, we need a new property of the kinetic measure $q$, that is, almost surely,
 \begin{equation}\label{eq:limit for N q}
\underset{N\to\infty}{\mathrm{lim\,inf}}\big[Nq(\mathbb{T}^d\times[N,N+1]\times[0,T/2])\big]=0.
\end{equation}
 Its proof is postponed to the end of this proof. We now proceed to estimate $I_{4,2}$ under the assumption that (\ref{eq:limit for N q}) holds.
 Using \eqref{eq:limit for N q}, Definition \ref{def:def of stochastic kinetic solution} (vii), the non-negativity of measure $q$ and function $-\zeta_N^\prime$, and Fatou's lemma, there almost surely exists a non-relabeled subsequence $N\to\infty$ such that
\begin{align}
\lim_{N\to\infty}\lim_{\tau\to0}I_{4,2}&\leq-\lim_{N\to\infty}\lim_{\tau\to0}\int_{0}^{T}\int_{\mathbb{R}}\int_{\mathbb{T}^{d}}\int_{y}\gamma(t)\zeta_N^\prime(\xi)\xi\kappa_{d}^{\varsigma}(x-y)\mathrm{d}q(x,\xi,t)\label{eq:estimates of I 4,2}\\
&\leq-\lim_{N\to\infty}\int_{0}^{T/2}\int_{\mathbb{R}}\int_{\mathbb{T}^{d}}\zeta_N^\prime(\xi)\xi\,\mathrm{d}q(x,\xi,t)\nonumber\\
&\leq C\lim_{N\to\infty}(N+1)q(\mathbb{T}^d\times[N,N+1]\times[0,T/2])=0.\nonumber
\end{align}
Therefore, based on \eqref{eq:split I4}, \eqref{eq:estimates of I 4,1} and \eqref{eq:estimates of I 4,2}, we have
\[
\lim_{\beta\to0}\lim_{N\to\infty}\lim_{\tau\to0}I_4\leq0.
\]
Combining with \eqref{test for u init}-\eqref{eq:estimate of I6}, we have
\begin{align*}
&\lim_{\varsigma\to0}\lim_{\beta\to0}\lim_{N\to\infty}\lim_{\tau\to0}\bigg[-\int_{0}^{T}\int_{x,y,\xi}\partial_t\gamma(t)\varphi_\beta(\xi)(\xi-u_{\text{init}}(y))\kappa_{d}^{\varsigma}(x-y)\chi(x,\xi,t)\mathrm{d}t\Bigg]\\
&\leq-\frac{\gamma(0)}{2}\int_{y}|u_{\mathrm{init}}(y)|^2.
\end{align*}
Combining with \eqref{eq:split the term for uinit}, \eqref{eq:split the term for u}, and the continuity of translations in $L^2(\mathbb{T}^d)$, the result (\ref{r-4}) follows.

In what follows, we present the proof of (\ref{eq:limit for N q}). This proof follows a similar argument to that of Proposition \ref{prop:proposition for limit measure}, and relies on \eqref{eq:preservation of energy in def}.
For any $\epsilon>0$, $N\in\mathbb{N}$ and $\beta\in(0,1)$, by taking smooth test functions in \eqref{eq:kinetic equation test with time} which converge to $\tilde{\alpha}_{\epsilon,T}(s)\zeta_{N}(\xi)\varphi_{\beta}(\xi)\xi$ with $\tilde{\alpha}_{\epsilon,T}$ defined by (\ref{r-2}), taking expectation, and using the Lebesgue's dominated convergence theorem, we have
\begin{align}
& -\mathbb{E}\int_{0}^{T}\int_{x,\xi}\chi(x,\xi,s)\partial_s\tilde{\alpha}_{\epsilon,T}(s)\zeta_N(\xi)\varphi_\beta(\xi)\xi\,\mathrm{d}s\label{eq:u^2 for equation}\\
 & =\mathbb{E}\int_{x,\xi}\bar{\chi}(u_{\mathrm{init}},\xi)\tilde{\alpha}_{\epsilon,T}(0)\zeta_N(\xi)\varphi_\beta(\xi)\xi\nonumber\\
 & \quad-\mathbb{E}\int_{0}^{T}\int_{\mathbb{R}}\int_{\mathbb{T}^{d}}\tilde{\alpha}_{\epsilon,T}(s)\xi\partial_{\xi}(\zeta_N(\xi)\varphi_\beta(\xi))\mathrm{d}q(x,\xi,s)\nonumber\\
 &\quad-\mathbb{E}\int_{0}^{T}\int_{\mathbb{R}}\int_{\mathbb{T}^{d}}\tilde{\alpha}_{\epsilon,T}(s)\zeta_N(\xi)\varphi_\beta(\xi)\mathrm{d}q(x,\xi,s)\nonumber \\
&\quad -\mathbb{E}\int_{0}^{T}\int_{x}\tilde{\alpha}_{\epsilon,T}(s)\zeta_{N}(u)\varphi_{\beta}(u)u\nabla\cdot g(u)\mathrm{d}s\nonumber\\
 & \quad+\frac{1}{2}\mathbb{E}\int_{0}^{T}\int_{x}\tilde{\alpha}_{\epsilon,T}(s)\Big(\sigma(u)\sigma^{\prime}(u)\nabla u\cdot F_{2}(x)+\sigma^{2}(u)F_{3}(x)\Big)\zeta_N(u)\varphi_\beta(u)\mathrm{d}s\nonumber \\
  & \quad+\frac{1}{2}\mathbb{E}\int_{0}^{T}\int_{x}\tilde{\alpha}_{\epsilon,T}(s)\Big(\sigma(u)\sigma^{\prime}(u)\nabla u\cdot F_{2}(x)+\sigma^{2}(u)F_{3}(x)\Big)u\partial_\xi(\zeta_N\varphi_\beta)(u)\mathrm{d}s\nonumber \\
 & \quad-\mathbb{E}\int_{0}^{T}\int_{\mathbb{T}^{d}}\tilde{\alpha}_{\epsilon,T}(s)\psi\zeta_N(\psi)\varphi_{\beta}(\psi)\mathrm{d}\nu(x,s).\nonumber
\end{align}
We will consider the limit of each term in \eqref{eq:u^2 for equation} when $N\to\infty$ and $\beta\to0$. For simplicity, denote the equation (\ref{eq:u^2 for equation}) by $J=\sum^7_{i=1}J_i$.

Using \cite[Remark 5.26]{fehrman2024well} and applying the Lebesgue's dominated convergence theorem, with the integrability of $\sigma^2(u)$ and the definitions of $\chi$ and $\bar{\chi}$, we have
\begin{align*}
\lim_{\beta\to0}\lim_{N\to\infty}J
&=-\frac{1}{2}\mathbb{E}\int_{0}^{T}\int_{x}|u(x,s)|^2\partial_s\tilde{\alpha}_{\epsilon,T}(s)\mathrm{d}s,\\
\lim_{\beta\to0}\lim_{N\to\infty}J_1&=\mathbb{E}\int_{x}|u_{\mathrm{init}}(x)|^2\tilde{\alpha}_{\epsilon,T}(0),
\\
\lim_{\beta\to0}\lim_{N\to\infty}J_3
&=-\mathbb{E}\int_{0}^{T}\tilde{\alpha}_{\epsilon,T}(s)\mathrm{d}q(\mathbb{T}^d,[0,\infty),s),\\
\lim_{\beta\to0}\lim_{N\to\infty}J_7&=-\mathbb{E}\int_{0}^{T}\int_{\mathbb{T}^{d}}\tilde{\alpha}_{\epsilon,T}(s)\psi(x,s)\mathrm{d}\nu(x,s),
\end{align*}
and
\begin{align*}
&\lim_{\beta\to0}\lim_{N\to\infty}J_5\\
&=\lim_{\beta\to0}\lim_{N\to\infty}\mathbb{E}\int_{0}^{T}\int_{x}\tilde{\alpha}_{\epsilon,T}(s)\Big(-\llbracket\zeta_N\varphi_N\sigma\sigma^\prime\rrbracket(u) \nabla\cdot F_{2}(x)+\zeta_N(u)\varphi_\beta(u)\sigma^{2}(u)F_{3}(x)\Big)\mathrm{d}s\\
&=\mathbb{E}\int_{0}^{T}\int_{x}\tilde{\alpha}_{\epsilon,T}(s)\sigma^2(u)\Big(-\frac{1}{2}\nabla\cdot F_{2}(x)+F_{3}(x)\Big)\mathrm{d}s.
\end{align*}
Moreover, due to $u\leq\psi\leq M$ on $\Omega\times Q_T$, using Assumption \ref{assu:assu for unique} for the limit of $\sigma^2(r)$ when $r\to0^+$, we have for $N>M$,
\begin{align*}\notag
&\lim_{\beta\to0}\lim_{N\to\infty}J_5\\
&=\lim_{\beta\to0}\mathbb{E}\int_{0}^{T}\int_{x}\tilde{\alpha}_{\epsilon,T}(s)\Big(-\nabla\cdot F_{2}(x)\cdot\int_0^u r\varphi_\beta^\prime(r) \sigma(r)\sigma^\prime(r)\mathrm{d}r+u\varphi_\beta^\prime(u)\sigma^{2}(u)F_{3}(x)\Big)\mathrm{d}s\\ \notag
&=0.
\end{align*}
Note that
\[
J_4=-\mathbb{E}\int_0^T\int_x\tilde{\alpha}_{\epsilon,T}(s)\nabla\cdot\int_0^u \zeta_N(r)\varphi_\beta(r)g(r)\mathrm{d}r\mathrm{d}s=0.
\]
Combining all the previous estimates, we arrive at
\begin{align*}
 \lim_{\beta\to0}\lim_{N\to\infty}J_2
  &=-\frac{1}{2}\mathbb{E}\int_{0}^{T}\int_{x}|u(x,s)|^2\partial_s\tilde{\alpha}_{\epsilon,T}(s)\mathrm{d}s-\mathbb{E}\int_{x}|u_{\mathrm{init}}(x)|^2\tilde{\alpha}_{\epsilon,T}(0)
\\
&+\mathbb{E}\int_{0}^{T}\tilde{\alpha}_{\epsilon,T}(s)\mathrm{d}q(\mathbb{T}^d,[0,\infty),s)+\mathbb{E}\int_{0}^{T}\int_{\mathbb{T}^{d}}\tilde{\alpha}_{\epsilon,T}(s)\psi(x,s)\mathrm{d}\nu(x,s)\\
  &-\mathbb{E}\int_{0}^{T}\int_{x}\tilde{\alpha}_{\epsilon,T}(s)\sigma^2(u)\Big(-\frac{1}{2}\nabla\cdot F_{2}(x)+F_{3}(x)\Big)\mathrm{d}s.
\end{align*}
With the aid of \eqref{eq:preservation of energy in def} with $\phi=\tilde{\alpha}_{\epsilon,T}$, we reach
\begin{align}\label{eq:limit for partialxi q}
\lim_{\beta\to0}\lim_{N\to\infty}J_2=\lim_{\beta\to0}\lim_{N\to\infty}\mathbb{E}\int_{0}^{T}\int_{\mathbb{R}}\int_{\mathbb{T}^{d}}\tilde{\alpha}_{\epsilon,T}(s)\xi\partial_{\xi}(\zeta_N(\xi)\varphi_\beta(\xi))\mathrm{d}q(x,\xi,s)=0.
\end{align}
Moreover, it follows from Proposition \ref{prop:proposition for limit measure} that
\begin{align}
0&\leq\lim_{\beta\to0}\lim_{N\to\infty}\mathbb{E}\int_{0}^{T}\int_{\mathbb{R}}\int_{\mathbb{T}^{d}}\tilde{\alpha}_{\epsilon,T}(s)\xi\varphi_\beta^\prime(\xi)\zeta_N(\xi)\mathrm{d}q(x,\xi,s)\label{eq:limit for paritalxi beta q}\\
&=\lim_{\beta\to0}\mathbb{E}\int_{0}^{T}\int_{\mathbb{R}}\int_{\mathbb{T}^{d}}\tilde{\alpha}_{\epsilon,T}(s)\xi\varphi_\beta^\prime(\xi)\mathrm{d}q(x,\xi,s)\nonumber\\
&\leq C\lim_{\beta\to0}\mathbb{E}\big(\beta^{-1}q(\mathbb{T}^{d}\times[\beta/2,\beta]\times[0,T))\big)=0.\nonumber
\end{align}
Combining \eqref{eq:limit for partialxi q} and \eqref{eq:limit for paritalxi beta q}, by using the fact that $\tilde{\alpha}_{\epsilon,T}=1$ on $[0,T/2]$ for $\epsilon$ small enough, we have
\begin{align*}
&\lim_{N\to\infty}\mathbb{E}\int_{0}^{T/2}\int_{\mathbb{R}}\int_{\mathbb{T}^{d}}\xi\zeta
^\prime_N(\xi)\mathrm{d}q(x,\xi,s)\\
&=\lim_{\beta\to0}\lim_{N\to\infty}\mathbb{E}\int_{0}^{T/2}\int_{\mathbb{R}}\int_{\mathbb{T}^{d}}\tilde{\alpha}_{\epsilon,T}(s)\xi\partial_{\xi}(\zeta_N(\xi)\varphi_\beta(\xi))\mathrm{d}q(x,\xi,s)=0.
\end{align*}
As a result of the definition of $\zeta_N$, we reach
\begin{align}
&0=\lim_{N\to\infty}\mathbb{E}\int_{0}^{T/2}\int_{\mathbb{R}}\int_{\mathbb{T}^{d}}\xi\zeta
^\prime_N(\xi)\mathrm{d}q(x,\xi,s)\nonumber\\
&\leq-\lim_{N\to\infty}N\mathbb{E}q(\mathbb{T}^d\times[N,N+1]\times[0,T/2])\leq0.\nonumber
\end{align}
Hence, (\ref{eq:limit for N q}) follows by using Fatou's lemma.\end{proof}
\end{appendix}

\bibliographystyle{plain}
\bibliography{bi}

\end{document}